\documentclass{amsart}
\usepackage{graphicx} 
\usepackage{hyperref}
\usepackage{amsmath, amsthm, amssymb, amscd}
\usepackage{mathtools}
\usepackage{enumerate}
\usepackage{hyperref}
\usepackage{verbatim}
\usepackage{cite}
\usepackage{color}
\usepackage{esint}
\usepackage{tikz-cd}
\usetikzlibrary{shapes.geometric}
\usepackage[shortlabels]{enumitem}

\title[On Constructions of Fractal Spaces and the CLP property]{On Constructions of Fractal Spaces Using Replacement and the Combinatorial Loewner Property }
\author{Riku Anttila}
\email{riku.t.anttila@jyu.fi}
\address{Department of Mathematics and Statistics \\
P.O. Box 35 \\
FI-40014 University of Jyväskylä}

\author{Sylvester Eriksson-Bique}
\email{sylvester.d.eriksson-bique@jyu.fi}
\address{Department of Mathematics and Statistics \\
P.O. Box 35 \\
FI-40014 University of Jyväskylä}

\thanks{The first author was partially supported by the Eemil Aaltonen foundation. The second author was supported by the Research Council of Finland grant 354241. We especially thank Mario Bonk, Jeff Cheeger, Guy C. David, Bruce Kleiner, Mathav Murugan and Ryosuke Shimizu for discussions on Kleiner's conjecture. We also thank Shimizu for comments on a draft of this paper. }
\date{\today}

\newcommand{\diam}{{\rm diam}}
\newcommand{\len}{{\rm len}}
\newcommand{\id}{{\rm id}}
\newcommand{\intr}{{\rm int}}
\newcommand{\dims}{{\rm dim}}
\newcommand{\Mod}{{\rm Mod}}
\newcommand{\degr}{{\rm deg}}
\newcommand{\divr}{{\rm div}}
\newcommand{\dist}{{\rm dist}}
\newcommand{\supp}{{\rm supp}}

\newcommand{\AR}{{\rm AR}}
\newcommand{\cCap}{{\rm Cap}}
\newcommand{\sgn}{{\rm sgn}}

\newcommand{\cN}{{\rm \mathcal{N}}}

\DeclarePairedDelimiter{\abs}{\lvert}{\rvert}
\DeclarePairedDelimiter\ceil{\lceil}{\rceil}

\newtheorem{theorem}[equation]{Theorem}
\newtheorem{lemma}[equation]{Lemma}
\newtheorem{proposition}[equation]{Proposition}
\newtheorem{corollary}[equation]{Corollary}
\newtheorem{assumption}[equation]{Assumption}

 \numberwithin{equation}{section}

\theoremstyle{definition}
\newtheorem{definition}[equation]{Definition}

\theoremstyle{remark}
\newtheorem{remark}[equation]{Remark}
\newtheorem{conjecture}[equation]{Conjecture}
\newtheorem{example}[equation]{Example}

\newcommand{\N}{\mathbb{N}}
\newcommand{\R}{\mathbb{R}}
\newcommand{\cG}{\mathcal{G}}
\newcommand{\cE}{\mathcal{E}}
\newcommand{\cF}{\mathcal{F}}
\newcommand{\cH}{\mathcal{H}}
\newcommand{\cM}{\mathcal{M}}

\makeatletter
\let\c@equation\c@figure
\makeatother

\begin{document}

\begin{abstract}
    The combinatorial Loewner property was introduced by Bourdon and Kleiner as a quasisymmetrically invariant substitute for the Loewner property for general fractals and boundaries of hyperbolic groups. While the Loewner property is somewhat restrictive, the combinatorial Loewner property is very generic -- Bourdon and Kleiner showed that many familiar fractals and group boundaries satisfy it. If $X$ is quasisymmetric to a Loewner space, it has the combinatorial Loewner property. Kleiner conjectured in 2006 that the converse to this holds for self-similar fractals -- the hope being that this would lead to the existence of many exotic Loewner spaces. We disprove this conjecture and give the first examples of spaces which are self-similar, combinatorially Loewner and which are not quasisymmetric to Loewner spaces. 
    
    In the process we introduce a self-similar replacement rule, called iterated graph systems (IGS), which is inspired by the work of Laakso. This produces a new rich class of fractal spaces, where closed form computations of potentials and their conformal dimensions are possible. These spaces exhibit a rich class of behaviors from analysis on fractals in regards to diffusions, Sobolev spaces, energy measures and conformal dimensions. These behaviors expand on the known examples of Cantor sets, gaskets, Vicsek sets, and the often too difficult carpet-like spaces. Especially the counterexamples to Kleiner's conjecture that arise from this construction are interesting, since they open up the possibility to study the new realm of combinatorially Loewner spaces that are not quasisymmetric to Loewner spaces.
\end{abstract}

\maketitle

\section{Introduction}

\subsection{Overview}
In this paper we answer to the negative a conjecture posed by Kleiner on whether approximately self-similar combinatorially Loewner spaces are quasisymmetric to Loewner spaces, \cite[Conjecture 7.5]{KleinerICM}. A concise overview of the problem and its significance can be gleaned from \cite{KleinerICM} and \cite{claisthesis,Bourdon2002}.  In the process, we give a new and rich class of spaces, derived from iterated graph systems, which exhibits curious behaviors and offers a tractable, and unexplored, setting to study analysis on fractals, see \cite{strichartz,kigami,barlow} for surveys. The precise terminology is presented in detail in Section \ref{sec:concepts}, which a non-expert reader may first wish to consult, before reading the statements of the main results.
\begin{figure}[!ht]
\begin{tikzpicture}[scale=.75]
\draw (-6,0)--(-4,0);
\draw[<-,thick] (-3,0)--(-2,0);

\draw (-1,-1)--(0,0)--(1,1)--(2,0)--(3,1);
\draw (0,0)--(1,-1)--(2,0);
\draw (1,-1)--(1,1);
\draw (-1,1)--(0,0);
\draw (2,0)--(3,-1);
\draw[<-,thick] (4,0)--(5,0);

\draw (6,1.25)--(6.25,0.75)--(6.65,0.75)--(6.75,0.25)--(7,0.35);
\draw (6,0.65)--(6.25,0.75)--(6.35,0.25)--(6.75,0.25)--(7,-0.35);
\draw (6.65,0.75)--(6.35,0.25);
\draw (6,-1.25)--(6.25,-0.75)--(6.65,-0.75)--(6.75,-0.25)--(7,-0.35);
\draw (6,-0.65)--(6.25,-0.75)--(6.35,-0.25)--(6.75,-0.25)--(7,0.35);
\draw (6.65,-0.75)--(6.35,-0.25);

\draw (8,1.25)--(7.75,0.75)--(7.35,0.75)--(7.25,0.25)--(7,0.35);
\draw (8,0.65)--(7.75,0.75)--(7.65,0.25)--(7.25,0.25)--(7,-0.35);
\draw (7.35,0.75)--(7.65,0.25);
\draw (8,-1.25)--(7.75,-0.75)--(7.35,-0.75)--(7.25,-0.25)--(7,-0.35);
\draw (8,-0.65)--(7.75,-0.75)--(7.65,-0.25)--(7.25,-0.25)--(7,0.35);
\draw (7.35,-0.75)--(7.65,-0.25);

\draw (8,1.25)--(8.15,0.5)--(8.2,0)--(7.85,-0.5)--(8,-1.25);
\draw (8,0.65)--(8.15,0.5)--(7.8,0)--(7.85,-0.5)--(8,-0.65);
\draw (8.2,0)--(7.8,0);

\draw (8,1.25)--(8.25,0.75)--(8.65,0.75)--(8.75,0.25)--(9,0.35);
\draw (8,0.65)--(8.25,0.75)--(8.35,0.25)--(8.75,0.25)--(9,-0.35);
\draw (8.65,0.75)--(8.35,0.25);
\draw (8,-1.25)--(8.25,-0.75)--(8.65,-0.75)--(8.75,-0.25)--(9,-0.35);
\draw (8,-0.65)--(8.25,-0.75)--(8.35,-0.25)--(8.75,-0.25)--(9,0.35);
\draw (8.65,-0.75)--(8.35,-0.25);

\draw (10,1.25)--(9.75,0.75)--(9.35,0.75)--(9.25,0.25)--(9,0.35);
\draw (10,0.65)--(9.75,0.75)--(9.65,0.25)--(9.25,0.25)--(9,-0.35);
\draw (9.35,0.75)--(9.65,0.25);
\draw (10,-1.25)--(9.75,-0.75)--(9.35,-0.75)--(9.25,-0.25)--(9,-0.35);
\draw (10,-0.65)--(9.75,-0.75)--(9.65,-0.25)--(9.25,-0.25)--(9,0.35);
\draw (9.35,-0.75)--(9.65,-0.25);

\foreach \p in {(6,1.25),(6.25,0.75),(6.65,0.75),(6.75,0.25),(7,0.35),(6,0.65),(6.35,0.25),(7,-0.35),(6,-1.25),(6.25,-0.75),(6.65,-0.75),(6.75,-0.25),(6,-0.65),(6.35,-0.25),(8,1.25),(7.75,0.75),(7.35,0.75),
(7.75,-0.75),(7.35,-0.75),(7.25,-0.25),(7.65,-0.25),(7.25,0.25),
(8,0.65),(7.65,0.25),(8,-1.25),(8.25,-0.75),(8.65,-0.75),(8.75,-0.25),
(8.25,0.75),(8.65,0.75),(8.75,0.25),(9,0.35),
(8,-0.65),(10,1.25),(9.75,0.75),(9.35,0.75),(9.25,0.25),
(10,-1.25),(9.75,-0.75),(9.35,-0.75),(9.25,-0.25),(9,-0.35),(9.65,-0.25),
(10,-0.65),(10,0.65),(9.65,0.25),(8.35,0.25),(8.35,-0.25),
(8.15,0.5),(8.2,0),(7.85,-0.5),(7.8,0)}
\node at \p [circle,fill,inner sep=.8pt]{};

\foreach \q in {(-6,0),(-4,0),(-1,1),(-1,-1),(0,0),(1,1),(1,-1),(2,0),(3,1),(3,-1)}
\node at \q [circle,fill,inner sep=1pt]{};

\node at (-1,1.4) {1};
\node at (-1,-1.5) {2};
\node at (0,0.4) {3};
\node at (1,1.4) {4};
\node at (1,-1.5) {5};
\node at (2,0.4) {6};
\node at (3,1.4) {7};
\node at (3,-1.5) {8};
\end{tikzpicture}
\caption{Figure of an IGS that produces a counterexample to Kleiner's question. The example, which is described in more detail in Example \ref{ex:counterexample}, is obtained by recursively replacing every edge in the graph by the graph depicted in the middle, and the copies corresponding to adjacent edges are glued along their left or right boundaries. The edges in the new graph have length of the original edges divided by four. The labels in the figure correspond to those given in Example \ref{ex:counterexample}. The figure shows two steps of the iteration.}
\label{fig:replacementrule}
\end{figure}

\subsection{Main theorem}
The theory of quasiconformal mappings was revealed in settings of increasing generalities: first developed on surfaces \cite{Ahlfors}, then higher dimensional Euclidean spaces \cite{Rickmann, Vaisala}, followed by certain Carnot groups \cite{Pansu} and ultimately Loewner spaces \cite{HK}; see \cite{BonkICM,KleinerICM,HeNonSmooth,Bourdon2002,Kapovichboundaries} for nice surveys. This generalization was partly motivated by quasi-isometric rigidity problems in geometric group theory (e.g. Mostow's rigidity theorem, \cite{Mostow,Bourdon,Kapovichboundaries}). Indeed, it was observed that whenever the boundary of a hyperbolic group possessed a Loewner structure in its conformal gauge, this implied rigidity properties for the group \cite{KleinerICM, HeNonSmooth,Bourdon2002}. The conformal gauge consists of all metrics quasisymmetric to a given (visual) reference metric. While successful, this approach suffered a somewhat serious limitation: there are only very few hyperbolic groups whose boundaries are known to possess a Loewner structure. An essentially exhaustive list is given by the following families of examples: boundaries of groups acting by isometries co-compactly on a rank-one symmetric space \cite{Pansu, MT} or a Bourdon-Pajot building \cite{BourdonPI,bourdon2000rigidity}. Prompted by this, it was asked by several authors  \cite{KleinerICM,HeNonSmooth,claisthesis, BourdonPajot,Kapovichboundaries} if more hyperbolic group boundaries would admit a Loewner structure. 

Towards an answer, Bourdon and Kleiner observed in \cite{BourK} that a seemingly similar combinatorial Loewner property (CLP) was much more generic than the Loewner property, and was much easier to establish in given examples. Indeed, they and Clais \cite{clais} showed that far more group boundaries and fractals satisfy it: the Sierpi\'nki carpet, the Menger curve, boundaries of certain higher dimensional hyperbolic buildings, and boundaries of certain Coxeter groups. None of these examples were known to possess a Loewner metric in their conformal gauge. Given the amplitude of such examples, Kleiner conjectured in \cite{KleinerICM} the following.

\begin{conjecture}\label{conj:kleiner} All approximately self-similar (or: group boundary) combinatorially Loewner metric spaces are quasisymmetric to a Loewner space.
\end{conjecture}
Kleiner's conjecture, and related problems, have attracted considerable attention, see e.g. \cite{claisthesis,murugan2023first,MN, CEB,Kwapisz, CM}, and its statement is quite natural. Due to a lack of counterexamples, and hopes of finding new exotic Loewner spaces, there have been several attempts at proving the conjecture true, and an expectation by many that the conjecture would have a positive answer. However, a negative answer also would be quite intriguing: it opens up the possibility to study the combinatorial Loewner porperty as something strictly weaker and more general than the Loewner property.

As we will discuss below, when we introduce the terminology in detail, this problem is also equivalent to the attainment problem for the Ahlfors regular conformal dimension. This attainment problem has also been studied extensively, and is known to be quite difficult. In a different disguise, the same problem has also been studied for conductively homogeneous spaces, see \cite{murugan2023first, kigami}. In that setting, it is closely related to the attainment problem of other critical dimensions, such as the conformal walk dimension -- whose attainment problem is much better understood thanks to \cite{MN}.

Our main contribution is to answer the main part of Kleiner's conjecture to the negative by giving the first examples of approximately self-similar combinatorially Loewner metric spaces which are not quasisymmetric to Loewner spaces. 

\begin{theorem}\label{thm:mainthm} There exists a compact metric space $X$, which is approximately self-similar and combinatorially $Q$-Loewner for some $Q \in (1,\infty)$, but which is not quasisymmetric to a Loewner space. This space also does not attain its Ahlfors regular conformal dimension.
\end{theorem}

The counterexamples are obtained by a new construction of fractal spaces that we call \emph{iterated graph systems (IGS)}. An IGS consists of a given graph $G_1=(V_1,E_1)$ together with an iteration procedure that produces a sequence of graphs $G_m=(V_m,E_m)$. The graphs $G_{m+1}$ for $m\in \N$, are obtained recursively by replacing each edge of $G_m$ by a copy of $G_1$ and identifying neighboring copies by a given rule. The resulting graphs $G_n$ are equipped with a re-scaled path metric $d_n$. Under some easily checkable conditions the sequence of metric spaces $(G_n,d_n)$ Gromov-Hausdorff converge to a limit space $X$. (At this juncture, an eager reader may check Definition \ref{def:IGS}.) We note that a slightly similar idea has appeared in \cite[Definition 2.1]{Leeslash}.

With specific graphs, see e.g. Figure \ref{fig:laakso} and Example \ref{ex:nonsym}, such IGSs produce variants of fractals called Laakso-spaces \cite{Laakso}, although the construction is distinct from the quotient-construction of Laakso and the identifications used are different. These can also be realized by an inverse limit construction, see 
\cite{CK_PI}, and as we will see in Example \ref{ex:loewnerreplacement} they are Loewner spaces.  Our key insight is that such constructions can be modified by adding edges, which we call \emph{removable edges}, to produce spaces that do not arise as inverse limits. Such spaces are still combinatorially Loewner. It is one of these modified examples that are depicted in Figure \ref{fig:replacementrule} which yields a counterexample to Kleiner's conjecture. Indeed, the example in this figure shows the following explicit version of Theorem \ref{thm:mainthm}.

\begin{proposition}\label{prop:counterxample} The space $X$ from Example \ref{ex:counterexample} is $\log(9)/\log(4)$-Ahlfors regular, $3/2$-combina\-torially Loewner and approximately self-similar. Its conformal dimension is equal to $3/2$, and it is not attained. In particular, the space $X$ is not quasisymmetric to a Loewner space.
\end{proposition}

\begin{figure}
\begin{tikzpicture}[scale=0.75]
\draw (-6,0)--(-4,0);
\draw[<-,thick] (-3,0)--(-2,0);

\draw (-1,-1)--(0,0)--(1,1)--(2,0)--(3,1);
\draw (0,0)--(1,-1)--(2,0);
\draw (-1,1)--(0,0);
\draw (2,0)--(3,-1);
\draw[<-,thick] (4,0)--(5,0);

\draw (6,1.25)--(6.25,0.75)--(6.65,0.75)--(6.75,0.25)--(7,0.35);
\draw (6,0.65)--(6.25,0.75)--(6.35,0.25)--(6.75,0.25)--(7,-0.35);
\draw (6,-1.25)--(6.25,-0.75)--(6.65,-0.75)--(6.75,-0.25)--(7,-0.35);
\draw (6,-0.65)--(6.25,-0.75)--(6.35,-0.25)--(6.75,-0.25)--(7,0.35);

\draw (8,1.25)--(7.75,0.75)--(7.35,0.75)--(7.25,0.25)--(7,0.35);
\draw (8,0.65)--(7.75,0.75)--(7.65,0.25)--(7.25,0.25)--(7,-0.35);
\draw (8,-1.25)--(7.75,-0.75)--(7.35,-0.75)--(7.25,-0.25)--(7,-0.35);
\draw (8,-0.65)--(7.75,-0.75)--(7.65,-0.25)--(7.25,-0.25)--(7,0.35);

\draw (8,1.25)--(8.25,0.75)--(8.65,0.75)--(8.75,0.25)--(9,0.35);
\draw (8,0.65)--(8.25,0.75)--(8.35,0.25)--(8.75,0.25)--(9,-0.35);
\draw (8,-1.25)--(8.25,-0.75)--(8.65,-0.75)--(8.75,-0.25)--(9,-0.35);
\draw (8,-0.65)--(8.25,-0.75)--(8.35,-0.25)--(8.75,-0.25)--(9,0.35);

\draw (10,1.25)--(9.75,0.75)--(9.35,0.75)--(9.25,0.25)--(9,0.35);
\draw (10,0.65)--(9.75,0.75)--(9.65,0.25)--(9.25,0.25)--(9,-0.35);
\draw (10,-1.25)--(9.75,-0.75)--(9.35,-0.75)--(9.25,-0.25)--(9,-0.35);
\draw (10,-0.65)--(9.75,-0.75)--(9.65,-0.25)--(9.25,-0.25)--(9,0.35);

\foreach \p in {(6,1.25),(6.25,0.75),(6.65,0.75),(6.75,0.25),(7,0.35),(6,0.65),(6.35,0.25),(7,-0.35),(6,-1.25),(6.25,-0.75),(6.65,-0.75),(6.75,-0.25),(6,-0.65),(6.35,-0.25),(8,1.25),(7.75,0.75),(7.35,0.75),
(7.75,-0.75),(7.35,-0.75),(7.25,-0.25),(7.65,-0.25),(7.25,0.25),
(8,0.65),(7.65,0.25),(8,-1.25),(8.25,-0.75),(8.65,-0.75),(8.75,-0.25),
(8.25,0.75),(8.65,0.75),(8.75,0.25),(9,0.35),
(8,-0.65),(10,1.25),(9.75,0.75),(9.35,0.75),(9.25,0.25),
(10,-1.25),(9.75,-0.75),(9.35,-0.75),(9.25,-0.25),(9,-0.35),(9.65,-0.25),
(10,-0.65),(10,0.65),(9.65,0.25),(8.35,0.25),(8.35,-0.25)}
\node at \p [circle,fill,inner sep=.8pt]{};

\foreach \q in {(-6,0),(-4,0),(-1,1),(-1,-1),(0,0),(1,1),(1,-1),(2,0),(3,1),(3,-1)}
\node at \q [circle,fill,inner sep=1pt]{};

\node at (-1,1.4) {1};
\node at (-1,-1.5) {2};
\node at (0,0.4) {3};
\node at (1,1.4) {4};
\node at (1,-1.5) {5};
\node at (2,0.4) {6};
\node at (3,1.4) {7};
\node at (3,-1.5) {8};
\end{tikzpicture}
\caption{Figure of two steps of the iteration that produces a variant of a Laakso space from \cite{Laakso}. In the figure, an edge is replaced by two copies glued at points which lie 1/4 and 3/4 of their lengths. The graph together with its labeling is described in more detail as part of Example \ref{ex:counterexample}.}
\label{fig:laakso}
\end{figure}

Our result will be more general, and we will indicate which assumptions and mechanisms lead to the failure of being quasisymmetric to a Loewner space. In particular, we give a fairly large family of examples, all of which share crucial features and all of which are fairly simple to describe using the substitution procedure above. Indeed, by relatively minor assumptions, see Assumptions \ref{Assumptions: CLP}, an IGS produces an approximately self-similar and combinatorially Loewner limit space. While we do not resolve the full question of Kleiner in all settings (in particular group boundaries or the Sierpi\'nski carpet), these examples indicate mechanisms and tools by which one can prove that other examples are also not quasisymmetric to Loewner spaces. Further work is needed to determine which combinatorially Loewner fractals attain their conformal dimensions. Next, we shall briefly discuss relationships to recent work. 

\subsection{Background}
Until the present work, not much definitive was known about Kleiner's question. Due to the work of \cite{KL, Carrasco,Shaconf,MurCA,Kigamiweighted,MT} it was known how to construct metrics in the conformal gauge of $X$. The construction was powerful enough to lead to a characterization of the conformal dimension as a critical exponent for certain discrete modulus problems, \cite{BourK,Carrasco}. It was also known, that the hypothetical Loewner metric would be a minimizer for the Ahlfors regular conformal dimension, see e.g. \cite[Introduction]{CEB}. These insights suggested that, if a Loewner metric were to exist in the conformal gauge of a metric space, it could be obtained by studying minimizers for discrete modulus problems $\Mod_Q(\Gamma, G_m)$ with $Q$ equal to the conformal dimension.

By applying discretization, certain structures from this hypothetical Loewner metric can be gleaned form the original undeformed space. Versions of Sobolev spaces were constructed in \cite{BS,murugan2023first, Lindquist}, which would coincide with the Sobolev space associated to the Loewner space in the sense of \cite{Ch,Shanmun,AdMS}, if the space was quasisymmetric to a Loewner space. Also, discrete modulus at the critical exponent would be comparable to the continuous modulus of the associated Loewner space \cite{Haissinski,P89,TQuasi}.  Finally, Murugan and Shimizu constructed an energy measure associated to the Sobolev spaces mentioned before, that would yield a measure with respect to which the Hausdorff $Q$-measure of the Loewner structure would be mutually absolutely continuous \cite{murugan2023first}. Their work built upon the work of Kigami \cite{kigami} and Shimizu \cite{shimizu}.

 Indeed, the work of Murugan and Shimizu presented a strategy by which the existence of a Loewner metric in the conformal gauge could be determined in general. This was inspired by earlier work of Kajino and Murugan \cite{MN}, where they observed that a similar strategy could be used for $p=2$ to resolve a problem on the attainment of a related conformal walk dimension. In this latter work the energy was associated with the exponent $p=2$, and in the first, $p$ is chosen to coincide with the conformal dimension. In both works, one observes that if the hypothesized structure were to exist, the measure and metric could be obtained by considering Sobolev functions on the space and their energy measures. Our argument also is based on computing minimizers for modulus, but the logic is more direct, and we do not discuss energy measures or the associated Sobolev space in this work.

\subsection{New Fractal Spaces}
Finally, we note that the fractal spaces that we construct using the framework of IGS, may have broader interest. Indeed, there has been much work on analysis on fractals, especially regarding diffusions see e.g. \cite{barlow,kigamibook,strichartz,strichartzbook} and conformal geometry \cite{MT}. Some major examples in this realm have been the Cantor sets, Sierpi\'nski gasket, general gaskets, Sierpi\'nski carpets, Menger curve and the Vicsek set. Indeed, these classes of spaces are often presented as the prototypical examples of fractals and all arise from iterated function systems of Euclidean space. Often, the Cantor sets, gasket and Vicsek cases have been solved first, and then (with considerably more work) the Sierpi\'nski carpet/Menger curve case is pursued. Cantor sets are quantitatively totally disconnected and, in the gasket and Vicsek spaces, cut points give significant simplifications for calculations. In the cases with local cut-points many problems have been fully resolved, such as determining the conformal dimension and its attainment \cite{MT,TysonWu}, and the construction of diffusions \cite{kigamibook}. For carpets, the lack of cut points is reflected in significant difficulty in performing calculations. Consequently, most of the same questions such as the attainment problem of conformal dimension are not known for carpets (e.g. \cite{MT}), while others, such as the existence of a diffusion \cite{BarlowBass1989,barlowbass} and construction of a Dirichlet form \cite{KuzuokaZhou} are quite involved.  

The reason to study these example spaces is that they are ``toy models'' used to develop tools to approach much  more difficult problems, such as: the behavior of random walks on percolation clusters (e.g. \cite[Introduction]{barlow}), or Cannon's, Kleiner's and Kapovitch's conjectures in geometric group theory  \cite[Introduction]{CEB}. This is natural for the reason that connected components of supercritical percolation clusters in the plane, and planar group boundaries, are homeomorphic to carpets (see e.g. \cite{kapkleiner}). It has been suggested that the next step in difficulty from gasket-type spaces is carpet-like spaces.  However, at present, beyond some intriguing work in \cite{BourK, kigami,shimizu,murugan2023first}, there is not much progress in the study of carpets. Indeed, the difficulties in the analysis of carpet-like spaces has stymied development of the theory.

The class of fractals that we present here are a new class of natural spaces where to study analysis and which do not arise as attractors of iterated function systems in Euclidian spaces. (Indeed, the metrics on these spaces often do not bi-Lipschitz embed into Euclidean space, see e.g. \cite{LangPlaut} for an argument for the example given in Figure \ref{fig:laaksodiamond}). Similar to carpet's, the examples (often) do not have cut points. Despite this, the computations in these examples are even somewhat simpler than those for gaskets. (This is especially true for exponents $p\neq 2$. For $p=2$ the computations for gaskets are also very tractable, see e.g. \cite{kigami, strichartz} and compare these to \cite{Strichartz04A,Strichartz04B, Caopgasket}, which showcase the $p\neq 2$ theory for gaskets.) Thus, in a sense, the spaces lie in difficulty between gaskets and carpets, and offer a new setting to explore prototypical behaviors of random walks and conformal geometry.  The examples have some common features: they all have topological and Assouad-Nagata dimension equal to $1$, see Remark \ref{rmk:topdim}. Further, as explained in Example \ref{ex:counterexample} and as follows from \cite{Anderson-1,Anderson-2}, the constructions often yield spaces homeomorphic to the Menger curve. 

Often analysis on fractals hinges upon potential theory and computing minimizers of energy. In general this is quite difficult - especially for carpet-like spaces (see e.g. \cite{Kwapisz}). The inability to perform exact computations has left many interesting questions unanswered. The distinguished feature of the graphs arising from iterated graph systems is the ability to effectively compute minimizers of energy. This explicit computability and richness of the family suggest that studying the class further could yield crucial insights and may allow one to answer further questions in analysis on fractals that have heretofore resisted effort.  Answering Conjecture \ref{conj:kleiner} seems to be just the first example of these insights. Another example is understanding when the so called $p$-walk dimension $d_{w,p}$ is strictly grater than $p$, which we define in our framework in Definition \ref{def:walkdim}. It has been observed in multiple different works (see e.g. \cite{kajino2024korevaarschoen,murugan2023first,GrigoryanBesov2}) that this problem is closely related to the structure of Sobolev spaces. See also \cite[Problem 4 in Section 6.3]{kigami}. In Subsection \ref{subsec: p-walkdim} we present a simple characterisation in our framework of when $d_{w,p} \geq p$ is a strict inequality and present examples for both cases.

Next, we will go over each of the main concepts: conformal dimension, the Loewner property and the combinatorial Loewner property.  We will also explain the proof of the counterexamples using them. The iterated graph systems are described and studied in Section \ref{sec:linrep}. A reader mostly interested in the construction of such fractal spaces and graphs can read this section independent of most of the other sections, and this section does not require most of the notions from quasiconformal geometry. Modulus on graphs is presented in Section \ref{sec:modulus} and the combinatorial Loewner property for certain IGSs is proven in Section \ref{sec:combloew}.   In Section \ref{sec:mainthm} we study porosity and prove the main theorems. The concepts that we use are mostly standard, and their classic theory can be read from textbooks such as \cite{He,MT, shabook}.

\section{Core concepts}\label{sec:concepts}

\subsection{Conformal dimension} Let $Q \in (1,\infty)$ and let $(X,d)$ be a metric space. 
A compact metric space $X$ is said to be $Q$-\emph{Ahlfors regular}, if there exists a constant $C>1$ so that the $Q$-dimensional Hausdorff measure satisfies $C^{-1} r^Q \leq \cH^Q(B(x,r))\leq C r^Q$ for all $r\in (0,\diam(X)]$ and all $x\in X$. Here, the \emph{Hausdorff measure} is given by
\[
\cH^Q(A):=\lim_{\delta \to 0} \cH^Q_\delta(A),
\]
where the \emph{Hausdorff content}'s are defined by
\[
\cH^Q_\delta(A):=\inf\{\sum_{i=1}^\infty \diam(A_i)^Q : A \subset \bigcup_{i=1}^\infty A_i, \diam(A_i)\leq \delta\}.
\] 
The value $Q$ is equal to the \emph{Hausdorff dimension} for a $Q$-Ahlfors regular space. In general, if $\mu$ is any Radon measure on $X$, we say that $(X,d,\mu)$ is $Q$-Ahlfors regular if $C^{-1} r^Q \leq \mu(B(x,r))\leq C r^Q$ for all $r\in (0,\diam(X)]$ and $x\in X$. In this case there exists a constant $C$ for which $\frac{1}{C}\mu \leq \cH^Q \leq C \mu$, and the metric space $X$ is also $Q$-Ahlfors regular. 
 
Our focus will be on self-similar spaces with self-similairity defined as in \cite{BourK}. See \cite{Carthesis} for a more detailed discussion on this, and related, notions.
\begin{definition}\label{def:selfsim} We say that a metric space $X$ is \emph{approximately self-similar}, if there exists a constant $L\geq 1$ so that for every $x\in X$ and every $r\in (0,\diam(X)]$, there exists an open set $U_{x,r}\subset X$ and a  $L$-biLipschitz map $f_{x,r}:(B(x,r),d/r)\to U_{x,r}$, where the ball $B(x,r)$ is equipped with the scaled metric $d/r$.
\end{definition}
The maps $f_{x,r}$ in the statement will be refered to as scaling maps.  A map $f:(X,d_X)\to (Y,d_Y)$ is a $L$-\emph{biLipschitz map} (or embedding) if for all $x,y\in X$, we have
\[
\frac{1}{L} d_X(x,y) \leq d_Y(f(x),f(y)) \leq L d_X(x,y).
\]

Let $\eta:[0,\infty)\to [0,\infty)$ be a continuous homeomorphism.  A homeomorphism $f:(X,d_X)\to (Y,d_Y)$ is said to be $\eta$-\emph{quasisymmetric} if for every $x,y,z\in Z$ with $x\neq z$, we have
\[
\frac{d_Y(f(x),f(y))}{d_Y(f(x),f(z))} \leq \eta\left(\frac{d_X(x,y)}{d_X(x,z)}\right).
\]
We say that $f:X\to Y$ is a \emph{quasisymmetry}, if it is an $\eta$-quasisymmetry for some $\eta$. In these cases, we say that $X$ is quasisymmetric to $Y$. For more background on these mappings and the notions below, see \cite{He}. At this juncture, we note that we will usually omit the subscript of a metric $d$, where the space is clear from context.

The \emph{conformal gauge} of a metric space $X$ is given by
\[
\cG(X,d):=\{\hat{d} : \hat{d} \text{ is a metric on } X \text{ s.t. } \id:(X,d)\to (X,\hat{d}) \text{ is a quasisymmetry}\}.
\]
Here $\id$ is the identity map. This gauge is generally quite big -- see \cite{Carrasco} for a description of metrics in this gauge (also, \cite{MurCA,Shaconf}). We will focus mostly on those metrics, which are Ahlfors regular. This gives the \emph{Ahlfors regular conformal gauge}:
\begin{align*}
\cG_{\rm AR}(X,d):=\{\hat{d} : &\ \  \id:(X,d)\to (X,\hat{d}) \text{ is a quasisymmetry and } \\
& (X,\hat{d}) \text{ is $\widehat{Q}$-Ahlfors regular for some } \widehat{Q}>0\}.
\end{align*}
Where the metric $d$ is clear from context, we will drop it in the notation and simply write $\cG(X), \cG_{\rm AR}(X)$.  Among all such metrics, one may wish to find an ``optimal metric'', which satisfies some nice analytic or symmetry properties. A frequently considered problem is to minimize the Hausdorff dimension among metrics in the (Ahlfors regular) conformal gauge: the \emph{(Ahlfors regular) conformal dimension} is given by
\[
\dims_{\rm AR}(X):=\inf\{\widehat{Q} : \hat{d}\in \cG_{\rm AR}(X,d) \text{ and } (X,\hat{d}) \text{ is $\widehat{Q}$-Ahlfors regular}\}.
\]
A version of this invariant, the conformal Hausdorff dimension was  first considered by Pansu \cite{P89}, and the present definition was given in \cite{BourdonPajot}. Further variants of this conformal dimension, such as conformal Assouad, have alse been studied. See \cite{MT} for an introduction to conformal dimension and \cite{CEB,Carrasco} for some further work on these dimensions. The various definitions are equivalent in self-similar settings, see \cite{EBConf}. Since all the spaces considered in this paper are self-similar, we will mostly speak of just conformal dimension to refer to these notions.

If the infimum in the definition of the (Ahlfors regular) conformal dimension is attained, we say that $X$ attains its (Ahlfors regular) conformal dimension. The question to determine if this happens for a given space, or a class of spaces, is called the \emph{attainment problem}. Determining the conformal dimension and the attainment problem are notoriously hard. The value of the conformal dimension can, however, be numerically estimated. Numerical computations of the conformal dimension of the Sierpi\'nski carpet were performed in \cite{Kwapisz}. These calculations are based on characterizations of the conformal dimension in terms of combinatorial modulus, see \cite{Carrasco, BourK} and the related work \cite{MurCA,Shaconf,Kigamiweighted}.  In order to describe these characterizations, we will next define combinatorial and continuous moduli.

\subsection{Continuous and combinatorial modulus}\label{subsec:moduli}

Out of the two notions of moduli, the continuous one is easier to define, and is much more classical, see e.g. \cite{Ziemer,shabook}. We will define it only in the case when $X$ is $Q$-Ahlfors regular for some $Q \in (1,\infty)$. While modulus can be defined with respect to any measure on $X$, we will only consider Hausdorff $Q$-measures in the present paper. Let $\Gamma$ be a family of  curves. We say that a Borel measurable $\rho:X\to [0,\infty]$ is admissible for $\Gamma$ if for all rectifiable curves $\gamma\in \Gamma$ we have $\int_\gamma \rho ds \geq 1$, where the integration is with respect to the length measure. If $p\in [1,\infty)$, the $p$-modulus of a curve family $\Gamma$ is defined as
\[
\Mod_p(\Gamma):=\inf\left\{ \int \rho^p d\cH^Q : \rho \text{ is admissible for } \Gamma\right\}.
\]
For a quick overview of modulus and its properties, see \cite{Ziemer,shabook}. A crucial issue with this definition is that it is not a quasisymmetry invariant. Indeed, different metric in the conformal gauge do not yield mutually absolutely continuous Hausdorff-measures, nor do they agree on the class of rectifiable curves. The only setting, where invariance can be obtained, is if $f:X\to Y$ is a quasisymmetry between two $Q$-Ahlfors regular metric spaces -- see \cite{TQuasi}.

In order to get full invariance, the notion of modulus needs to be discretized. There are a variety of approaches to do this, \cite{Haissinski, TQuasi, HKInvent, KL, kigami, BourK, EBConf, murugan2023first,Lindquist}. For some comparisons between different approaches, see \cite{EBConf}. In each approach the space and curves are discretized in different ways. We will follow \cite{BourK} since  the combinatorial Loewner property was originally expressed in terms of it. (It may be possible to express the combinatorial Loewner property also in terms of other notions of modulus and by employing other graphical approximations, but this is not essential for the present work.)

Let $\alpha,L_*>1$. An $\alpha$-approximation to a compact metric space $X$ is a sequence of incidence graphs $\{G_m=(V_m,E_m)\}_{m \in \N}$, where for each $m\in \N$ the set of vertices $V_m$ is a collection of subsets of $X$, which forms a covering of $X$, and which satisfies the following conditions
\begin{enumerate}
    \item For every $v\in V_m$ there exists a $z_v\in v$ for which \[
    B(z_v, \alpha^{-1}L_*^{-m}) \subset v \subset B(z_v, \alpha L_*^{-m}).
    \]
    \item For every pair of distinct $v,w\in V_m$ we have 
    \[B(z_v,\alpha^{-1}L_*^{-m})\cap B(z_w, \alpha^{-1}L_*^{-m})=\emptyset.\] 
\end{enumerate}
Here, an incidence graph $G=(V,E)$ of $X$ is a collection $V$ of subsets of $X$, which cover $X$, where $\{v,u\}\in E$ is an edge if and only if $u\cap v \neq \emptyset$. For a given subset $F \subseteq X$ we write
\[
    G[F] := \{ v \in V : X_v \cap F \neq \emptyset \}.
\]
If $\Gamma$ is a family of subsets of $X$ we define
\[
    G[\Gamma] := \{ G[\gamma] : \gamma \in \Gamma \}.
\]

In \cite{BourK}, the value $L_*=2$ is used exclusively, but the results apply just the same with a different $L_*$. It will be useful to allow a different $L_*>1$ to conform more directly with our examples. We note that in our examples, the logic of \cite{BourK} is a bit reversed. We construct a sequence of graphs $G_m$, and $X$ is a limit space under Gromov-Hausdorff convergence. A posteriori, the graphs $G_m$ can also be identified with an $\alpha$-approximation in the sense of Bourdon and Kleiner.
Fix $m\in \N, p \geq 1$ and let $\rho:V_m\to \R_{\geq 0}$, and $\gamma\subset X$ be a continuous curve identified with its image. We define the \emph{$\rho$-length} of $\gamma$ as
\[
L_\rho(\gamma):=\sum_{v \in G_m[\gamma]} \rho(v).
\]
Let $p\in [1,\infty)$. We define the \emph{$p$-mass} of $\rho$ as
\[
\cM_p(\rho):=\sum_{v\in V_m} \rho(v)^p.
\]
Given a family of continuous curves $\Gamma$ in $X$, we say that $\rho : V_m \to \R_{\geq 0}$ is \emph{$\Gamma$-admissible} if $L_{\rho}(\gamma) \geq 1$ for all $\gamma \in \Gamma$. We define the \emph{discrete modulus} of $\Gamma$ with respect to $G_m$ by 
\[
\Mod_p^{D}(\Gamma, G_m):=\inf\{\cM_p(\rho) : L_\rho(\gamma)\geq 1, \forall \gamma \in \Gamma, \rho:V_m\to \R_{\geq 0}\}.
\]
Using the discrete modulus, Keith and Kleiner describe the conformal dimension of self-similar spaces as a critical exponent -- see \cite{Carrasco} for a proof by Carrasco, which was developed independently. For $\delta>0$ we write $\Gamma_\delta =\{\gamma : \diam(\gamma)>\delta\}$  and 
\begin{equation}\label{eq:moddeltadef}
\cM_{\delta,p}^{(m)} := \Mod_p^{D}(\Gamma_{\delta}, G_m).
\end{equation}

\begin{proposition}[Corollary 1.4 \cite{Carrasco}]\label{prop:confdimchar} Let $X$ be an approximately self-similar metric space. There exists a $\delta_0>0$ so that for all $\delta\in (0,\delta_0)$ we have
\[
\dims_{\rm AR}(X)=\inf\left\{p \geq 1 : \lim_{k\to\infty} \cM_{\delta,p}^{(k)} =0 \right\}.
\]
\end{proposition}

This characterization is very powerful, but difficult to use since the moduli in question are quite challenging to compute. At present, for most spaces, this can only be done via numerical approximation. See e.g. \cite{Kwapisz} for an explicit numerical approximation of the conformal dimension of the Sierpi\'nski carpet based on this approach. 

\subsection{Loewner condition}

Given the work of Heinonen and Koskela \cite{HK}, and given its implications for the study of quasisymmetric maps, one special property of a metric in the conformal gauge is the Loewner property. This property is expressed using the concept of a continuous modulus.  Now, through its connection to capacity, see e.g. \cite{kigami, Ziemer}, it is particularly relevant to consider the moduli of path families connecting sets. Let $E,F\subset X$ be closed sets, and let $\Gamma(E,F)$ be the collection of rectifiable curves connecting $E$ to $F$. The Loewner condition is given by a lower bound for the modulus of curves between two continua. If $E,F\subset X$ are two non-degenerate continua (i.e. compact connected subsets with $\diam(E),\diam(F)\neq 0$), then their relative distance is given by $\Delta(E,F)=\frac{d(E,F)}{\min\{\diam(E),\diam(F)\}}$. 

A metric space $(X,d)$ is said to be $Q$-Loewner, if it is $Q$-Ahlfors regular and
\[
\Mod_Q(\Gamma(E,F))\geq \phi(\Delta(E,F)^{-1}),
\]
for some increasing function $\phi:[0,\infty)\to[0,\infty)$ and for all non-degenerate continua $E,F\subset X$.

The problems of minimizing Hausdorff dimension in the conformal gauge, and that of finding a Loewner metric in the conformal gauge are very closely related. First, if $(X,d')$ is $Q'$-Loewner for some metric $d'\in \cG_{\rm AR}(X)$, then $Q'=\dims_{\rm AR}(X)$, and thus the metric $d'$ attains the infimum in the definition of the conformal dimension. This follows from \cite{TQuasi,P89}. Indeed, the Loewner property is somewhat excessive here, since it suffices that $\Mod_{Q'}(\Gamma)>0$ for any family of rectifiable curves.

Conversely, if $(X,d')$ is $Q'$-Ahlfors regular, $d'\in \cG_{\rm AR}(X)$ and  $Q'=\dims_{\rm AR}(X)$, then $X$ almost admits a positive modulus family of curves. Indeed, Keith and Laakso showed that a weak tangent of $X$ admits such a positive modulus family of curves \cite{KL}. If $X$ is further assumed to be self-similar, then this family of curves can be elevated to the space $X$, see e.g. \cite[Proof of Corollary 1.6]{BK05}.

The Loewner property is, however, stronger than a mere positivity of modulus for some family of curves. It involves a quantitative, and scale-invariant, lower bound for the connecting modulus between any pair of continua. However, it was already observed in \cite{BK05}, that further symmetry in the space elevates having positive modulus to a Loewner property. Thus, the Loewner property can be seen as involving three ingredients: positive modulus, scale invariance and symmetry. Further, as a consequence, at least for group boundaries, the attainment of conformal dimension is equivalent to the existence of a Loewner metric in the conformal gauge, \cite{BK05}. How can one then recognize the existence of a Loewner metric in the conformal gauge? Bourdon and Kleiner proposed a discrete version of the Loewner property, the combinatorial Loewner property as a potential condition for this \cite{BourK}.

\begin{figure}[!ht]
\begin{tikzpicture}
\draw (-7,-1)--(-6,0)--(-5,1)--(-4,0)--(-3,1);
\draw (-6,0)--(-5,-1)--(-4,0);
\draw (-5,-1)--(-5,1);
\draw (-7,1)--(-6,0);
\draw (-4,0)--(-3,-1);

\draw[->,thick] (-2.5,0)--(-1.5,0);

\node at (-7,1) [circle,fill,inner sep=1pt]{};
\node at (-6,0) [circle,fill,inner sep=1pt]{};
\node at (-5,1) [circle,fill,inner sep=1pt]{};
\node at (-4,0) [circle,fill,inner sep=1pt]{};
\node at (-3,1) [circle,fill,inner sep=1pt]{};
\node at (-5,-1) [circle,fill,inner sep=1pt]{};
\node at (-1,1) [circle,fill,inner sep=1pt]{};
\node at (-7,-1) [circle,fill,inner sep=1pt]{};
\node at (-3,-1) [circle,fill,inner sep=1pt]{};

\draw (-1,-1)--(0,0);
\draw (2,0)--(3,1);
\draw (1,-.12)--(1,.12);
\draw (-1,1)--(0,0);
\draw (2,0)--(3,-1);

\draw [-] (0,0) to [bend right = 5em] (1,-.12);
\draw [-] (1,-.12) to [bend right = 5em] (2,0);

\draw [-] (0,0) to [bend left = 5em] (1,.12);
\draw [-] (1,.12) to [bend left = 5em] (2,0);

\node at (-1,1) [circle,fill,inner sep=1pt]{};
\node at (-1,-1) [circle,fill,inner sep=1pt]{};
\node at (0,0) [circle,fill,inner sep=1pt]{};
\node at (1,.12) [circle,fill,inner sep=1pt]{};
\node at (1,-.12) [circle,fill,inner sep=1pt]{};
\node at (2,0) [circle,fill,inner sep=1pt]{};
\node at (3,1) [circle,fill,inner sep=1pt]{};
\node at (3,-1) [circle,fill,inner sep=1pt]{};

\node at (1.2,0) {\small $\varepsilon$};
\end{tikzpicture}
\caption{The figure shows a deformation of the example in Figure \ref{fig:replacementrule}, where the central edge has a length $\varepsilon$, while each other edges still have length $1/4$. It can be shown that this deformation is a quasisymmetry.}
\label{fig:decreasedim} 
\end{figure}

\begin{example} At this juncture it is helpful to return to the primary counterexample from Figure \ref{fig:replacementrule}. We will be a bit informal and will not explain all details, since we will later give a detailed proof of the main theorem using somewhat different ideas. However, the argument given here gives a good intuition of what is going on.

Let $Y$ be obtained from the IGS in Figure \ref{fig:laakso}. It is self-similar and consists of $8$ copies scaled by $1/4$. Thus this example has Hausdorff dimension equal to $\log(8)/\log(4)=3/2$. The space $Y$ is also  $3/2$-Loewner space, and thus has conformal dimension equal to its Hausdorff dimension. This can be seen in one of three ways. Two proofs are obtained by recognizing the construction as a variant of the Laakso space construction from \cite{Laakso}, or the inverse limit construction of \cite{CK}. The proofs of the relevant Poincar\'e inequality in these settings, and the equivalence of the Poincar\'e inequality and Loewner condition from \cite{HK} implies the Loewner condition. We will also give a third self-contained argument in Example \ref{ex:loewnerreplacement}, which is based on an explicit discrete modulus comutations.

Our counterexample $X$ which arises form the IGS in Figure \ref{fig:replacementrule} has Hausdorff dimension $\log(9)/\log(4)>3/2$. Since this space contains a copy of $Y$ from the previous paragraph, we must have $\dims_{\rm AR}(X)\geq 3/2.$ It is not minimal for the dimension, and Figure \ref{fig:decreasedim} shows how the dimension can be reduced by iteratively reducing the length of the central segment to $\varepsilon$. This reduction is repeated similar to a multiplicative cascade: The central edge arising from replacing an edge $e$ at some level will have length $\varepsilon$ times the length of the edge $e.$ For example, the edges arising from replacing the central edge at the first level have length $\varepsilon/4$ and $\varepsilon^2$, where the latter is the length of the central edge of the central edge. It can be shown, that for every $\varepsilon>0$ such a cascade produces a new metric in the conformal gauge of $X$.

As $\varepsilon\to 0$, the Hausdorff dimension of this new metric approaches $3/2$, establishing that $\dims_{\rm AR}(X)=3/2$. This argument gives an intuitive computation of the conformal dimension and an explanation for why it is not attained. Attainment would require setting $\varepsilon = 0$, which would degenerate the space by collapsing the central edge as in Figure \ref{fig: Why not attain}. This is not allowed for homeomorphisms, let alone for quasisymmetries. This argument can also be made rigorous, but our proof involves a simpler and more general argument that uses the fact that $X$ is combinatorially $3/2$-Loewner and Proposition \ref{prop:porous}.

The construction with $\varepsilon>0$ can be thought of as a construction of a metric by using an admissible weight function for the left-right modulus problem similar to \cite{Carrasco} and \cite{KL}, which comes from a modulus problem -- see also \cite{Kigamiweighted} for similar ideas and \cite{MT} for similar constructions. The fact that the optimal weight function vanishes for the central edge corresponds with the fact that attainement is not possible.
\begin{figure}[!h]
\begin{tikzpicture}
    \draw (-1,-1)--(0,0);
\draw (2,0)--(3,1);
\draw (1,-.12)--(1,.12);
\draw (-1,1)--(0,0);
\draw (2,0)--(3,-1);

\draw [-] (0,0) to [bend right = 5em] (1,-.12);
\draw [-] (1,-.12) to [bend right = 5em] (2,0);

\draw [-] (0,0) to [bend left = 5em] (1,.12);
\draw [-] (1,.12) to [bend left = 5em] (2,0);

\node at (-1,1) [circle,fill,inner sep=1.5pt]{};
\node at (-1,-1) [circle,fill,inner sep=1.5pt]{};
\node at (0,0) [circle,fill,inner sep=1.5pt]{};
\node at (1,.12) [circle,fill,inner sep=1.5pt]{};
\node at (1,-.12) [circle,fill,inner sep=1.5pt]{};
\node at (2,0) [circle,fill,inner sep=1.5pt]{};
\node at (3,1) [circle,fill,inner sep=1.5pt]{};
\node at (3,-1) [circle,fill,inner sep=1.5pt]{};

\node at (1.2,0) {$\varepsilon$};

\draw[->,thick] (3.2,0) to ["\small{$\varepsilon \to 0$}"] (4.8,0)  ;

\draw (5,-1)--(6,0);
\draw (5,1)--(6,0);
\draw [-] (6,0) to [bend right = 5em] (7,0);
\draw [-] (6,0) to [bend left = 5em] (7,0);

\draw [-] (7,0) to [bend right = 5em] (8,0);
\draw [-] (7,0) to [bend left = 5em] (8,0);

\draw (8,0)--(9,1);
\draw (8,0)--(9,-1);

\foreach \q in {(5,-1),(6,0),(5,1),(7,0),(8,0),(9,1),(9,-1)}
\node at \q [circle,fill,inner sep=1.5pt]{};
\end{tikzpicture}
\caption{Assigning the parameter value $\varepsilon = 0$ degenerates the geometry.}
\label{fig: Why not attain}
\end{figure}
\end{example}

\subsection{Combinatorial Loewner property}\label{subsec:combloewdef}
By discretizing the continuous modulus, Bourdon and Kleiner gave a quasisymmetrically invariant combinatorial Loewner property, \cite{BourK}.
\begin{definition}[CLP]
    Let $Q > 1$ and $L_* > 1$.
    We say that a compact metric space satisfies the \emph{combinatorial $Q$-Loewner property} if there exists two positive increasing functions $\phi,\psi$ on $(0,\infty)$ with $\lim_{t \to 0} \psi(t) = 0$ which satisfies the following two conditions:
    \begin{enumerate}[label=\textcolor{blue}{(CLP\arabic*)}]
        \item \label{CLP1} If $F_1,F_2 \subseteq X$ are two disjoint non-degenerate continua so that $L_*^{-n} \leq \diam(F_1) \land \diam(F_2)$, then for all $m \in \N$
        \begin{equation*}
            \phi(\Delta(F_1,F_2)^{-1}) \leq \Mod_Q^{D}(\Gamma(F_1,F_2), G_{n + m}).
        \end{equation*}
        \item \label{CLP2} If $x \in X$, $r \geq L_*^{-n}$ and $C > 0$ then for all $m \in \N$
        \begin{equation*}
            \Mod_Q^{D}\left(\Gamma\left(\overline{B(x,r)},X \setminus B(x,Cr)\right), G_{n + m}\right) \leq \psi(C^{-1}).
        \end{equation*}
    \end{enumerate}
\end{definition}
The definition does not depend on the choice of an $\alpha$-approximation, as shown in \cite[Proposition 2.2]{BourK}, and it is a quasisymmetry invariant \cite[Theorem 2.6]{BourK}. The definition is slightly different from that in \cite{BourK}, and we adopt the corrected version from \cite{clais}. As noted by \cite{BourK}, and shown in detail by \cite{clais}, the combinatorial Loewner property implies the following \emph{LLC property}: For each $x,y\in X$, there exists a connected set $E_{x,y}\subset X$ with $\diam(E)\leq Cd(x,y)$.

If $(X,d)$ is $Q$-Loewner, then it is also $Q$-combinatorially Loewner, see \cite[Theorem 2.6]{BourK}. Moreover, by the quasisymmetric invariance, if $\cG_{\rm AR}(X)$ possesses a Loewner metric, then $(X,d)$ is also combinatorially Loewner. Kleiner's conjecture regards the converse to this statement. We also note that, regardless of attainment, it follows from Proposition \ref{prop:confdimchar} that if $X$ is combinatorially $Q$-Loewner for $Q \in (1,\infty)$, then $Q=\dims_{\rm AR}(X)$; see also \cite[Lemma 4.2]{EBConf}.

For combonatorially Loewner spaces the existence of a Loewner metric in the conformal gauge is equivalent to minimizing conformal dimenion. This follows from work by Ha\"{\i}ssinski on the comparability of discrete and continuous modulus \cite[Proposition B.2]{Haissinski}, and was explicitly stated in \cite{CEB}.
\begin{proposition}\label{prop:noloew} Let $Q>1$. If $X$ is a compact combinatorially $Q$-Loewner metric space, then there exists a metric $d'\in \cG_{\rm AR}(X)$ which is $Q$-Ahlfors regular with  $Q=\dims_{\rm AR}(X)$ if and only if there exists a metric $d'\in \cG_{\rm AR}(X)$ which is $Q$-Loewner.
\end{proposition}
This equivalence is crucial for us, as we will answer Kleiner's conjecture by precluding the attainment of the conformal dimension.

\subsection{Proof of counterexamples}
There are two key insights that lead to our counterexamples.
\begin{enumerate}
\vskip.3cm
    \item \textbf{Explicit computability:} There is a family of spaces, 
    \[
    \{X : X \text{ is a limit of an iterated graph system}\},
    \] 
    similar to those introduced by Laakso in \cite{Laakso}, for which discrete moduli are ``easy'' to compute. Consequently for these examples, we can compute their conformal dimensions explicitly, see Section \ref{sec:modulus} and Theorem \ref{thm:combloew}.
    \vskip.3cm
    \item \textbf{Porosity:} There are pairs of spaces $(Y,X)$ in the given family where $Y\subset X$ is a porous subset, and where using the first insight we can find $\dims_{\rm AR}(X)=\dims_{\rm AR}(Y)$.  
    \vskip.3cm
\end{enumerate}
We say that $Y\subset X$ is a (uniformly) \emph{porous} subset of $X$, if there exists an $r_0>0$ and $c>0$ so that for every $y\in Y$, and every $r\in (0,r_0)$ there exists $x\in B(y,r)\subset X$ for which $B(x,cr)\cap Y = \emptyset$. Porous sets of Ahlfors regular spaces have strictly smaller Assouad dimension, and the conformal dimension is bounded from above by the Assouad dimension of any quasisymmetric image. These facts together with the quasisymmetric invariance of porosity yield the following proposition. See Section \ref{subsec:porosity} for definitions and a detailed proof.

\begin{proposition}\label{prop:porous}
    Let $Y\subset X$ be a porous subset. If $\dims_{\rm AR}(Y)=\dims_{\rm AR}(X)$, then $X$ does not attain its conformal dimension.
\end{proposition}

In our case, $X$ will be the space arising from the IGS in Figure \ref{fig:replacementrule}, and $Y$ the Laakso-type space arising from Figure \ref{fig:laakso}. A direct computation will show that their conformal dimensions are equal. The fact that the IGS for $Y$ arises from a sub-graph of that of $X$ yields that $Y$ is porous in $X$. From these, we see that $X$ can not attain its conformal dimension, and thus by Proposition \ref{prop:noloew} can not contain a Loewner metric in its conformal gauge. This, together with the fact that $X$ is combinatorially Loewner and approximately self-similar yields the proof of the main theorem.

\section{Iterated graph systems}\label{sec:linrep}

\subsection{General notation}

Simple graphs (without loops) are pairs $G=(V,E)$, where $V$ is some finite set of vertices and $E\subset \{\{x,y\} : x,y\in V, x\neq y\}$ is a finite  set of edges. An edge $e$ with end points $x,y$ is denoted $\{x,y\}$. Notice that we do not permit loops in our graphs. If $\{v,w\}\in E$ we say that $w$ is a neighbor of $v$. The degree of a vertex $v\in V$ is the number $\degr(v)$ of edges to which it is adjacent. The degree of the graph $G$ is defined as $\degr(G) = \max_{v \in V} \degr(v)$. 

\begin{remark}\label{rmk:multigraph} For simplicity of notation and to avoid some tedious case analysis, we restrict attention to simple graphs. Our results can also be stated for multi-graphs, and these require mostly small changes to the statements and arguments, e.g. in Proposition \ref{prop: Similarity maps}.
\end{remark}

It will be convenient to identify $G$ with its geometric realization, which is obtained by gluing copies of unit edges indexed by the edges of $G$ along their end points, whenever two edges are adjacent. This makes $G$ into a simplicial complex. If this geometric realization is connected, we call $G$ connected.

A vertex path in $G$ is a sequence $\theta = [v_1,\dots, v_k]$, where $\{v_i,v_{i+1}\}\in E$ for each $i=1,\dots, k-1$. Here $\len(\theta) = k-1$ is the length of the vertex path $\theta$, and we say that the path connects $v_1$ to $v_n$. The path metric $d_G(v,w)$ in a connected metric space $G$ between two vertices $v,w\in V$ is the smallest $n\in \N$ such that there is a vertex path of length $n$ connecting $v$ to $w$. If $A,B \subseteq V$ then $\Theta(A,B)$ is the set of paths connecting a point $v \in A$ to a point $u \in B$. Further if $d$ is a metric on $V$ and $A\subseteq V$ is a set of vertices, then $\diam(A,d):=\max_{a,b\in A} d_G(a,b)$ is the diameter in the metric $d$. Where the metric $d$ is clear from context, we may drop it from the notation. 

\subsection{Definition of IGS}
We now define iterated graph systems based on edge replacements. A construction similar to the following definition has appeared earlier in \cite[Definition 2.1]{Leeslash}. For iterated graph systems based on replacing vertices, see \cite{ReplacementGraphs2024}.

\begin{definition}\label{def:IGS}
An \emph{iterated graph system (IGS)} consists of a graph $G_1=(V_1,E_1)$, together with the following data.
\begin{enumerate}
\vskip.3cm
    \item A finite set $I$ called the \emph{gluing set}.
    \vskip.3cm
    \item Each edge $e=\{v,w\}$ and endpoint point $v\in e$ is associated with an injective mapping $\phi_{v,e}:I\to V_1$, whose image $I_{v,e} := \phi_{v,e}(I)$ is an independent set of $V_1$ - that is, no edges between vertices in $I_{v,e}$.
    \vskip.3cm
    \item We have $I_{v,e}\cap I_{w,e}=\emptyset$ for each edge $e=\{v,w\}$. 
\end{enumerate}
\end{definition}
The set $I$ together with the maps $\phi_{v,e}$ will be referred to as the gluing rules. In general, a graph $G=(V,E)$ is said to be labeled by $G_1$ if for every edge $e$ of $G$ and any $v\in e$, there is an associated injective mapping $\phi_{v,e}:I\to V_1$, whose image is an independent set. Given a labeled graph $G$ we can form a replacement graph $\hat{G}=(\hat{V},\hat{E})$, whose vertices are $\hat{V}=V_1\times E / \sim$, where we identify 
\[
(\phi_{x,e}(v),e)\sim (\phi_{x,f}(v),f)
\]
for every $e,f\in E$ which share an end point $x\in V$ and $v\in I$. Further, we define edge sets
\[
\hat{E}=\{[(v,e)],[(w,e)]: \{v,w\}\in E_1, e\in E\}.
\]
This amounts to replacing each edge in $G$ by a copy of $G_1$, which are glued along the images of the mappings $\phi_{v,e}$. We can also define a labeling for $\hat{G}$ by 
\[
\phi_{[(v,e)], \{[(v,e)],[(w,e)]\}} = \phi_{v,\{v,w\}}
\]

This \emph{replacement rule} can be applied recursively to $G_1$ to produce a sequence of graphs $G_k$: Set $G_{k+1}:=\hat{G_k}$ for $k\in \N$. 
The edges and vertices of $G_k=(V_k,E_k)$ can be described as follows.
\begin{enumerate}
\vskip.3cm
    \item Let $V_{k+1}=V_1\times E_k / \sim$, where we identify vertices with the relationships $(\phi_{x,e}(v),e)\sim (\phi_{x,f}(v),f)$ for every $e,f\in E$ which share an end point $x$ and $v\in I$.
    \vskip.3cm
    \item $\{[v,e],[w,e]\}\in E_{k+1}$ if $\{v,w\}\in E_1$. 
    \vskip.3cm
    \item $\phi_{[v,e], \{[v,e],[w,e]\}} = \phi_{v,\{v,w\}}$.
    \vskip.3cm
\end{enumerate}
Notice that in the notation of equivalence classes $[(v,e)]$, we drop the parenthesis and write $[v,e]$. When $(\phi_{x,e}(v),e)\sim (\phi_{x,f}(v),f)$, we say that the equivalence relation is given by the vertex $x$. We call the graphs $G_k$ thus constructed \emph{replacement graphs.}

\begin{remark}\label{rmk:symmetricIGS}
    We have chosen the above fairly general framework in order to encompass a wide variety of examples. In all relevant examples of this paper, we will focus on a significantly simpler subclass of iterated graph systems. In our case, the graph $G_1=(V_1,E_1)$ is a fixed graph, and we have just two maps $\phi_{\pm}:I\to V_1$. The choice of mapping corresponds to a choice of orientation for an edge. An orientation for edges is a choice of an ordered pair $(v,w)$ (or $(w,v)$), for each edge $e=\{v,w\}$, in which case we say that $e$ is positively oriented from $v$ and to $w$ (or $w$ to $v$). Given such an orientation, we define $\phi_{v,e}=\phi_-$ if $e=\{v,w\}$ is positively oriented from $v$, and $\phi_{w,e}=\phi_+$ if $e$ is positively oriented towards $w$. We call these types of constructions \emph{oriented iterated graph systems}.  One simple way to obtain an orientation for edges is by ordering vertices with a well-ordering and declaring $\{v,u\}$ positively ordered if $v<u$. This will be the primary way we introduce an orientation in our examples.
    
    This simplification reduces the number of different maps $\phi_{v,e}$ one would need to consider below. Another simplification, which will hold in many, but not all, of our examples, is that there is a graph isomorphism $\eta_e:G_1\to G_1$ with $\eta_e \circ \phi_\mp = \phi_\pm$. If an oriented IGS satisfies this condition, we say that the IGS is \emph{symmetric}. For symmetric IGS, it does not matter how the edges are oriented to define maps $\phi_{v,e}$. Indeed, in such cases, the sequence of graphs produced with different orientations are isomorphic. 

    Given this independence on the orientation for symmetric IGSs, we have striven to use notation that avoids the use of orientations wherever possible. Indeed, the only instance where we need these orientations is in defining the explicit maps $\phi_{v,e}$ in our examples. The desire to avoid orientations leads to the framework above and considering most of the time the maps $\phi_{v,e}$, instead of using the maps $\phi_{\pm}$ together with an (arbitrary) orientation.
\end{remark}

\subsection{Examples}
 Next, we shall describe some important examples of IGSs and the fractals they produce.

\begin{figure}[!ht]
\begin{tikzpicture}[scale=0.8]
\draw (-5,0)--(-4,0);
\draw[<-,thick] (-3,0)--(-2,0);

\draw (-1,0)--(0,0)--(1,1)--(2,0)--(3,0);
\draw (0,0)--(1,-1)--(2,0);
\draw[<-,thick] (4,0)--(5,0);

\draw (6,0)--(6.25,0)--(6.5,0.25)--(6.75,0)--(7,0);
\draw (6.25,0)--(6.5,-.25)--(6.75,0);
\draw (9,0)--(9.25,0)--(9.5,0.25)--(9.75,0)--(10,0);
\draw (9.25,0)--(9.5,-.25)--(9.75,0);
\draw (7,0)--(7.25,0.25)--(7.25,0.75)--(7.75,.75)--(8,1);
\draw (7.25,0.25)--(7.75,0.25)--(7.75,.75);
\draw (7,0)--(7.25,-0.25)--(7.25,-0.75)--(7.75,-.75)--(8,-1);
\draw (7.25,-0.25)--(7.75,-0.25)--(7.75,-.75);
\draw (8,1)--(8.25,0.75)--(8.75,0.75)--(8.75,.25)--(9,0);
\draw (8.25,0.75)--(8.25,0.25)--(8.75,.25);
\draw (8,-1)--(8.25,-0.75)--(8.75,-0.75)--(8.75,-.25)--(9,0);
\draw (8.25,-0.75)--(8.25,-0.25)--(8.75,-.25);

\node at (-4,0) [circle,fill,inner sep=1pt]{};
\node at (-5,0) [circle,fill,inner sep=1pt]{};

\node at (-1,0) [circle,fill,inner sep=1pt]{};
\node at (0,0) [circle,fill,inner sep=1pt]{};
\node at (1,1) [circle,fill,inner sep=1pt]{};
\node at (1,-1) [circle,fill,inner sep=1pt]{};
\node at (2,0) [circle,fill,inner sep=1pt]{};
\node at (3,0) [circle,fill,inner sep=1pt]{};

\node at (6,0) [fill,inner sep=1.5pt]{};
\node at (6.25,0) [circle,fill,inner sep=1pt]{};
\node at (6.5,0.25) [circle,fill,inner sep=1pt]{};
\node at (6.75,0) [circle,fill,inner sep=1pt]{};
\node at (6.5,-.25) [circle,fill,inner sep=1pt]{};
\node at (7,0) [fill,inner sep=1.5pt]{};

\node at (9,0) [fill,inner sep=1.5pt]{};
\node at (9.25,0) [circle,fill,inner sep=1pt]{};
\node at (9.5,0.25) [circle,fill,inner sep=1pt]{};
\node at (9.5,-0.25) [circle,fill,inner sep=1pt]{};
\node at (9.75,0) [circle,fill,inner sep=1pt]{};
\node at (10,0) [fill,inner sep=1.5pt]{};

\node at (7.25,0.25) [circle,fill,inner sep=1pt]{};
\node at (7.25,0.75) [circle,fill,inner sep=1pt]{};
\node at (7.75,0.25) [circle,fill,inner sep=1pt]{};
\node at (7.75,0.75) [circle,fill,inner sep=1pt]{};
\node at (8,1) [fill,inner sep=1.5pt]{};

\node at (7.25,-0.25) [circle,fill,inner sep=1pt]{};
\node at (7.25,-0.75) [circle,fill,inner sep=1pt]{};
\node at (7.75,-0.25) [circle,fill,inner sep=1pt]{};
\node at (7.75,-0.75) [circle,fill,inner sep=1pt]{};
\node at (8,-1) [fill,inner sep=1.5pt]{};

\node at (8.75,0.25) [circle,fill,inner sep=1pt]{};
\node at (8.75,0.75) [circle,fill,inner sep=1pt]{};
\node at (8.25,0.25) [circle,fill,inner sep=1pt]{};
\node at (8.25,0.75) [circle,fill,inner sep=1pt]{};

\node at (8.75,-0.25) [circle,fill,inner sep=1pt]{};
\node at (8.75,-0.75) [circle,fill,inner sep=1pt]{};
\node at (8.25,-0.25) [circle,fill,inner sep=1pt]{};
\node at (8.25,-0.75) [circle,fill,inner sep=1pt]{};

\node at (-1,0.4) {1};
\node at (0,0.4) {2};
\node at (1,1.4) {3};
\node at (1,-0.6) {4};
\node at (2,0.4) {5};
\node at (3,0.4) {6};
\end{tikzpicture}
\caption{Figure of a symmetric IGS that produces the ``Laakso diamond'' space that first appeared in \cite{Laakso}, and was studied e.g. in \cite{Yair, LangPlaut}. The figure shows two steps of the replacement. }
\label{fig:laaksodiamond}
\end{figure}

\begin{example}\label{ex:laaksodiamond}
The much studied Laakso diamond space also arises as an IGS. This is shown in Figure \ref{fig:laaksodiamond}, where the first two stages of the replacement are shown. Here, an edge is replaced iteratively by graphs $G_1=(V_1,E_1)$ where $V_1:=\{1,2,3,4,5,6\}$ and \[E_1:=\{\{1,2\},\{2,3\},\{2,4\},\{4,5\},\{3,5\}, \{5,6\} \}.\] The gluing rules for the graphs are given by the index set $I=\{a\}$, and $\phi_+(a)=6$, and $\phi_-(a)=1$. We set $\phi_{v,\{v,w\}}=\phi_-$ if $v<w$ and otherwise $\phi_{v,\{v,w\}}=\phi_+$. The construction yields a sequence of planar graphs, which helps in drawing them and makes them simpler to visualize. This family of graphs has cut points, since the gluing set $I$ is a singleton set. This implies that the conformal dimension is equal to one, and these examples do not yield counterexamples to Conjecture \ref{conj:kleiner}.
\end{example}

\begin{figure}[!ht]
\begin{tikzpicture}[scale=0.8]
\draw (-6,0)--(-4,0);
\draw[<-,thick] (-3,0)--(-2,0);

\draw (-1,1)--(0,0)--(1,1);
\draw (-1,-1)--(0,0)--(1,-1);
\draw[<-,thick] (2,0)--(3,0);

\draw (4,1.25)--(4.5,0.5)--(5,0.25);
\draw (4,0.75)--(4.5,0.5)--(5,-.25);
\draw (4,-1.25)--(4.5,-0.5)--(5,0.25);
\draw (4,-0.75)--(4.5,-0.5)--(5,-0.25);

\draw (5,0.25)--(5.5,0.5)--(6,1.25);
\draw (5,-0.25)--(5.5,0.5)--(6,0.75);
\draw (5,0.25)--(5.5,-0.5)--(6,-1.25);
\draw (5,-0.25)--(5.5,-0.5)--(6,-0.75);

\node at (-4,0) [circle,fill,inner sep=1pt]{};
\node at (-6,0) [circle,fill,inner sep=1pt]{};

\node at (-1,1) [circle,fill,inner sep=1pt]{};
\node at (0,0) [circle,fill,inner sep=1pt]{};
\node at (1,1) [circle,fill,inner sep=1pt]{};
\node at (1,-1) [circle,fill,inner sep=1pt]{};
\node at (-1,-1) [circle,fill,inner sep=1pt]{};

\node at (4,1.25) [circle,fill,inner sep=1pt]{};
\node at (4.5,0.5) [circle,fill,inner sep=1pt]{};
\node at (4,0.75) [circle,fill,inner sep=1pt]{};
\node at (4.5,-0.5) [circle,fill,inner sep=1pt]{};
\node at (4,-0.75) [circle,fill,inner sep=1pt]{};
\node at (4,-1.25) [circle,fill,inner sep=1pt]{};
\node at (5,-0.25) [circle,fill,inner sep=1pt]{};
\node at (5,0.25) [circle,fill,inner sep=1pt]{};

\node at (5.5,0.5) [circle,fill,inner sep=1pt]{};
\node at (5.5,-0.5) [circle,fill,inner sep=1pt]{};
\node at (6,1.25) [circle,fill,inner sep=1pt]{};
\node at (6,0.75) [circle,fill,inner sep=1pt]{};
\node at (6,-0.75) [circle,fill,inner sep=1pt]{};
\node at (6,-1.25) [circle,fill,inner sep=1pt]{};

\node at (-1,1.4) {1};
\node at (-1,-0.6) {2};
\node at (0,0.4) {3};
\node at (1,1.4) {4};
\node at (1,-0.6) {5};
\end{tikzpicture}
\caption{Figure of a symmetric IGS that produces a Laakso space in the sense of \cite{La00}. The figure shows two replacements}
\label{fig:laaksospace}
\end{figure}

\begin{example}\label{ex:laaksospace}
The Laakso spaces from \cite{La00}, and its variants, can be obtained by IGSs, see Figures \ref{fig:laakso} and \ref{fig:laaksospace} for examples. Notice how in this example the gluing set $I$ consists of two vertices, which prevents the existence of cut-points. Whenever one employs such IGSs, the resulting graph will not be planar. In the example of Figure \ref{fig:laaksospace}, an edge is replaced iteratively by graphs $G_1=(V_1,E_1)$ where $V_1:=\{1,2,3,4,5\}$ and \[E_1:=\{\{1,3\},\{2,3\},\{3,4\},\{3,5\}\}.\] The gluing rules for the graphs are given by the index set $I=\{a,b\}$, and $\phi_+(a)=4, \phi_+(b)=5$, and $\phi_-(a)=1, \phi_-(a)=2$. We set $\phi_{v,\{v,w\}}=\phi_-$ if $v<w$ and otherwise $\phi_{v,\{v,w\}}=\phi_+$. These examples yield Loewner spaces, and we will see that one needs to add edges to these graphs to obtain counterexamples to Conjecture \ref{conj:kleiner}. (Since our proof does not explicitly require this Loewner property, we omit it.) By Remark \ref{rmk:topdim}, the spaces thus constructed have topological dimension one. We note that by the results of \cite{Anderson-1,Anderson-2}, any topologically one dimensional compact connected set without local cut points, and which does not have open sets homeomorphic to a subset of the plane, is homeomorphic to the Menger curve. This is not essential to us, but this characterization implies that this example, and more generally all the examples which satisfy Assumption \ref{Assumptions: CLP}, which guarantees the combinatorial Loewner property, are homeomorphic to Menger curves. For this, one simply needs to observe that the graphs $G_m$ are not planar for $m$ large enough, and that it follows form Proposition \ref{prop: Similarity maps} that any open set contains a homeomorphic copy of $G_m$. The non-planarity of $G_m$ can be easily checked by the classical Kuratowski's theorem \cite{kuratowski} on forbidden graph minors. 
\end{example}

\begin{example}\label{ex:counterexample} We describe now in detail the graph from Figure \ref{fig:replacementrule}. The labeling of vertices is also shown in that figure. Let $\hat{G}_1 = (\hat{V}_1,\hat{E}_1)$ be the graph with eight vertices $\hat{V}_1:=\{1,\dots, 8\}$ with the edges 
\begin{align*}
    \hat{E}_1:=\{&\{1,3\}, \{3,4\}, \{4,6\},\{6,7\}\\
             &\{2,3\},\{3,5\}, \{5,6\},\{6,8\} \\
             &\{4,5\}\}.
\end{align*}
Here, the first and second line yield the edges corresponding to the two copies of an interval, and the final line gives the added central edge. We also set $I=\{a,b\}$, and define two maps $\phi_\pm: I \to \hat{V}_1$ by $\phi_-(a)=1,\phi_-(b)=2$ and $\phi_+(a)=7, \phi_+(b)=8$. These correspond to the left and right end points of each edge. Then, the maps $\phi_{v,e}$ are defined by $\phi_{v,\{v,w\}}=\phi_-$ if $v<w$ and $\phi_{v,\{v,w\}}=\phi_+$ if $v>w$

If we remove the edge $\{4,5\}$ from $\hat{E}_1$ and keep all vertices, we obtain a graph $G_1 = (V,E)$, which yields the IGS for the Laakso-type space in Figure \ref{fig:laakso}.
\end{example}

\subsection{Basic properties}

Our assumptions on the iterated graph systems guarantee that equivalence relations are simple to describe. 
\begin{lemma} The relation $(\phi_{x,e}(v),e)\sim (\phi_{x,f}(v),f)$  for every $e,f\in E$ which share an end point $x$ is an equivalence relation.
\end{lemma}
\begin{proof}
    The relation is clearly reflexive and symmetric. We check transitivity. If $(v_1,e_1)\sim(v_2,e_2)\sim(v_3,e_3)$, then $e_1$ and $e_2$ share an end point $x$, and $e_2$ and $e_3$ share an end point $y$. If $x=y$, then $v_i=\phi_{x,e_i}(v)$ for $i=1,2,3$ and $(v_1,e_1)\sim (v_3,e_3)$. If $x\neq y$, then $v_2=\phi_{x,e_2}(v)$ and $v_2=\phi_{y,e_2}(w)$ for some $v,w\in I$, but then $I_{x,e_2}\cap I_{y,e_2} \neq \emptyset$, which by (3) in Definition \ref{def:IGS} is a contradiction. This completes the proof.
\end{proof}

We define maps $\pi_{k+1}:V_{k+1}\cup E_{k+1} \to V_k \cup E_k$ as follows. For each vertex $[(v,e)]\in V_{k+1}$ define $\pi_{k+1}([(v,e)])=e$ if $v\not\in I_{w,e}$ for any $w\in e$, and otherwise set $\pi_{k+1}([(v,e)])=w$ if $v\in I_{w,e}$. In Figure \ref{fig:laaksodiamond}, the vertices marked with squares are mapped to vertices, and vertices marked with circles are mapped to edges. In other words, vertices in gluing sets are mapped to vertices, and all other vertices are mapped to the edges which gave rise to them.

The map $\pi_{k+1}$ is easier to define for edges. Each edge in the replaced graph arises from some edge. Indeed, we set $\pi_{k+1}(\{[(v,e)],[(w,e)]\})=e$ for every edge $\{[(v,e)],[(w,e)]\}\in E_{k+1}$. Because $I_{w,e}$ is an independent set for each $w\in e$, the mapping $\pi_{k+1}$ is well defined. 

\begin{lemma} The mapping $\pi_k$ is well-defined, and $\pi_k:E_k\to E_{k-1}$ for all $k\geq 1$, and $\pi_{k}^{-1}(V_{k-1})\subset V_{k}$.
\end{lemma}
\begin{proof}
    First, we check the equivalence class for vertices. For the well-definedness it suffices to verify that if $(v,e)\sim (w,f)$, for $f\neq e$, then the map is well defined. In this case, $e$ and $f$ share a unique end point $x$. The equivalence relation means that $v\in I_{x,e}$ and $w\in I_{x,f}$, and thus in both cases $\pi_{k+1}([(v,e)])=\pi_{k+1}([(w,f)])=x$, as desired. The claim about images of edges and preimages of vertices follows from the definition. 
\end{proof}

For $n>m\geq 1$, define $\pi_{n,m}:=\pi_{m+1}\circ\cdots \circ\pi_n:V_{n}\cup E_{n}\to V_m \cup E_m$. We call a vertex/edge $x\in G_n$ an ancestor of a vertex/edge of $y\in G_m$ if $\pi_{n,m}(x)=y$. Let $e=\{v,w\}$ be an edge in $G_m$. We call the set $I_{e,n}=\pi_{n,m}^{-1}(e)\cup\pi_{n,m}^{-1}(v)\cup \pi_{n,m}^{-1}(w)$ a tile. When $n=m+1$, then $I_{e,n}$ is equivalent to the copy of $e\times G_1$ which exists in $G_{m+1}$.

\subsection{Scaling maps and paths}

For each $k\in \N$ the graphs $G_{k+1}$ consist of copies of $G_1$ glued together, but we can also see $G_{n+m}$ as arising from copies of $G_m$ glued along edges of $G_n$.  The gluing sets can be described as follows. For $m=1$, we define $I_{v,e}^{(m)}=\phi_{v,e}(I)$ and for $m>1$ we define $I_{v,e}^{(m)}=\pi_{m,1}^{-1}(\phi_{v,e}(I))$, if $\phi_{v,e}$ is some given map in a labeled graph $G_n$ for some $n\in \N$. The following proposition makes this all precise and introduces maps $\sigma_{e,m}$, which will be useful later. The proof is a straightforward induction which uses the definition, but is a bit technical. Thus at first reading one may want to skip the proof and to just focus on the conclusion, which is more intuitive.

\begin{proposition}\label{prop: Similarity maps}
For every $n,m \in \N$ and $e \in E_n$ there are maps $\sigma_{e,m} : V_m \to V_{n + m}$, the image of $\sigma_{e,m}$ denoted as $e \cdot G_m$ and the edges contained in this image as $e \cdot E_m$, with the following properties.
\begin{enumerate}[label=\textcolor{blue}{(SM\arabic*)}]
    \item \label{SM1} For every $e \in E_n$ the mapping $\sigma_{e,m}$ is injective and the collection of subsets $\{ e \cdot G_m \}_{e \in E_n}$ is a covering of $V_{n + m}$. Moreover, if $v,u \in V_{m}$ then $\{v,u\} \in E_m$ if and only if $\{ \sigma_{e,m}(v), \sigma_{e,m}(u) \} \in E_{n + m}$ and 
    \begin{equation}\label{eq:phieq}
    \phi_{v,\{ v,u \}} = \phi_{\sigma_{e,m}(v),\{\sigma_{e,m}(v),\sigma_{e,m}(u)\}}.
    \end{equation}
    \item \label{SM2} For distinct edges $e,f \in E_n$ the subsets $e \cdot G_m$ and $f \cdot G_m$ intersect if and only if $e,f$ have a common vertex $v$.
    Moreover, their intersection is
    \[
        \sigma_{e,m}\left(I_{v,e}^{(m)}\right) = \pi_{n+m,n}^{-1}(v) = \sigma_{f,m}\left(I_{v,f}^{(m)}\right)
    \]
    the set of ancestors of $v$.
    \item \label{SM3} For every $e \in E_n$ we have $e \cdot E_m = \pi_{n+m,n}^{-1}(e)$.
    In particular, $\{ e \cdot E_m \}_{e \in E_n}$ is a partition of $E_{n + m}$.
\end{enumerate}
\end{proposition}

\begin{proof}
    Fix $n \in \N$ and we will prove the existence of the mapping $\sigma_{e,m}$ by induction on $m$. The case $m = 1$ follows by setting $\sigma_{e,1} : V_1 \to V_{n + 1}, z \mapsto [z,e]$ for all $z \in V_1$ and $e\in E_n$.

    Assume that $m\geq 1$ and we have constructed $\sigma_{e,m}:V_m\to V_{n+m}$ for all $e \in E_n$ that satisfies \ref{SM1}, \ref{SM2} and \ref{SM3}. We define $\sigma_{e,m+1} : V_{m+1} \to V_{n + m +1}$ by
    \begin{equation}\label{eq:sigmadef}
        \sigma_{e,m+1}([z,\{ v,u \}]) = [z,\{\sigma_{e,m}(v),\sigma_{e,m}(u)  \}],
    \end{equation}
    where each vertex of $V_{m+1}$ can be represented by an equivalence class $[z,\{ v,u \}]$ for some edge $\{v,u\}\in E_m$ and some vertex $z\in V_1$.  By the induction hypothesis (IH) we get $\phi_{v,\{ v,u \}} = \phi_{\sigma_{e,m}(v),\{\sigma_{e,m}(v),\sigma_{e,m}(u)\}}$ and therefore $\sigma_{e,m+1}$ is a well-defined. The fact that $\sigma_{e,m+1}$ is an injection can be seen as follows. Suppose that $\sigma_{e,m+1}([z_1,\{ v_1,u_1 \}])=\sigma_{e,m+1}([z_2,\{ v_2,u_2 \}])$. Then the edges $\{\sigma_{e,m}(v_1),\sigma_{e,m}(u_1)  \}$ and $\{\sigma_{e,m}(v_2),\sigma_{e,m}(u_2)  \}$ must be either equal or adjacent. But, since $\sigma_{e,m}$ is an injection that preserves edges, this means that $\{ v_1,u_1 \}$ and $\{ v_2,u_2 \}$ must be either equal or adjacent. In the first case, $z_1=z_2$, and that the map is an injection follows. If they are adjacent, and (without loss of generality) the equivalence relation is given by the vertex $v_1=v_2$, we get $\phi_{\sigma_{e,m}(v_1),\{\sigma_{e,m}(v_1),\sigma_{e,m}(u_1)\}}(z_1)=\phi_{\sigma_{e,m}(v_2),\{\sigma_{e,m}(v_2),\sigma_{e,m}(u_2)\}}(z_2)$. From \eqref{eq:phieq} we get $\phi_{v_1,\{v_1,u_1\}}(z_1)=\phi_{v_2,\{v_2,u_2\}}(z_2)$ and $[z_1,\{ v_1,u_1 \}]=[z_2,\{ v_2,u_2 \}]$. This concludes injectivity.
    
    Also,
    \begin{align*}
        V_{n + m + 1} & = \bigcup_{\{v,u\} \in E_{n + m}} \{ [z,\{v,u \}] : z  \in V_1 \}\\
        & \stackrel{\text{IH}}{=} \bigcup_{e \in E_n} \bigcup_{\{ v,u\} \in E_m} \{ [z,\{\sigma_{e,m}(v),\sigma_{e,m}(u)\}] : z \in V_1 \}\\
        & = \bigcup_{e \in E_n} e \cdot G_{m + 1}.
    \end{align*}
    Furthermore,
    \begin{align*}
        & \{ \sigma_{e,m + 1}([z_1,\{v_1,u_1\}]), \sigma_{e,m+1}([z_2,\{v_2,u_2\}]) \} \in E_{n + m + 1}\\
        \stackrel{\eqref{eq:sigmadef}}{\iff} & \{\sigma_{e,m}(v_1),\sigma_{e,m}(u_1) \} = \{\sigma_{e,m}(v_2),\sigma_{e,m}(u_2) \} \text{ and } \{ z_1,z_2 \} \in E_{1}\\
        \stackrel{\text{IH}}{\iff} & \{ v_1,u_1 \} = \{ v_2,u_2 \} \text{ and } \{ z_1,z_2 \} \in E_{1}\\
        \iff & \{ [z_1,\{v_1,u_1\}], [z_2,\{v_2,u_2\}] \} \in E_{m + 1}.
    \end{align*}
    From these, it is also direct to verify \eqref{eq:phieq}. Next, let $\{[z_1,\{v,u\}], [z_2,\{v,u\}]\}\in E_{m+1}$ for some $\{v,u\}\in E_m$, and compute:
    \begin{align*}
        & \phi_{ \sigma_{e,m + 1}([z_1,\{v,u\}]),\{ \sigma_{e,m + 1}([z_1,\{v,u\}]), \sigma_{e,m+1}([z_2,\{v,u\}])\}}  \\
        & \quad = \phi_{ [z_1,\{\sigma_{e,m}(v),\sigma_{e,m}(u)\}],\{ [z_1,\{\sigma_{e,m}(v),\sigma_{e,m}(u)\})], [z_2,\{\sigma_{e,m}(v),\sigma_{e,m}(u)\}]\}} \\
     & \quad = \phi_{ \{z_1,\{z_1,z_2\}\} } \\
     & \quad = \phi_{ [z_1,\{v,u\}],\{ [z_1,\{v,u\}], [z_2,\{v,u\}]},
    \end{align*}
    where in the last two lines we used (3) from Definition \ref{def:IGS}. This covers \ref{SM1}.
    
    We move on to \ref{SM2}. Fix distinct edges $e,f \in E_{n}$. Let $\{v_1,u_1\}, \{v_2,u_2\}\in E_m$ and $z_1,z_2\in V_1$. Consider the following chain of equivalences.
    \begin{align*}
        & \sigma_{e,m + 1}([z_1\{ v_1,u_1 \}]) = \sigma_{f,m + 1}([z_2,\{ v_2,u_2 \}])\\
       \stackrel{\eqref{eq:sigmadef}}{\iff} & [z_1,\{ \sigma_{e,m}(v_1), \sigma_{e,m}(u_1)\}] = [z_2,\{ \sigma_{f,m}(v_2), \sigma_{f,m}(u_2)\}]\\
       \iff & \exists u \in \{ \sigma_{e,m}(v_1), \sigma_{e,m}(u_1) \} \cap \{ \sigma_{f,m}(v_2), \sigma_{f,m}(u_2)\} \text{ and }\\
       & \exists a \in I, z_1 = \phi_{u,\{\sigma_{e,m}(v_1), \sigma_{e,m}(u_1)\} }(a), z_2 = \phi_{u,\{\sigma_{f,m}(v_2), \sigma_{f,m}(u_2)\} }(a)\\
       \stackrel{\eqref{eq:phieq}}{\iff} & \exists u \in \{ \sigma_{e,m}(v_1), \sigma_{e,m}(u_1) \} \cap \{ \sigma_{f,m}(v_2), \sigma_{f,m}(u_2)\} \text{ and }\\
       & \exists a \in I, z_1 = \phi_{\sigma_{e,m}^{-1}(u),\{v_1, u_1 \}}(a), z_2 = \phi_{\sigma_{f,m}^{-1}(u),\{v_2,u_2\} }(a).
    \end{align*}
    Notice that in the equivalences above $u\in e\cdot G_m\cap f\cdot G_m$. Hence, by (IH), it follows that $e \cdot G_{m + 1} \cap f \cdot G_{m + 1} \neq \emptyset$ if and only if $e$ and $f$ have a common vertex. Indeed the "only if" part follows by choosing an arbitrary $u \in e\cdot G_m\cap f\cdot G_m$ and $a \in I$. Then assume that $e$ and $f$ indeed has a common vertex $u$ and we prove the expression. Again, from the equivalences, we see that $e \cdot G_{m + 1} \cap f \cdot G_{m + 1}$ is the set of ancestors of the vertices in $e \cdot G_{m} \cap f \cdot G_{m}$. This yields
    \begin{align*}
        e \cdot G_{m + 1} \cap f \cdot G_{m + 1} & = \pi_{n + m + 1, n + m}^{-1}(e \cdot G_{m} \cap f \cdot G_{m})\\
        & \stackrel{\text{IH}}{=}  \pi_{n + m + 1, n + m}^{-1}(\pi_{n+m,n}^{-1}(u))\\
        & = \pi_{n+m+1,n}^{-1}(u).
    \end{align*}
    Furthermore,
    \begin{align*}
        \pi_{n+m+1,n}^{-1}(u) & \stackrel{\text{IH}}{=} \pi_{n+m+1,n+m}^{-1}\left(\sigma_{e,m}\left(I_{u,e}^{(m)}\right)\right)\\
        & = \pi_{n+m+1,n+m}^{-1}\left(\sigma_{e,m}\left(\pi_{m,1}^{-1}(I_{u,e})\right)\right)\\
        & \stackrel{eq\ref{eq:phieq}}{=} \sigma_{e,m + 1}\left(\pi_{m+1,m}^{-1}\left(\pi_{m,1}^{-1}(I_{u,e}\right)\right)\\
        & = \sigma_{e,m + 1}\left(I_{u,e}^{(m + 1)}\right).
    \end{align*}
    The second last equality follows from the following equivalence.
    \begin{align*}
        & v_1 \in \pi_{n+m+1,n+m}^{-1}(\sigma_{e,m}\left(\pi_{m,1}^{-1}(I_{u,e})\right))\\
        \iff & \exists v_2 \in \sigma_{e,m}\left( \pi_{m,1}^{-1}(I_{u,e}) \right),\, \exists u_2 \in e \cdot G_m, \, \exists a \in I, \text{ so that }\\
        & v_1 = [\phi_{v_2,\{ v_2,u_2 \}}(a), \{ v_2,u_2 \}]\\
        \stackrel{\eqref{eq:phieq}}{\iff} & \exists v_2 \in \sigma_{e,m}\left( \pi_{m,1}^{-1}(I_{u,e}) \right),\, \exists u_2 \in e \cdot G_m, \, \exists a \in I, \text{ so that }\\
        & v_1 = [\phi_{\sigma_{e,m}^{-1}(v_2),\{ \sigma_{e,m}^{-1}(v_2),\sigma_{e,m}^{-1}(u_2) \}}(a), \{ v_2,u_2 \}]\\
        \stackrel{\eqref{eq:sigmadef}}{\iff} & \exists v_2 \in \sigma_{e,m}\left( \pi_{m,1}^{-1}(I_{u,e}) \right),\, \exists u_2 \in e \cdot G_m, \, \exists a \in I, \text{ so that }\\
        & v_1 = \sigma_{e,m+1}([\phi_{\sigma_{e,m}^{-1}(v_2),\{ \sigma_{e,m}^{-1}(v_2),\sigma_{e,m}^{-1}(u_2) \}}(a),\{ \sigma_{e,m}^{-1}(v_2),\sigma_{e,m}^{-1}(u_2) \}])\\
        \stackrel{\ref*{SM1}}{\iff} & \exists v_3 \in \pi_{m,1}^{-1}(I_{u,e}) , \, \exists u_3 \in V_m, \, \exists a \in I \text{ so that }\\
        & v_1 = \sigma_{e,m+1}([\phi_{v_3,\{ v_3,u_3 \}}(a), \{ v_3,u_3 \}])\\
        \iff & v_1 \in \sigma_{e,m+1}(\pi_{m+1,m}^{-1}(\pi_{m,1}^{-1}(I_{u,e}))).
    \end{align*}
Hence \ref{SM2} is clear.

Lastly we argue \ref{SM3} fix $e \in E_n$. Choose an edge $\hat{e} \in e \cdot E_{m + 1}$. By \ref{SM1} we have the expression
\begin{align*}
    \hat{e} & = \{ \sigma_{e,m+1}([z_1,\{ v_1,u_1 \}]), \sigma_{e,m+1}([z_2,\{ v_1,u_1 \}]) \}\\
    & = \{ [z_1, \{ \sigma_{e,m}(v_1),\sigma_{e,m}(v_2) \}],[z_2, \{ \sigma_{e,m}(v_1),\sigma_{e,m}(v_2) \}]  \}
\end{align*}
so $\pi_{n + m,n}(\hat{e}) \in e \cdot E_m \stackrel{\text{IH}}{=} \pi_{n+m,n}^{-1}(e)$. Hence $e \in \pi_{n+m+1,n}^{-1}(e)$ and $e \cdot E_{m+1} \subseteq \pi_{n+m+1,n}^{-1}(e)$. To prove the other inclusion, choose $\widehat{e} \in \pi_{n+m+1,n}^{-1}(e)$. By construction, $\hat{e}$ has a unique presentation as of form $\widehat{e} = \{ [z_1,f], [z_2,f] \}$ for $f \in E_{n+m}$. Then $\pi_{n+m+1,n+m}(\widehat{e}) = f \in e \cdot E_{m}$ so $f = \{ \sigma_{e,m}([z_3,f']),\sigma_{e,m}([z_4,f']) \}$ and
\begin{align*}
    \hat{e} & = \{ \sigma_{e,m+1}([z_1, \{ [z_3,f'],[z_4,f'] \} ]), \sigma_{e,m+1}([z_2,\{ [z_3,f'],[z_4,f'] \}])  \} \in e \cdot E_{m+1}.
\end{align*}
Hence $e \cdot E_{m+1} = \pi_{n+m+1,n}^{-1}(e)$.
\end{proof}

\begin{corollary}\label{Corollary: Subpaths L to R}
    Let $\theta = \left[v_1,\dots,v_k\right] $ be a path in $G_{n+m}$. Suppose there are distinct vertices $v,u \in V_n$ and indices $1 \leq j_1, j_2 \leq k$ so that
    \[
        v_{j_1} \in \pi_{n+m,n}^{-1}(v) \text{ and } v_{j_2} \in \pi_{n+m,n}^{-1}(u).
    \]
    Then there is an edge $\hat{e} = \{ v,\hat{u} \} \in E_n$ and a sub-path $\hat{\theta}$ of $\theta$ from $\pi_{n+m,n}^{-1}(v)$ to $\pi_{n+m,n}^{-1}(\hat{u})$ contained in $\hat{e} \cdot G_m$.
\end{corollary}

\begin{proof}
    We shall write $e_j = \pi_{n+m,n}(\{ v_{j +1} ,v_j\})$ for $j=1,\dots, k-1$. For simplicity, we assume that $j_1 = 1$ and $j_2 = k$. Furthermore, by possibly taking a sub-path, we may assume that if $u' \neq v$ then $v_j \notin \pi_{n+m,n}^{-1}(u')$ for all $1 \leq j \leq k - 1$.
    Indeed then, by \ref{SM3}, we must have $v \in e_j$ for all $1 \leq j \leq k - 1$.

    We set $\hat{e} := \{ v,\hat{u} \} = e_{k-1}$ and
    \[
    l := \min \{ 1 \leq j \leq k-1 : e_s = \hat{e} \text{ for all } j \leq s \leq k-1 \}.
    \]
    By definition, $e_j = \hat{e}$ for all $l \leq j \leq k-1$.
    By \ref{SM1} and \ref{SM3}, the vertices on the sub-path $\hat{\theta} := \left[v_{l}, \dots, v_k \right]$ are contained in $\hat{e} \cdot G_m$. Now if $l = 1$, we are done. Assume $l > 1$. By construction, $e_{l - 1} \neq \hat{e}$, which, by \ref{SM2} yields
    \[
        v_l \in (e_{l-1} \cdot G_m) \cap (\hat{e} \cdot G_m) = \pi_{n+m,n}^{-1}(v).
    \]
\end{proof}

\begin{proposition}\label{prop: Path between disjoint edges}
    Let $\theta = \left[v_1,\dots,v_k\right] $ be a path in $G_{n+m}$. Suppose there are edges $e,f \in E_n$ with no common vertex, and indices $1 \leq j_1, j_2 \leq k$ so that
    \begin{equation*}
        v_{j_1} \in e \cdot G_m \text{ and } v_{j_2} \in f \cdot G_m.
    \end{equation*}
    Then there is an edge $\hat{e} = \{ v,u \} \in E_n$ so that $v \in e$ and a subpath $\hat{\theta}$ of $\theta$ from $\pi_{n+m,n}^{-1}(v)$ to $\pi_{n+m,n}^{-1}(u)$ contained in $\hat{e} \cdot G_m$.
\end{proposition}

\begin{proof}
    By \ref{SM2} and \ref{SM3} there are vertices $v' \in e$ and $u' \in f$ so that $v_{j_1} \in \pi_{n+m,n}^{-1}(v')$ and $v_{j_2} \in \pi_{n+m,n}^{-1}(u')$ for some $1 \leq j_1, j_2 \leq k$. By assumption, $v' \neq u'$, so the claim follows from Corollary \ref{Corollary: Subpaths L to R}.
\end{proof}

\begin{proposition}\label{prop: Path decomposition}
    Let $\theta = \left[v_1,\dots,v_k\right] $ be a path in $G_{n+m}$ and $A,B \subseteq V_n$ be non-empty disjoint sets.
    If $v_1 \in \pi_{n+m,m}^{-1}(A)$ and $v_k \in \pi_{n+m,m}^{-1}(B)$ then
    there are numbers $1 \leq k_1 < s_1 \leq k_2 < s_2 \leq \dots, \leq k_l < s_l \leq k$ and vertices $u_1,\dots, u_l \in V_n$ so that
    \begin{enumerate}
        \item $[u_1,\dots,u_l] \in \Theta(A,B)$, and
        \item the sub-path $\theta_i = [v_{k_i},\dots,v_{s_i}]$ of $\theta$ connects $\pi_{n + m,n}^{-1}(u_i)$ to $\pi_{n + m,n}^{-1}(u_{i + 1})$ and is contained in $ e_i \cdot G_m$ where $e_i = \{ u_i,u_{i+1} \} \in E_n$.
    \end{enumerate}
\end{proposition}

\begin{proof}
    We will define the numbers $k_i < s_i$ inductively as follows: First we apply Corollary \ref{Corollary: Subpaths L to R} to $\theta$ and obtain an edge $e_1 = \{ u_1,u_2 \}$ so that $\pi_{n+m,n}(v_1) = u_1$ and a sub-path $\theta_1$ of $\theta$ contained in $e_1 \cdot G_m$. We set $k_1$ to be the first index, and $s_1$ to be the last index of this sub-path, so we have $\theta_1 = [v_{k_1},\dots,v_{s_1}]$. If $v_{s_1} \in B$ then we terminate the process. Otherwise, we apply Corollary \ref{Corollary: Subpaths L to R} to $[v_{s_1},v_{s_1 + 1}, \dots, v_k]$ and obtain an edge $e_2 = \{ u_2,u_3 \}$ and a sub-path $\theta_2$ contained in $e_1 \cdot G_m$. We contiunue this process until for some $l \geq 1$ we have $v_{s_l} \in B$.
\end{proof}

\subsection{Metrics on replacement graphs}

We will next consider the problem of defining natural metrics on the sequence of graphs produced by an IGS. There are many ways to do this, but with the following assumption the construction becomes easier to describe. 
\begin{definition}
    For $L_* \geq 2$, we say that an iterated graph system satisfies \emph{$L_*$-uniform scaling property} if for all $e=\{v,u\} \in E_1$ and for every $x\in I_{v,e}, y\in I_{u,e}$ we have $d_{G_1}(x,y)=L_*$.
\end{definition}
It is easy to construct examples of IGSs with this property. Indeed, all examples in this paper satisfy it, and the value of $L_*$ is readily computable. 

Throughout this paper, $C_{\diam}$ denotes the diameter of the graph $G_1$ equipped with the path metric $d_{G_1}$ and $C_{\deg} := \deg(G_1)$ denotes the maximum degree the graph $G_1$.
Clearly $L_*$-uniform scaling property implies $L_* \leq C_{\diam}$, however, in general, we may have $L_* < C_{\diam}$.
We equip the set $V_n$ with the metric $d_n := L_*^{-n} \cdot d_{G_n}$ which is the path metric on $G_n$ scaled by $L_*^{-n}$.
Furthermore define the semi-metric $d_n$ on $E_n$ by
\[
    d_n(e,f) := \min_{x \in e, y \in f} d_n(x,y).
\]
We also write $d_n(e,y) = d_n(y,e) := \min_{x \in e} d_n(x,y)$ for $y \in V_n$ and $e \in E_n$.
\begin{lemma}\label{lemma: Distance of ancestors}
    Assume that the iterated graph system satisfies the $L_*$-uniform scaling property and fix $n \in \N, m \geq 0$.
    If $x_{n + m},y_{n + m} \in V_{n + m}$ are ancestors of vertices $x_n,y_n\in V_n$ respectively, then the following hold.
    \begin{enumerate}[label=\textcolor{blue}{(DL\arabic*)}]
        \item \label{lemma: Distance of ancestors (1)} If $x_n \neq y_n$ then $d_{n+m}(x_{n+m},y_{n+m}) = d_{n}(x_{n},y_{n})$.
        \item \label{lemma: Distance of ancestors (2)} If $x_n = y_n$ then $d_{n+m}(x_{n+m},y_{n+m}) \leq C_{\diam}\cdot L_*^{-n}$.
        \item \label{lemma: Distance of ancestors (3)} 
        Suppose $e_n \in E_n$ so that $x \in e_n \cdot G_{m}$.
        If $x_n \in e_n$ and $\widehat{x}_n \in V_{n+m}$ is an ancestor of $x_n \in V_n$ then
        $d_{n+m}(x,\widehat{x}_n) \leq C_{\diam} \cdot L_*^{-n}$.
    \end{enumerate}
\end{lemma}

\begin{proof}
    We will first consider \ref{lemma: Distance of ancestors (1)}.
    It is sufficient to prove the case $m = 1$.
    Let $\theta_{n+1} = [v_1,\dots,v_k]$ be a path in $G_{n + 1}$ from $x_{n + 1}$ to $y_{n + 1}$.
    By Proposition \ref{prop: Path decomposition}, there is a path $[u_1,\dots u_l]$ from $x_{n}$ to $y_{n}$ and for $i=1,\dots, l-1$ pairwise disjoint sub-paths $\hat{\theta}_{i}$ from $\pi_{n}^{-1}(u_i)$ to $\pi_{n}^{-1}(u_{i + 1})$ which are contained in $\{ u_i,u_{i+1} \} \cdot G_1$.
    By unifrom scaling property, $\len(\hat{\theta}_{i}) \geq L_*$ for all $i$, which gives
    \[
        \len(\theta) \geq \sum_{i = 1}^{l - 1} \len\left(\hat{\theta}_{i}\right) \geq L_* (l-1) \geq L_*^{1+n} \cdot d_{n}(x_{n},y_{n}).
    \]
    Since the path $\theta$ was arbitary, we have $d_{n+1}(x_{n+1},y_{n+1}) \geq d_n(x_n,y_n)$.
    
    To prove the other inequality, let $\theta_{n} = [u_1,\dots,u_l]$ be a shortest path from $x_{n}$ to $y_{n}$.
    For every $1 \leq i \leq l$ fix an ancestor $v^{(i)}$ of $u_i$ so that $v^{(1)} = x_{n+1}$ and $v^{(l)} = y_{n+1}$.
    By uniform scaling property, for each $1 \leq i \leq l - 1$, there is a path $\theta^{(i)}$ from $v^{(i)}$ to $v^{(i + 1)}$ of length $L_*$.
    By joining these paths, we obtain a path $\theta_{n+1}$ from $x_{n+1}$ to $y_{n+1}$ of length
    \[\len(\theta_{n+1}) = \sum_{i = 1}^{l - 1} \len\left(\theta^{(i)}\right) = L_* \cdot (l - 1) = L_*^{n+1} \cdot d_n(x_n,y_n)\]
    which yields $d_{n+1}(x_{n+1},y_{n+1}) \leq d_n(x_n,y_n)$.
    
    Next we move to \ref{lemma: Distance of ancestors (2)}. Without loss of generality assume $x_{n+m}\neq y_{n+m}$.
    Let $1 \leq k \leq m$ be the smallest integer so that
    \[x_{n+k} := \pi_{n+m,n+k}(x_{n+m}) \neq y_{m + k} := \pi_{n+m,n+k}(y_{n+m}).\]
    Since $\pi_{n+k}(x_{n+k}) = \pi_{n+k}(y_{n+k})$, $x_{n+k},y_{n+k} \in e \cdot G_1$ for some $e \in E_{n+k-1}$. Hence
    \[d_{n+m}(x_{n+m},y_{n+m}) \stackrel{\ref*{lemma: Distance of ancestors (1)}}{=} d_{n+k}(x_{n+k},y_{n+k}) \stackrel{\ref*{SM1}}{\leq} C_{\diam} \cdot L_{*}^{-(n+k)} \leq C_{\diam} \cdot L_*^{-n}\]

    Lastly we prove \ref{lemma: Distance of ancestors (3)}.
    We choose edges $e_{n + i} \in E_{n + i}$ for $0 \leq i \leq m - 1$ so that $x \in e_{n + i} \cdot G_{m - i}$
    and $\pi_{n+i + 1}(e_{n + i + 1}) = e_{n + i}$.
    For every $0 \leq i \leq m - 1$ we also choose $x_{n + i} \in e_{n + i}$.
    Then we fix the ancestors $\widetilde{x}_{n+i} \in V_{n + i + 1}$ of $x_{n + i}$ and their ancestors  $\widehat{x}_{n + i} \in V_{n + m}$.
    For every $0 \leq i \leq m - 1$, as $x_{n + i + 1},\widetilde{x}_{n+i} \in e_{n + i} \cdot G_1$, there is a path of length at most $C_{\diam}$ between these two vertices. Therefore $d_{n + i + 1}(x_{n + i + 1},\widetilde{x}_{n + i}) \leq C_{\diam} \cdot L_{*}^{-(n+i+1)}$.
    By \ref{lemma: Distance of ancestors (1)} and \ref{lemma: Distance of ancestors (2)} we have
    \[
        d_{n+m}(\widehat{x}_{n+i},\widehat{x}_{n + i + 1}) \leq C_{\diam} \cdot L_*^{(-n + i + 1)}.
    \]
    Similarly $d_{n+m}(x,\widehat{x}_{n + m - 1}) \leq C_{\diam} \cdot L_*^{-(n+m)}$.
    We now obtain
    \begin{align*}
        d_{n+m}(x,\widehat{x}_{n}) & \leq d_{n+m}(x,\widehat{x}_{n+m - 1}) + \sum_{i = 0}^{m - 2} d_{n+m}(\widehat{x}_{n + i},\widehat{x}_{n + i + 1})\\
        & \leq \sum_{i = 0}^{m - 1} C_{\diam} \cdot L_*^{-(n + i + 1)}\\
        & \leq C_{\diam} \cdot L_*^{-n}.
    \end{align*}
\end{proof}
\begin{corollary}\label{corollary: Distance between edges}
    If $n > m \in \N$ then the following hold.
    \begin{enumerate}[label=\textcolor{blue}{(DL\arabic*)}]
    \setcounter{enumi}{3}
        \item \label{corollary: Distance between edges (1)} $d_m(\pi_{n,m}(e),\pi_{n,m}(f)) \leq d_n(e,f)$.
        \item \label{corollary: Distance between edges (2)} $d_{n}(e,f) \leq d_{m}(\pi_{n,m}(e),\pi_{n,m}(f)) + 2C_{\diam}\cdot L_*^{-m}$.
        \item \label{corollary: Distance between edges (3)} $1 \leq \diam(G_n,d_n) \leq 2 \cdot C_{\diam}$.
    \end{enumerate}
\end{corollary}

\begin{proof} 
    \ref{corollary: Distance between edges (1)} is clear if $\pi_{n,m}(e),\pi_{n,m}(f)$ share a common vertex.
    Otherwise, let $x\in e$ and $y\in f$ be such that $d_n(e,f)=d_n(x,y).$ If $\theta$ is a path connecting $x$ to $y$, then by Proposition \ref{prop: Path decomposition}, $\theta$ contains a sub-path from an ancestors $\widehat{x}$ and $\widehat{y}$ of some $x_e \in \pi_{n,m}(e)$ and $x_f \in \pi_{n,m}(f)$ respectively.
    Hence
    \[
        d_n(e,f) =d_n(x,y) \geq d_n(\widehat{x},\widehat{y}) \stackrel{\ref*{lemma: Distance of ancestors (1)}}{=} d_m(x_e,x_f) \geq d_m(\pi_{n,m}(e),\pi_{n,m}(f)).
    \]

    To prove \ref{corollary: Distance between edges (2)}, suppose $x \in \pi_{n,m}(e), y \in \pi_{n,m}(f)$ so that
    \[
        d_m(\pi_{n,m}(e),\pi_{n,m}(f)) = d_m(x,y).  
    \]
    If $x = y$, then the claim follows from \ref{lemma: Distance of ancestors (3)}.
    Otherwise, pick $\widehat{x},\widehat{y}$ ancestors of $x,y$ respectively.
    Then
    \begin{align*}
        d_n(e,f) & \leq d_n(e,\widehat{x}) + d_n(\widehat{x},\widehat{y}) + d_n(\widehat{y},f)\\
        & \stackrel{\ref*{lemma: Distance of ancestors (3)}}{\leq} 2C_{\diam} \cdot L_*^{-m} + d_n(\widehat{x},\widehat{y})\\
        & \stackrel{\ref*{lemma: Distance of ancestors (1)}}{=} 2C_{\diam} \cdot L_*^{-m} + d_m(x,y)\\
        & = 2C_{\diam} \cdot L_*^{-m} + d_m(\pi_{n,m}(e),\pi_{n,m}(f)).
    \end{align*}
    Lastly, we prove \ref{corollary: Distance between edges (3)}. The lower bound follows from \ref{lemma: Distance of ancestors (1)}. Let $x,y \in V_n$ and choose $e,f \in E_1$ so that $x \in e \cdot G_{n-1}$ and $y \in f \cdot G_{n-1}$. Then choose $\tilde{x} \in e, \tilde{y} \in f$ and a path $[z_1,\dots,z_k]$ from $\tilde{x}$ to $\tilde{y}$ in $G_n$ of length at most $C_{\diam}$. By choosing ancestors of $\widehat{z}_i \in V_n$ of $z_i$, we have
    \begin{align*}
        d_n(x,y) & \leq d_n(x,\widehat{z}_1) + d_n(\widehat{z}_k,y) + \sum_{i = 1}^{k-1} d_n(\widehat{z}_i,\widehat{z}_{i+1})\\
        & \stackrel{\ref*{lemma: Distance of ancestors (1)}}{=} d_n(x,\widehat{z}_1) + d_n(\widehat{z}_k,y) + \sum_{i = 1}^{k-1} d_1(z_i,z_{i+1})\\
        & \stackrel{\ref*{lemma: Distance of ancestors (3)}}{\leq} 2C_{\diam}L_*^{-1} + C_{\diam}\\
        & \leq 2C_{\diam}.
    \end{align*}
\end{proof}

\subsection{Doubling, regularity and limit space}

A metric space $(X,d)$ is said to be \emph{metrically doubling} if there exists a constant $N$ so that for each $x\in X$ and every $r>0$, there exists points $x_1, \dots, x_N$ so that $B(x,r)\subset \bigcup_{i=1}^N B(x_i, r/2)$. We say that a sequence $\{X_k\}_{k\in \N}$ is \emph{uniformly doubling}, if each space in the sequence if $N$-doubling for a fixed $N$. The following definition gives a simple necessary and sufficient condition for an IGS to produce a uniformly doubling sequence of graphs.

\begin{definition}
An iterated graph system is \emph{doubling}, if for every 
$e = \{ v,u \} \in E_1$ and $z \in I_{v,e}$ we have $\degr(z) = 1$.
For such $z \in I_{v,e}$ we set $\mathfrak{n}(z)$ to be the unique neighbour of $z$ and for the corresponding edge we use the notation $\mathfrak{e}_z$.
\end{definition}

\begin{lemma}\label{lemma: Doubling}
    If the iterated graph system is doubling, then $\deg(G_n) = C_{\deg}$ for all $n$.
    Furthermore, if the iterated graph system also satisfies the $L_*$-uniform scaling property, then the metric spaces $(G_n,d_n)$ are doubling with the doubling constant $N_D = N_D(C_{\deg},C_{\diam},L_*)$.
\end{lemma}

\begin{proof}
    It follows from \ref{SM1} that $\deg(G_n) \geq C_{\deg}$ for all $n \in \N$.
    We prove by induction, that $\deg(G_n) = C_{\deg}$.
    The base case $n = 1$ is trivial so we assume it holds for $n$.
    If $x \in V_{n + 1}$ and $x$ is not an ancestor of any $x_n \in V_n$, by \ref{SM1} and \ref{SM2}, $\deg(x) \leq \deg(G_1) = C_{\deg}$.
    On the other hand, if $x$ is an anscestor of $x_{n} \in V_n$, then, by the doubling property of the IGS, $\deg(x) = \deg(x_n) \leq C_{\deg}$.

    Next we prove the metric doubling property of $d_n$.
    Fix $x \in V_n$ and $r > 0$.
    First assume $r \leq 4 C_{\diam}\cdot L_*^{-n}$.
    Then every vertex $y \in B(x,r)$ can be reached from $x$ by a path of length at most $4 C_{\diam}$.
    Hence $\abs{B(x,r)} \leq (C_{\deg})^{4 C_{\diam}}$ so in this case we may set $N_D = (C_{\deg})^{4 C_{\diam}}$. On the other hand, if $r \geq 4C_{\diam}$, then $B(x,r/2) = V_n$ so we can choose $N_D = 1$. Hence we may assume that
    \[4 C_{\diam} \cdot L_*^{-m} < r \leq 4 C_{\diam} \cdot L_*^{-m + 1}\]
    for some $m = 1,\dots,n$.
    Fix $\widetilde{x} \in V_{n}$ so that it has an ancestor in $V_m$ and
    \[d_n(x,\widetilde{x}) \stackrel{\ref*{lemma: Distance of ancestors (3)}}{\leq} C_{\diam} \cdot L_*^{-m} < r/4.\]
    Similarly, for each $y \in B(x,r)$, we choose $\widehat{y} \in V_n$ having an ancestor in $V_m$ and
    \[d_n(y,\widehat{y}) \stackrel{\ref*{lemma: Distance of ancestors (3)}}{\leq} C_{\diam} \cdot L_*^{-m} < r/4.\]
    Then
    \[d_m(\pi_{n,m}(\widetilde{x}),\pi_{n,m}(\widehat{y})) \stackrel{\ref*{lemma: Distance of ancestors (1)}}{\leq} d_n(\widetilde{x},\widehat{y}) \leq r/2 + d_n(x,y) < 2r\]
    which gives
    \[S := \{\pi_{n,m}(\widehat{y}) \in V_m : y\in B(x,r) \} \subseteq B(\pi_{n,m}(\widetilde{x}),2r).\]
    Since $2r \leq 8C_{\diam}L_* \cdot L_*^{-m}$, $S$ contains at most $(C_{\deg})^{8C_{\diam}L_*}$ vertices.
    Lastly, for every $v \in S$ we choose $x_v \in \pi_{n+m}^{-1}(v)$.
    Then for any $y \in B(x,r)$, we have
    \[d_n(\widehat{y},x_{\pi_{n,m}({\widehat{y}})}) \stackrel{\ref*{lemma: Distance of ancestors (2)}}{\leq} C_{\diam} \cdot L_*^{-m} < r/4\]
    and we finally obtain
    \[B(x,r) \subseteq \bigcup_{y \in B(x,r)} B(\widehat{y},r/4) \subseteq \bigcup_{v \in S} B(x_v,r/2)\]
    so we may set $N_D = (C_{\deg})^{8C_{\diam}L_*}$.
\end{proof}

Next, we will show that there is a fractal space $X$ which can be obtained as a Gromov-Hausdorff limit of the sequence $(G_n,d_n)$. This limit space has a symbolic description, which we now give.

\begin{definition}
    Given an iterated graph system, we define 
    \[
      \Sigma := \{ (e_i)_{i = 1}^{\infty} : e_i \in E_i \text{ and } \pi_m(e_m) = e_{m - 1} \}.  
    \]
    For $e \in E_n$, we define the subsets $\Sigma_e \subseteq \Sigma$ by
    \[
        \Sigma_e := \left\{ (e_i)_{i = 1}^{\infty} \in \Sigma : e_n = e \right\}.
    \]
    We give $\Sigma$ the natural metric
    \[
        d_{\Sigma}((e_i)_{i = 1}^\infty,(f_i)_{i = 1}^\infty) :=
        \begin{cases}
            0 & \text{ if } e_i = f_i \text{ for all } i \in \N\\
            2^{-\max \left\{ k \ : \  (e_i)_{i = 1}^k = (f_i)_{i = 1}^k \right\}} & \text{ otherwise}
        \end{cases}
    \]
    and the Radon measure $\mu_{\Sigma}$, for which
    \[
        \mu_{\Sigma}(\Sigma_e) = \frac{1}{\abs{E_1}^n}
    \]
    for all $e \in E_n$.
\end{definition}

\begin{remark}
    The measure is well-defined by Kolmogorov's extension theorem and Proposition \ref{prop: Similarity maps}. There is also an alternative way to construct $\mu_\Sigma$. The space $\Sigma$ and can be identified as $(E_1)^{\N}$ by identifying $(e_i)_{i = 1}^\infty \in \Sigma$, with $\left(\hat{e}_i\right)_{i = 1}^\infty \in (E_1)^\N$, where
    \begin{enumerate}
        \item $e_1 = \hat{e}_1$
        \item $\hat{e}_{n + 1} = \{ z_1,z_2 \}$, where $e_{n + 1} = \{ [z_1,e_n], [z_2,e_n] \}$.
    \end{enumerate}
    Let $T: (E_1)^\N \to \Sigma$ be the map $T(\left(\hat{e}_i\right)_{i = 1}^\infty)=(e_i)_{i = 1}^\infty$, and let $\nu$ be the Bernoulli measure on $(E_1)^\N$ where each element in $E_1$ has equal measure $\abs{E_1}^{-1}$. Then the measure can also be constructed as the push-forward $\mu_{\Sigma}=(T)_*(\nu)$, where $T_*$ is the push-forward operation for measures.
\end{remark}

\begin{definition}
    The \emph{limit space} of an iterated graph system is
    \[
        X := \Sigma/\sim, \text{ where } (e_i)_{i = 1}^{\infty} \sim (f_i)_{i = 1}^{\infty} \iff e_i \cap f_i \neq \emptyset \text{ for all } i
    \]
    and is given the metric
    \begin{equation}\label{eq:dxdef}
        d_X([(e_i)_{i = 1}^{\infty}],[(f_i)_{i = 1}^{\infty}]) := \lim_{n \to \infty} d_n(e_n,f_n).
    \end{equation}
    The canonical projection $\Sigma \to X$ is denoted by $\chi$
    and the push-forward measure on $X$ is denoted as $\mu_X := \chi_*(\mu_\Sigma)$.
    We also denote $X_e := \chi(\Sigma_e)$.
    \end{definition}

The limit in \eqref{eq:dxdef} exists and it defines a metric by the following lemma.

\begin{lemma}
    The function $d_X$ is a well-defined metric on $X$ and $\chi : (\Sigma,d_\Sigma) \to (X,d_X)$ is a continuous function.
\end{lemma}

\begin{proof}
    Given any pair of sequences $(e_i)_{i = 1}^\infty,(f_i)_{i = 1}^\infty \in \Sigma$, by \ref{corollary: Distance between edges (1)} and \ref{corollary: Distance between edges (3)},
    the sequence $d_n(e_n,f_n)$ converges as it is increasing and bounded.
    Moreover, $d_n$ defines a semi-metric on $E_n$ so that $d_n(e,f) = 0$ if and only if $e$ and $f$ share a vertex. Therefore $d_X$ is a well-defined metric on $X$.
    The continuity of $\chi$ follows from \ref{corollary: Distance between edges (2)}.
\end{proof}

The following lemma is immediate from \ref{corollary: Distance between edges (1)} and \ref{corollary: Distance between edges (2)}.

\begin{lemma}\label{lemma: properties of d_X}
    Let $x = [(e_i)_{i = 1}^\infty], y = [(f_i)_{i = 1}^\infty]$ be points in $X$. Then the following hold.
    \begin{enumerate}[label=\textcolor{blue}{(DL\arabic*)}]\setcounter{enumi}{6}
        \item \label{lemma: properties of d_X (1)} $d_X(x,y) \geq d_n(e_n,f_n)$ for all $n \in \N$.
        \item \label{lemma: properties of d_X (2)} If $e_m$ and $f_m$ have a common vertex then
        \[
            d_X(x,y) \leq 2 C_{\diam} \cdot L_*^{-m}.
        \]
        In particular, $\diam(X_e) \leq 2 C_{\diam} \cdot L_*^{-m}$ for all $e \in E_m$.
    \end{enumerate}
\end{lemma}

In summary, we can now show that the spaces $(G_n,d_n)$ Gromov-Hausdorff converge to $X$. First, we give the relevant definitions. An $\epsilon$-isometry is an isometry (i.e. a map which preserves distances) between metric spaces $f:X\to Z$, for which $Z\subset \bigcup_{x\in X} B(\iota_X(x),\epsilon)$. 
The \emph{Gromov-Hausdorff metric} between $X,Y$ is defined as 
\begin{align*}
d_{GH}(X,Y):=\inf\{\epsilon>0 : & \exists \text{ a metric space } Z \\
& \text{ and  $\epsilon$-isometries } \iota_X:X\to Z, \iota_Y: Y\to Z\}.
\end{align*}
A metric space $X$ is said to be \emph{geodesic}, if for all $x,y\in X$ there exists a rectifiable curve $\gamma$ connecting $x$ to $y$ with $d(x,y)=\ell(\gamma)$, where $\ell(\gamma)$ is the length of $\gamma$. See \cite{bbi} for background on these notions. 

\begin{proposition}\label{prop: GH-convergence}
    If the iterated graph system satisfies the $L_*$-uniform scaling- and the doubling property, then the sequence of metric spaces $\{ (G_n,d_n) \}_{n \in \N}$  Gro\-mov-Hausdorff converges to $X$. Moreover, $X$ is compact, geodesic and satisfies the metric-doubling property.
\end{proposition}

\begin{proof}
    Consider the metric spaces $(G_n,d_n)$ together with their geometric realizations, which are obtained by gluing edges of length $L_*^{-n}$ between adjacent vertices. Notice that the Gromov-Hausdorff distance of $(G_n,d_n)$ to its geometric realization is $L_*^{-n}$, since we can choose $Z$ as the geometric realization and use the inclusion maps in the definition of the Gromov-Hausdorff metric. Thus, if the metric space $(G_n,d_n)$ converge to $X$, so do the geometric realizations. On the other hand, the geometric realizations are compact, geodesic and uniformly doubling (by Lemma \ref{lemma: Doubling}), so it follows that if $\{ (G_n,d_n) \}_{n \in \N}$ Gromov-Hausdorff converges to $X$, then $X$ is compact, geodesic, and satisfies metric doubling property; see e.g. \cite[Theorem 7.5.1]{bbi}.
    
    We will construct $\epsilon(n)$-isometries $i_{G_n} : V_n \to X$ so that $\epsilon(n) \to 0$ as $n \to \infty$.
    First for every $x \in V_m$ we choose ancestors $\widehat{x} \in V_m$ of $x$ and define the map $i_{m,n} : V_m \to V_n, i_{m,n}(x) = \widehat{x}$.
    By \ref{lemma: Distance of ancestors (1)} $i_{m,n}$ is an isometry and,
    moreover, we can choose the ancestors so that for all $m>n>l$ we have
    \begin{equation}\label{eq: assumptions on embedding}
       i_{n,l} \circ i_{m,n} = i_{m,l} \text{ and } \pi_{l,n} \circ i_{m,l} = i_{m,n}.
    \end{equation}
    Next we construct maps $i_{m,n}^E : V_m \to E_n$ for all $m,n \in \N$.
    For $n \leq m$ we define $i_{n,m}^E : V_n \to E_m$ by choosing an edge $e_x$ for every $x \in V_n$ containing $i_{n,m}(x)$.
    For $n > m$ we define $i_{n,m}^E : V_n \to E_m$ by 
    \[
        i_{n,m}^E(x) =
        \begin{cases}
            \pi_{n,m}(x) & \text{ if } \pi_{n,m}(x) \in E_m  \\
            e_x & \text{ if  } \pi_{n,m}(x) \in V_m,
        \end{cases}
    \]
    where $e_x \in E_m$ is an edge containing $\pi_{n,m}(x) \in V_m$.
    By \eqref{eq: assumptions on embedding} we may choose the edges $e_x$ so that
    \begin{equation}\label{eq: projections and i^E}
        \pi_{n,k} \circ i_{m,n}^E = i_{m,k}^E.
    \end{equation}
    
    We now define the mappings $i_{G_n} : V_n \to X$ by $i(x) = [(i_{n,i}(x))_{i = 1}^{\infty}]$.
    Indeed, by \eqref{eq: projections and i^E}, $(i_{n,i}(x))_{i = 1}^{\infty} \in \Sigma$ for all $x \in V_n$ so $i_{G_n}$ is well-defined.
    We will prove that $i_{G_n}$ is a $2C_{\diam} \cdot L_*^{-n}$-isometry.
    Choose distinct vertices $x,y \in V_n$.
    Since the edges $i_{n,m}^E(x), i_{n,m}^E(y) \in E_m$ contain the ancestors $i_{n,m}(x), i_{n,m}(y) \in V_m$ of $x,y$ respectively, we have by \ref{lemma: Distance of ancestors (3)} 
    \begin{align*}
    d_m(i_{n,m}(x), i_{n,m}(y)) - 2 C_\diam \cdot L_*^{-m} & \leq d_m(i_{n,m}^E(x), i_{n,m}^E(y))\\
    & \leq d_m(i_{n,m}(x), i_{n,m}(y))\\
    & = d_n(x,y).
    \end{align*}
By letting $m \to \infty$, we have
\[
    d_X(i_{G_n}(x),i_{G_n}(y)) = \lim_{m \to \infty} d_m(i_{n,m}(x),i_{n,m}(y)) = d_n(x,y).
\]
Lastly, if $x = [(e_i)_{i = 1}^\infty] \in X$ and $x_n \in e_n$, then
\[
    d_X(x,i_{G_n}(x_n)) \stackrel{(\ref{lemma: properties of d_X (2)})}{\leq} 2 C_{\diam}\cdot L_*^{-n}. 
\]
Hence $i_{G_n}$ is a $2C_{\diam} \cdot L_*^{-n}$-isometry. This gives $d_{GH}((G_n,d_n),X)\leq 2C_{\diam} \cdot L_*^{-n}$ and thus $X$ is the Gromov-Hausdorff limit of the spaces $G_n$. 
\end{proof}

We can also compute that $X$ is Ahlfors regular, and identify the Hausdorff dimension of $X$. 

\begin{lemma}\label{lem:ahlforsregularity} 
If the iterated graph system satisfies the $L_*$-uniform scaling- and doubling properties, then the metric measure space $(X,d_X,\mu_X)$
is $Q$-Ahlfors regular with 
\[
    Q:=\frac{\log(\abs{E_1})}{\log(L_*)}.
\]
\end{lemma}

\begin{proof}
    Fix $x = [(e_i)_{i = 1}^\infty] \in X$ and $r \in (0,\diam(X)]$. Then
    \[r\leq \diam(X)\stackrel{\ref{corollary: Distance between edges (3)}}{\leq} 2C_{\diam}\] so that for some $n\in \N$ we have $2C_{\diam} \cdot L_*^{-n} < r \leq 2C_{\diam} \cdot L_*^{-(n - 1)}$.
    Now if $y = [(f_i)_{i = 1}^\infty]$ so that $f_n = e_n$,
    then, by \ref{lemma: properties of d_X (2)}, $d_X(x,y) \leq 2C_{\diam}\cdot L_*^{-n}$.
    Hence $X_{e_n} \subseteq B(x,r)$
    which yields
    \[
        \mu_X(B(x,r)) \geq \mu_X(X_{e_n})\geq \mu_{\Sigma}(\Sigma_{e_n}) = \frac{1}{\abs{E_1}^n} \geq (2 C_{\diam}L_*)^{-Q} \cdot r^Q. 
    \]
    
    Next we prove the upper bound. Let $f_n^{(1)},\dots,f_n^{(k)} \in E_n$ be the edges $f \in E_n$ so that $\chi^{-1}(B(x,r)) \cap \Sigma_f \neq \emptyset$.
    Choose $y = [(f_i)_{i = 1}^\infty] \in B(x,r)$. Then
    \[
        d_n(e_n,f_n) \stackrel{\ref{lemma: properties of d_X (1)}}{\leq} d_X(x,y) < r \leq 2C_{\diam} \cdot L_*^{-(n - 1)} = 2C_{\diam}L_* \cdot L_*^{-n}.
    \]
    Hence $f_n$ can be reached from $e_n$ by a path of length at most $2C_{\diam}L_*$.
    This yields $k \leq 2(C_{\deg})^{2C_{\diam}L_*}$ and further
    \begin{align*}
        \mu_X(B(x,r)) & \leq \mu_{\Sigma}\left(\bigcup_{i = 1}^k \Sigma_{f_n^{(i)}} \right)\\
        & \leq (C_{\deg})^{2C_{\diam}L_*} \cdot \frac{1}{\abs{E}^n}\\
        & \leq (C_{\deg})^{2C_{\diam}L_*} (2C_{\diam})^Q \cdot r^Q.
    \end{align*}
\end{proof}

\subsection{Approximate Self-similarity} \label{subsec:selfsimilarity}

Our spaces are constructed by a self-similar procedure. Thus, beyond a minor combinatorial issue, it is not too difficult to obtain approximate self-similarity; recall Definition \ref{def:selfsim}. Our proof utilizes a condition similar to that of the finite type condition of an iterated function system. This condition is also called a finite intersection condition.

\begin{definition}
    Suppose the iterated graph system satisfies the $L_*$-uniform scaling- and doubling properties.
    The \emph{fundamental neighbourhood} of $e \in E_n$ is the induced sub-graph $\cN(e) = (V(e),E(e)) \subseteq G_n$ so that
    \[
        V(e) = \left\{x \in V_n : d_{G_n}(x,e) \leq 2 \right\}.
    \]
    Notice that in above the distance is the un-scaled path metric.
    Furthermore, we define the  \emph{extended fundamental neighbourhood} of $e \in E_n$ to be the induced sub-graph $\widehat{\cN}(e) = \left(\widehat{V}(e),\widehat{E}(e)\right) \subseteq G_n$ so that
    \[
        \widehat{V}(e) = \left\{x \in V_n : d_{G_n}(x,e) \leq 10C_{\diam} \right\}.
    \]
    We say that the fundamental neighbourhoods $\cN(e), \cN(f)$ of $e = \{ x_1,y_1 \} \in E_n$ and $f = \{ x_2,y_2 \} \in E_m$ are \emph{equivalent} if there is a graph isomorphism $\beta : \widehat{\cN}(e) \to \widehat{\cN}(f)$ between the extended fundamental neighbourhoods satisfying
    \begin{enumerate}
        \item $\{ x_2,y_2 \} = \{ \beta(x_1),\beta(y_1) \}$ and
        \item $\phi_{x,\{ x,y \}} = \phi_{\beta(x),\{ \beta(x), \beta(y) \}}$ for all $\{ x,y \} \in \widehat{E}(e)$.
    \end{enumerate}
    We also say that $\beta$ is an isomorphism between the fundamental neighbourhoods $\cN(e)$ and $\cN(f)$.
    The set of equivalence classes of equivalent fundamental neighbourhoods is denoted as $\mathcal{N}$.
    \end{definition}
    The idea behind equivalent fundamental neighborhoods is their "similarity" on all scales. Indeed, if $e \in V_n$, we define the sub-graphs $\mathcal{N}_{k}(e) = (V_{k}(e),E_{k}(e))$ and $\widehat{\mathcal{N}}_{k}(e) = \left(\widehat{V}_{k}(e),\widehat{E}_{k}(e)\right)$ of $G_{n + k}$ where
    \[
        E_{k}(e) = \pi_{n + k,n}^{-1}(E(e)) \text{ and } V_{n}(e) = \bigcup_{f \in V(e)} f \cdot G_k
    \]
    and similarly
    \[
        \widehat{E}_{k}(e) = \pi_{n + k,n}^{-1}\left(\widehat{E}(e)\right) \text{ and } \widehat{V}_{k}(e) = \bigcup_{f \in \widehat{V}(e)} f \cdot G_k.
    \]
    Now given an isomorphism $\beta : \widehat{\cN}(e) \to \widehat{\cN}(f)$ between the fundamental neighbourhoods of $e$ and $f$, we define the mappings $\beta_k : \widehat{\cN}_k(e) \to \widehat{\cN}_k(f)$ inductively:
    \begin{enumerate}
        \item $\beta_0 = \beta$.
        \item If $\beta_k$ is defined then $\beta_{k + 1}([z,\{ x,y \}]) = [z,\{\beta_k(x),\beta_k(y)\}]$.
    \end{enumerate}
    The corresponding map between edges $\widehat{E}_k(e) \to \widehat{E}_k(f)$ is denoted as $\beta_k^E$.

    In the following lemma, recall that $d_{G_n}$ was the un-scaled path metric on $G_n$, and $d_n=L_*^{-n} d_{G_n}$.

\begin{lemma}\label{lemma: Lifting on fundamental nbhds}
    If $\beta : \widehat{\cN}(e) \to \widehat{\cN}(f)$ is an isomorphism between the fundamental neighbourhoods of $e \in E_n$ and $f \in E_m$ then the mappings $\beta_k$ are graph isomorphisms. These maps satisfy
    \[
       \pi_{m+k + 1} \circ \beta_{k+1}^E = \beta_k^E \circ \pi_{n+k + 1}
    \]
    for all $k \geq 0$. Furthermore, if the IGS satisfies the $L_*$-uniform scaling property and the doubling property, then the restriction $\beta_k : \cN_k(e) \to \cN_k(f)$ are isometries with respect to the (un-scaled) path metrics $d_{G_{n+k}}$ and $d_{G_{m + k}}$ respectively.
\end{lemma}

\begin{proof}
    As in the proof of Proposition \ref{prop: Similarity maps}, it is a straightforward inductive argument that $\phi_{x,\{ x,y \}} = \phi_{\beta_k(x), \{ \beta_k(x),\beta_k(y) \}}$ for all $k \in \N$ and $\{ x,y \} \in E_k(e)$. This observation yields that $\beta_k$ is a graph isomorphism. Moreover, the commutativity with the projections follows from 
    \begin{align*}
        \pi_{m + k + 1}(\beta_{k + 1}^E(\{ [z_1,e],[z_2,e] \})) & = \pi_{m + k + 1}(\{ [z_1,\beta_{k}^E(e)],[z_2,\beta_{k}^E(e)] \})\\
        & = \beta_k^E(e)\\
        & = \beta_k^E(\pi_{n + k + 1}(\{ [z_1,e],[z_2,e] \})).
    \end{align*}
    
    To prove that its restritcions $\beta_k : \cN_k(e) \to \cN_k(f)$ are isometries, we will show that a shortest path in $G_{n + k}$ between vertices in $\cN_k(e)$ is contained in $\widehat{\cN}_k(e)$. By applying the same argument to the vertices in $\cN_k(f)$, we would be done.
    
    First, if $e' \in E(e)$, then $d_{G_n}(e,e') \leq 2$, which gives $\diam(\mathcal{N}_k(e),d_{G_{n+k}}) \leq 8 C_{\diam} L_*^{k}$ by \ref{corollary: Distance between edges (2)} and the fact that $d_{n+k}=L_*^{-(n+k)} d_{G_{n+k}}$.
    Next, suppose $\theta$ is a path in $G_{n + k}$ containing a vertex $x \in \mathcal{N}_k(e)$ with $\diam(\theta) \leq 8 C_{\diam} \cdot L_*^{k}$. We will show that indeed $\theta$ is contained in $\widehat{\cN}(e)$. Suppose not. Then $\theta$ contains a vertex in $e' \cdot G_k$ where $d_{G_n}(e,e') \geq 10C_{\diam}$. This, by Proposition \ref{prop: Path decomposition}, gives that there are $y \in e, y' \in e'$ and a sub-path $\theta'$ from $\pi_{n+k,n}^{-1}(y)$ to $\pi_{n + k,n}^{-1}(y')$. Hence
    \[
        \len(\theta) \geq \len(\theta') \stackrel{\ref*{lemma: Distance of ancestors (1)}}{\geq} L_{*}^{n + k} \cdot d_n(y,y') = L_*^{k} \cdot d_{G_n}(y,y') \geq 10C_{\diam} \cdot L_*^k > \diam(\theta).
    \]
    This yields the contradiction so we are done.
\end{proof}

The previous lemma has the following immediate corollary. A map $h:(X,d_X)\to (Y,d_Y)$ between metric spaces is called a homothety if there exists a constant $s$ so that $d_Y(h(x),h(y))=s\cdot d_X(x,y)$ for all $x,y\in X$.

\begin{corollary}\label{corollary: alpha_infty}
    Let $\beta : \widehat{\cN}(e) \to \widehat{\cN}(f)$ be an isomorphism between the fundamental neighbourhoods of $e \in E_n$ and $f \in E_m$.
    Then the mapping
    \[
        \beta_{\infty} : \bigcup_{e' \in \cN(e)} X_{e'} \to \bigcup_{f' \in \cN(f)} X_{f'}, \quad [(e_i)_{i = 1}^{\infty}] \mapsto [(f_i)_{i = 1}^{\infty}],
    \]
     where
    \[
        f_i = 
        \begin{cases}
            \beta_{k}^E(e_{n + k}) & \text{ if } i = m + k \text{ for } k \geq 0,\\
            \pi_{m,i}(\beta(e_n)) & \text{ if } i < m
        \end{cases}
    \]
    is a homothety with the constant $L_*^{n - m}$.
\end{corollary}

\begin{proof}
    Let $\left(e_i\right)_{i = 1}^\infty, \left(e_i'\right)_{i = 1}^\infty \in \Sigma$ so that $e_n,e_n' \in \cN(e)$.
    By Lemma \ref{lemma: Lifting on fundamental nbhds} we have
    \[
        d_{n + k}(e_{n+k},e_{n+k}')\cdot L_*^{n+k} = d_{m + k}(\beta_k^E(e_{n+k}),\beta(e_{n+k}'))\cdot L_*^{m+k}
    \]
    which proves that $\beta_\infty$ is a homothety.
\end{proof}

\begin{proposition}
    If the iterated graph system satisfies the $L_*$-uniform scaling- and doubling properties, then the limit space $X$ is approximately self-similar.
\end{proposition}

\begin{proof}
    It follows from the doubling property of the IGS that the set $\cN$ of equivalence classes of equivalent fundamental neighbourhoods is finite.
    In particular, there is $m_* \in \N$ so that for all $n \in \N$ and $e \in E_n$, the fundamental neighbourhood $\cN(e)$ is equivalent to $\cN(f)$ for some $f \in E_m$ so that $m \leq m_*$.

    Let $x = [(e_i)_{i = 1}^\infty] \in X, r > 0$ and set $B = B(x,r)$. If $r \geq L_*^{-(m_* + 1)}$ then we choose the identity $\id_{B}:  B(x,r) \to B(x,r)$ as our scaling map. Otherwise, suppose $L_*^{-(n + 1)} < r \leq L_*^{-n}$ where $n > m_*$. Let $f \in E_m$ so that $m \leq m_*$ and $\beta : \cN(e_n) \to \cN(f)$ an isomorphism between fundamental neighbourhoods of $e_n$ and $f$. If $\beta_{\infty}$ is as in Corollary \ref{corollary: alpha_infty} then, by \ref{lemma: properties of d_X (1)}-\ref{lemma: properties of d_X (2)}, we have $B(x,r) \subseteq \bigcup_{e' \in \cN(e_n)} X_{e'}$ and  $\beta_{\infty}(B(x,r)) = B(\beta_\infty(x),L_*^{n - m} \cdot r)$. Furthermore, $\beta_\infty$ is a homothety with constant $s = L_*^{n - m}$, which is comparable to $r^{-1}$ by
    \[
    L_*^{-(m_*+1)} \cdot r^{-1} \leq s \leq r^{-1}.
    \]
    Hence approximate self-similarity follows and $\beta_\infty$ is the desired scaling map.
\end{proof}

\begin{definition}
    The \emph{approximation} of the limit space $X$ \emph{at scale} $m \in \N$ is the incidence graph $\mathbb{G}_m = (\{ X_v \}_{v \in V_m}, \mathbb{E}_m)$, where the covering of $X$ consists of the open subsets
    \[
       X_v := \intr \{ [(e_i)_{i = 1}^\infty] : e_m \text{ contains } v \}
    \]
    for all $v \in V_m$.
\end{definition}

\begin{proposition}\label{prop: Graph approximations}
    For all $m \in \N$ and $v \in V_m$ there are points $z_v \in X_v$ so that
    \[
        B(z_v, L_*^{-m}) \subseteq X_v \subseteq B(z_v, C_{\diam} \cdot L_*^{-m})
    \]
    and $d_X(z_v,z_u) = d_m(v,u)$ for all $v,u \in V_m$. In particular, $\mathbb{G}_m$ is an $\alpha$-approxi\-mation of $X$ at scale $m$ for $\alpha = C_{\diam}$. Moreover, for distinct $v,u \in V_m$, we have $X_v \cap X_u \neq \emptyset$ if and only if $\{ v,u \} \in E_m$.
\end{proposition}

\begin{proof}
    Recall the functions $i_{G_m}$ defined in Proposition \ref{prop: GH-convergence}. We define $z_v = i_{G_m}(v)$. It follows from \ref{lemma: properties of d_X (1)} and \ref{lemma: Distance of ancestors (1)} that $B(z_v, L_*^{-m}) \subseteq X_v$ and $d_X(z_v,z_u) = d_m(v,u)$ for all $v,u \in V_n$. The other inclusion follows from \ref{lemma: Distance of ancestors (3)} and the last assertion from \ref{lemma: properties of d_X (1)}.
\end{proof}

As the graphs $\mathbb{G}_m$ and $G_m$ are naturally isomorphic, we hereafter identify the replacement graphs $G_m$ with the incidence graphs $\mathbb{G}_m$.

\begin{remark}\label{rmk:topdim}
    The collection $\{ X_v \}_{v \in V_m}$ is an open cover of $X$ with sets with diameter less than $2C_\diam L_*^{-m}$ (by \ref{lemma: properties of d_X (2)}). The nerve complex is given by the graph $G_m$, and thus $X$ has topological dimension $1$, \cite[p. 126]{topdim}. Indeed, by a slight modification of $X_v$, it is direct to see that $X$ has Assouad-Nagata dimension equal to $1$; see e.g. \cite[Definition 1.1.]{LS}.
\end{remark}

\section{Modulus on graphs}\label{sec:modulus}

\subsection{Edge- and vertex modulus}

Ultimately, we will be interested in establishing the combinatorial Loewner property for the limit spaces $X$. We already defined in Subsection \ref{subsec:moduli} a notion combinatorial moduli for $X$ with respect to an $\alpha$-approximations. In Proposition \ref{prop: Graph approximations} we saw that $G_m$ can be identified with such an $\alpha$-approximation. For computational reasons, it will be useful to consider two additional notions of modulus on $G_m$: the edge modulus and the vertex modulus. We will shortly see that these notions of moduli are comparable. 
\begin{definition}\label{def:edgemod}
Let $p\in [1,\infty)$ and $G = (V,E)$ a graph. If $\theta = [v_1,\dots,v_n]$ is a non-constant path in $G$ and $\rho : E \to \R_{\geq 0}$, we write
\[
    L_\rho(\theta) := \sum_{i = 1}^{n - 1} \rho(\{ v_i,v_{i + 1} \}) \text{ and } \cM_p(\rho) := \sum_{e \in E} \rho(e)^p.
\]
Given a family of paths $\Theta$ in $G$, we say that $\rho$ is \emph{$\Theta$-admissible} if $L_\rho(\theta) \geq 1$ for all $\theta \in \Theta$. Furthermore, we denote the \emph{(edge) modulus} of $\Theta$ by
\begin{align*}
\Mod_p(\Theta, G) := \inf_{\rho} \cM_p(\rho)
\end{align*}
where the infimum is over all $\Theta$-admissible densities $\rho : E \to \R_{\geq 0}$.
For any density $\rho : E \to \R_{\geq 0}$ the \emph{support} of $\rho$ is $\supp(\rho) := \{ e \in E : \rho(e) \neq 0 \}$.
\end{definition}

In the following definition, we say that a subset $\theta\subset G$ is connected, if its induced sub-graph is connected. Most of the time we will perform calculations for the edge modulus. However, in relating the estimates to the modulus of Bourdon and Kleiner, we need to momentarily also discuss vertex modulus.

\begin{definition}\label{def:vertexmod}
    Let $p\in [1,\infty)$ and $G = (V,E)$ a graph. If $\theta \subseteq V$ is a non-empty connected set in $G$ and $\rho : V \to \R_{\geq 0}$, we write
\[
    L_\rho(\theta) := \sum_{v \in \theta} \rho(v) \text{ and } \cM_p(\rho) := \sum_{v \in V} \rho(v)^p.
\]
Given a family of paths $\Theta$ in $G$, we say that $\rho$ is \emph{$\Theta$-admissible} if $L_\rho(\theta) \geq 1$ for all $\theta \in \Theta$. Furthermore, we denote the \emph{(vertex) modulus} of $\Theta$ by
\begin{align*}
\Mod_p^V(\Theta, G) := \inf_{\rho} \cM_p(\rho)
\end{align*}
where the infimum is over all $\Theta$-admissible densities $\rho : V \to \R_{\geq 0}$.
\end{definition}

Given a density $\rho^E : E \to \R_{\geq 0}$, one can easily define a density $\rho^V : \to \R_{\geq 0}$ by considering the density of the edges adjacent to a given vertex and vice versa. This standard argument gives the comparability of edge- and vertex modulus. Given a path $\theta=[v_1,\dots, v_n]$ in $G$, we call $\{v_1,\dots, v_n\}$ the corresponding set in $V$. For simple paths $\theta$ we use $\theta$ also to denote the corresponding set.

\begin{lemma}\label{lemma: Edge- and Vertex modulus}
    Let $G = (V,E)$ be a graph and $\Theta$ a non-empty family of non-constant paths so that each path in $\Theta$ contains a simple sub-path in $\Theta$, and let $\hat{\Theta}$ be the corresponding subsets for each path. Then
    \begin{equation}\label{eq: Edge- and Vertex modulus}
        C^{-1}\cdot \Mod_p(\Theta, G) \leq \Mod_p^V(\hat{\Theta}, G) \leq C^{-1} \cdot \Mod_p(\Theta, G)
    \end{equation}
    where $C = C(p,\deg(G))$. In particular, \eqref{eq: Edge- and Vertex modulus} holds if $\Theta = \Theta(A,B)$ for non-empty disjoint sets $A,B \subseteq V$.
\end{lemma}

\begin{proof}
    Let $\rho^E : E \to \R_{\geq 0}$ be $\Theta$-admissible.
    Define $\widehat{\rho}^V: V \to \R_{\geq 0}$ by
    \[
        \widehat{\rho}^V(v) = \max_{\substack{e \in E
        \\ v \in e}} \rho^E(e).
    \]
    Let $\theta \in \hat{\Theta}$ and let $\theta' = \left[v_1,\dots, v_n \right] \in \Theta$ be a simple sub-path of $\theta$.
    Then
    \[
        \sum_{v \in \theta} \widehat{\rho}^V(v) \geq \sum_{v \in \theta'}\widehat{\rho}^V(v) = \sum_{i = 1}^n \widehat{\rho}^V(v_i) \geq
        \sum_{i = 1}^{n-1} \rho^E( \{ v_i,v_{i+1} \} ) \geq 1
    \]
    and
    \[
        \sum_{v \in V} \widehat{\rho}^V(v)^p \leq \sum_{v \in V} \sum_{\substack{e \in E \\ v \in e}} \rho^E(e)^p \leq 2 \cdot \sum_{e \in E} \rho^E(e)^p.
    \]
    Then let $\rho^V : V \to \R_{\geq 0}$ be $\hat{\Theta}$-admissible.
    Define
    \[
        \widehat{\rho}^E(\{ v,u\}) = 2 \cdot \max\{ \rho^V(v),\rho^V(u) \}. 
    \]
    Let $\theta \in \Theta$, and let $\hat{\theta}$ be the corresponding subset.
    Then
    \[
        \sum_{i = 1}^{n-1} \widehat{\rho}^E(\{v_i, v_{i+1}\}) \geq \sum_{i = 1}^n \rho^V(v_i) \geq \sum_{v \in \hat{\theta}} \rho^V(v) \geq 1 
    \]
    and
    \[
     \sum_{\{ v,u \} \in E} \widehat{\rho}^E(\{ v,u \})^p \leq 2^p \cdot \sum_{\{ v,u \} \in E} \left(\rho^V(v)^p + \rho^V(u)^p\right)
        \leq 2^p \deg(G) \cdot \sum_{x \in V} \rho^V(x)^p.
    \]
\end{proof}

Roughly dual to the notion of edge modulus is the notion of a flow. We will use the exponent $p\in[1,\infty)$ for modulus, and the exponent $q\in[1,\infty)$ for flows. Here, $q$ is called \emph{the dual exponent} of $p$ if $\frac{1}{p}+\frac{1}{q}=1$.  In all that follows, $q$ will always denote this dual exponent. \emph{A flow}  from $A\subset V$ to $B\subset V$ is a map $\cF:V\times V\to \R$ for which 
\begin{enumerate}
\item $\cF(x,y)=-\cF(y,x)$ and $\cF(x,y)=0$ unless $\{x,y\}\in E$;
\item $\sum_{\{x,y\}\in E}\cF(x,y)=0$ unless $x\in A\cup B$.
\end{enumerate}
If $\cF$  is a flow from $A$ to $B$ we also write that $\cF$ is a flow $A\to B$.
The \emph{support} of a flow $\cF$ is $\supp(\cF) := \{ \{ x,y \} \in E : \cF(x,y) \neq 0 \}$.
The \emph{total flow} is given by
\[
I(\cF):=\sum_{x\in A}\sum_{\{x,y\}\in E} \cF(x,y).
\]
A \emph{unit flow} from $\cF$ is one for which $I(\cF)=1$.
We also write the amount of flow from $x$ as
\[
\divr \cF(x) :=\sum_{\{x,y\}\in E} \cF(x,y).
\]
If $e=\{x,y\}\in E$, we also write $|\cF(e)|=|\cF(x,y)|$, since this does not depend on the orientation of the edge.
Define the \emph{energy} of the flow by 
\[
\mathcal{E}_q(\cF):= \sum_{e\in E} |\cF(e)|^q.
\]
\begin{definition} Let $q\in [1,\infty)$. If $A,B\subset V$, then let 
\[
\mathcal{E}_q(A,B,G) := \inf\{\mathcal{E}_q(\cF) : \cF \text{ is a unit flow from } A \text{ to } B\}.
\]
\end{definition}

Both optimal flows and optimal admissible functions are unique. We will use this lemma frequently without additional mention.
\begin{lemma} Let $A,B\subseteq G$ be disjoint and non-empty. If $p>1$, and $q$ is the dual exponent of $p$, then there is a unique $\mathcal{E}_q$-minimizing flow $\cF$ for $\mathcal{E}_q(A,B,G)$. Further for all collections $\Theta$ of paths (or connected subsets) in Definition \ref{def:edgemod} (or in \ref{def:vertexmod}) there is a unique optimal admissible function $\rho$.
\end{lemma}
\begin{proof}
    The proof follows immediately from the fact that the optimization problem is finite dimensional and strictly uniformly convex. 
\end{proof}

The following proposition for edge modulus is fairly well known, and was first shown in \cite[Theorem 5.1]{NakamuraYamasakiDuality}. See also \cite{ShimizuParabolicIndex} for another presentation of this argument, and \cite{Kwapisz} for a different approach. More recently, a variant involving probability measures on curves was shown in  \cite[Theorem 2.8]{ACFPC}. Such probability  measures can be identified with unit flows via superposition and by using the flow decomposition theorem.

\begin{proposition}\label{prop:duality} Let $A,B \subseteq V$ be non-empty disjoint subsets. Then for every $p\in (1,\infty)$ and its dual exponent $q\in (1,\infty)$
\[
\Mod_p(\Theta(A,B),G)^\frac{1}{p} \cdot \mathcal{E}_q(A,B,G)^{\frac{1}{q}}=1.
\]
Moreover, if $\mathcal{F}$ is a unit flow from $A$ to $B$ and $\rho$ is $\Theta(A,B)$-admissible so that
\begin{equation}\label{eq: Duality}
    \mathcal{M}_p(\rho)^{\frac{1}{p}} \cdot \mathcal{E}_q(\mathcal{F})^{\frac{1}{q}} = 1,
\end{equation}
then
\[
    \Mod_p(\Theta(A,B),G) = \mathcal{M}_p(\rho) \text{ and } \mathcal{E}_q(A,B,G) = \mathcal{E}_q(\mathcal{F}).
\]
\end{proposition} 

We introduce one more useful quantity that is equivalent to the edge modulus: the capacity.

\begin{definition}
    Let $p > 1, G = (V,E)$ be a graph and $\emptyset \neq A,B \subseteq V$ be disjoint subsets of $V$.
    The \emph{$p$-capacity} from $A$ to $B$ is
    \[
        \cCap_p(A,B,G) := \inf \left\{ \sum_{\{ x,y \} \in E} \abs{U(x) - U(y)}^p : U_{\mid A} = 0 \text{ and } U_{\mid B} = 1 \right\}
    \]
\end{definition}

We now state a few classical definitions and lemmas. The minimizers for capacity exist and are $p$-harmonic. \begin{definition}
    Let $p > 1, G = (V,E)$ be a graph and $\emptyset \neq V' \subseteq V$.
    For $p > 1$ we say that $U : V \to \R$ is \emph{$p$-harmonic in $V'$} if for every $x \in V'$
    \[
        \sum_{\{ x,y \} \in E} \sgn(U(x) - U(y))\abs{U(x) - U(y)}^{p - 1} = 0.
    \]
\end{definition}

\begin{lemma}\label{lemma: Energy minimizing and p-harmonic}
    Let $p > 1, G = (V,E)$ be a connected graph and $\emptyset \neq A,B \subseteq V$ be disjoint subsets of $V$.
    There is a unique function $U : V \to \R$ so that $U|_A = 0, U|_B = 1$ and
    \[
        \cCap_p(A,B,G) =  \sum_{\{ x,y \} \in E} \abs{U(x) - U(y)}^p.
    \]
    Furthermore, this function is $p$-harmonic in $V \setminus (A \cup B)$.
\end{lemma}

\begin{proof}
    The proof follows from uniform convexity and the variational principle, see e.g. \cite[Theorems 3.5 and 3.11]{Holopainen1997} for details.
\end{proof}

\begin{lemma}\label{lemma: Strong max principle}
    Let $p > 1, G = (V,E)$ be a finite graph and $\emptyset \neq A \subsetneq V$ a connected subset. If $U : V \to \R$ is $p$-harmonic in $A$ then
    \[
        \max_{x \in \overline{A}} U(x) = \max_{x \in \partial{A}} U(x).
    \]
    where
    \[
        \partial{A} = \{ x \in V \setminus A : \{ x,y \} \in E \text{ for some } y \in A \} \text{ and } \overline{A} = A \cup \partial A.
    \]
    Moreover, if $U(y) = \max_{x \in \overline{A}} U(x)$ for some $y \in A$ then $U$ is constant in $\overline{A}$.
\end{lemma}

\begin{proof}
    The definition of $p$-harmonic functions yields by a standard argument the maximum principle, for details see \cite[Theorem 3.14]{Holopainen1997}.
\end{proof}

\begin{lemma}\label{lemma: Mod-Cap}
    Let $p > 1, G = (V,E)$ be a finite connected graph and $\emptyset \neq A,B \subseteq V$ be disjoint subsets of $V$.
    Then
    \begin{equation}\label{eq: Mod-Cap}
        \Mod_p(A,B,G) = \cCap_p(A,B,G).
    \end{equation}
    Moreover, if $\widehat{\rho} : E \to \R_{\geq 0}$ is the optimal $\Theta(A,B)$-admissible density, $\mathcal{F} : V\times V \to \R$ is the optimal flow from $A$ to $B$ with respect to $\cE_q$ and $\widehat{U} : V \to \R$ so that 
    $\widehat{U}|_A = 0, \widehat{U}|_B = 1$ and
    \[
        \cCap_p(A,B,G) =  \sum_{\{ x,y \} \in E} \abs{\widehat{U}(x) - \widehat{U}(y)}^p.
    \]
    then 
    \[
    \widehat{\rho}(\{ x,y \}) = \abs{\widehat{U}(x) - \widehat{U}(y)}
    \]
    and
    \[
    \cF(x,y)=\Mod_p(A,B,G)^{-1}\sgn(\widehat{U}(y) - \widehat{U}(x)) \abs{\widehat{U}(y) - \widehat{U}(x)}^{p-1}
    \]
    for all $\{ x,y \} \in E$.
\end{lemma}

\begin{proof}
    Suppose $\rho : E \to \R$ is $\Theta(A,B)$-admissible and we define $U : V \to \R$ by
    \[
        U(x) := \min_{\gamma} \sum_{i = 1}^{n - 1} \rho(\{ x_i,x_{i+1} \})
    \]
    where the minimum is over all paths $\gamma = [x_1,\dots,x_n]$ so that $x_1 = x$ and $x_n \in A$.
    For $U' = U \land 1$, we have $U'|_A = 0, U'|_B = 1$ and
    \[
        \abs{U'(x) - U'(y)} \leq \abs{U(x) - U(y)} \leq \rho(\{ x,y \})
    \]
    for all $\{x,y\} \in E$. This proves
    \[
        \cCap_p(A,B,G) \leq \Mod_p(A,B,G).
    \]
    On the other hand, if $U|_A= 0, U|_B = 1$, then $\rho(\{ x,y \}) = \abs{U(x) - U(y)}$ is, by the triangle inequality, $\Theta(A,B)$-admissible.
    Hence \eqref{eq: Mod-Cap} holds.
    Furthermore, if we define the density $\widetilde{\rho}(\{ x,y \}) = \abs{\widehat{U}(x) - \widehat{U}(y)}$, then
    \[
        \cM_p(\widetilde{\rho}) = \Mod_p(A,B,G).
    \]
    By the uniqueness of the optimal density, $\widehat{\rho} = \widetilde{\rho}$.

    Finally, let $\cF$ be the optimal flow function from $A$ to $B$. By \cite[Lemma 3.1]{NakamuraYamasakiDuality}, we have
    \[
    1=\sum_{\{x,y\}\in E} (U(y)-U(x))\cF(x,y). 
    \]
    Thus, by H\"older's inequality, we have
    \[
    1 \leq \left(\sum_{\{x,y\}\in E} |U(y)-U(x)|^p\right)^{\frac{1}{p}} \left(\sum_{\{x,y\}\in E} \cF(x,y)^q\right)^{\frac{1}{q}}.
    \]
    By duality, Proposition \ref{prop:duality}, the inequality is an equality. This is only possible if $\cF(x,y)$ is proportional to $\sgn(\widehat{U}(y) - \widehat{U}(x)) \abs{\widehat{U}(y) - \widehat{U}(x)}^{p-1}$. The propotionality constant can be computed using Proposition \ref{prop:duality} again.
\end{proof}

The previous proposition also encodes the so called Ohm's law for optimal potential functions and flows; see \cite[Section 2]{Kwapisz}.

In what follows, we will often drop the graph $G$ from the notation of moduli,  capacity and flows, where it is clear from context and write e.g. $\Mod_p(A,B)$, $\cCap_p(A,B)$ and $\cE_q(A,B)$.

\subsection{Moduli computations on replacement graphs}\label{subsection: Moduli computations}
The biggest advantage of the framework of IGS is the simplicity of edge modulus computations. Indeed, all modulus problems that we study in this paper can be reduced to the modulus problem over the paths between the gluing sets $\Theta(I_{v,\{ v,u\} }, I_{u,\{ v,u \}})$, for $\{u,v\}\in E_1$. These modulus problems are explicitly and effectively computable.

First we introduce the basic notations. For $e = \{ v,u \} \in E_n$ and $m \in \N$ we will write $\Theta_{v,e}^{(m)} = \Theta\left( I_{v,e}^{(m)}, I_{u,e}^{(m)} \right)$ for the path family in $G_m$.
We denote the optimal $\Theta_{v,e}^{(m)}$-admissible density (for both edge- and vertex $p$-modulus) as $\rho_{e,m}$. Observe that, by symmetry, $\rho_{e,m}$ is the optimal $\Theta^{(m)}_{u,e}$-density. Hence we omit the vertex $v,u$ from the notation of the optimal function $\rho_{e,m}$. For $m = 1$ we simply write $\Theta_{v,e}^{(1)}= \Theta_{v,e}$ and $\rho_{e,1} = \rho_{e}$. We denote the $\cE_q$-minimizing unit flow from $I_{v,e}$ to $I_{u,e}$ as $\cF_{v,e}$. 
Lastly we define the family of paths
\[
    \Theta_{m} := \bigcup_{\substack{e \in E_1\\ v \in e}} \Theta_{v,e}^{(m)}
\]
and the $\Theta_m$-admissible density $\widetilde{\rho}_m : V_m \to \R_{\geq 0}$ by
\begin{equation}\label{eq: rho_m}
    \widetilde\rho_m(v) := \max_{e \in E_1} \rho_{e,m}(v).
\end{equation}

\begin{remark}
Observe that, in general, $\rho_{e}$ and $\cF_{v,e}$ depend on $p$ and $q$. Moreover, for $p = 1$, $\rho_e$ is not unique. Since in most arguments we deal with only one $p > 1$ and $q > 1$ is the dual exponent of $p$, we omit $p,q$ from the notation.
\end{remark}

Next we introduce conductively uniform-condition, under which the moduli computations are extremely simple and straightforward.

\begin{definition}[Conductively uniform]\label{Def: Conductively uniform}
    We say that an iterated graph system satisfying doubling property is \emph{conductively uniform}
    if for every $e = \{ v,u \} \in E_n$ the value
    \begin{equation}\label{eq: Flows agree on boundary}
        \mathcal{F}_a := \mathcal{F}_{v,e}(\phi_{v,e}(a),\mathfrak{n}(\phi_{v,e}(a)))
    \end{equation}
    only depends on $a \in I$ and the value
    \begin{equation}\label{eq: Conductively uniform}
        \mathcal{R}_p := \mathcal{E}_q(\mathcal{F}_{v,e})
    \end{equation}
    does not depend on $e$ or $v$.
\end{definition}

The first of these means that the optimal flows in $G_1$ between the gluing sets are distributed in a way that is independent of the gluing used. This is particularly important when defining flows at higher level graphs, since we need compatibility across the gluing sets. The second of these properties means that the modulus ``accross'' an edge is independent of the edge. This allows us to compute optimal moduli at the $n$'th level in a simple way. As we will shortly see, optimality of such flows and potential functions is shown using Proposition \ref{prop:duality}. But, first, we shall give a few examples. 

\begin{example}\label{ex:uniformity} Recall the symmetric IGSs from Remark \ref{rmk:symmetricIGS}. These are all conductively uniform by the first part of Lemma \ref{lem:removableedges}. Indeed, it seems that the symmetric case is the most natural setting to obtain uniformity. Consequently, since Figures \ref{fig:replacementrule} and \ref{fig:laakso} are symmetric, they also are conductively uniform. 
\end{example}

Symmetry is, however, not always needed for conductive uniformity, see e.g. Figures \ref{fig:condutctunif} and \ref{fig:condutctunif-notsym}. 

\begin{figure}[ht]
\begin{tikzpicture}[scale=0.75]
\draw (-6,0)--(-4,0);
\draw[<-,thick] (-3,0)--(-2,0);

\draw (-1,-1)--(0,0)--(1,1)--(2,0)--(3,1);
\draw (0,0)--(1,-1)--(2,0);
\draw (-1,1)--(0,0);
\draw (2,0)--(3,-1);
\draw (0,0)--(0,-1);

\foreach \q in {(0,-1),(-6,0),(-4,0),(-1,1),(-1,-1),(0,0),(1,1),(1,-1),(2,0),(3,1),(3,-1)}
\node at \q [circle,fill,inner sep=1pt]{};

\node at (-1,1.4) {1};
\node at (-1,-1.5) {2};
\node at (0,0.4) {3};
\node at (0,-1.5) {4};
\node at (1,1.4) {5};
\node at (1,-1.5) {6};
\node at (2,0.4) {7};
\node at (3,1.4) {8};
\node at (3,-1.5) {9};
\end{tikzpicture}
\caption{This non-symmetric oriented IGS is obtained from the one in Figure \ref{fig:laakso} by adding one vertex, and an edge connecting the vertex $3$ to this added vertex. The vertices are re-ordered left-to-right. A direct calculation shows that this yields a conductively uniform IGS, and that the edge $\{3,4\}$ is removable in the sense of Definition \ref{def:removable}. In general, taking any conductively uniform IGS, and gluing to it an edge along one of its end points produces another conductively uniform IGS. In this example, $I=\{a,b\}$ and $\phi_{v,\{v,u\}}(a)=1,\phi_{v,\{v,u\}}(b) = 2$ if $v<u$, and $\phi_{v,\{v,u\}}(a)=8,\phi_{v,\{v,u\}}(b) = 9$ if $u<v$.}
\label{fig:condutctunif}
\end{figure}

\begin{figure}
\begin{tikzpicture}[scale=0.75]
\draw (-6,0)--(-4,0);
\draw[<-,thick] (-3,0)--(-2,0);

\draw (-1,-1)--(0,0)--(1,1)--(2,0)--(3,1);
\draw (0,0)--(1,-1)--(2,0);
\draw (-1,1)--(0,0);
\draw (2,0)--(3,-1);
\draw (-1,-2)--(0,-2)--(1,-2)--(2,0)--(3,-2);
\draw[dashed] (0,0)--(0,-2);

\foreach \q in {(-6,0),(-4,0),(-1,1),(-1,-1),(0,0),(1,1),(1,-1),(2,0),(3,1),(3,-1),(-1,-2),(0,-2),(1,-2),(3,-2)}
\node at \q [circle,fill,inner sep=1pt]{};

\node at (-1,1.4) {1};
\node at (-1,-0.5) {2};
\node at (-1,-2.5) {3};
\node at (0,0.4) {4};
\node at (0,-2.5) {5};
\node at (1,1.4) {6};
\node at (1,-0.5) {7};
\node at (1,-2.5) {8};
\node at (2,0.4) {9};
\node at (3,1.4) {10};
\node at (3,-0.5) {11};
\node at (3,-2.5) {12};
\end{tikzpicture}
\caption{ The solid lines in the figure shows one of the non-symmetric spaces, with $N=3, L_*=4$ constructed in the fashion of Example \ref{ex:nonsym}, together with the ordering of vertices described. The dashed line indicates an edge, which can be added to produce another conductively uniform oriented IGS with the $L_*$-uniform scaling property, where the dashed edge is removable in the sense of Definition \ref{def:removable}. In general, any edge which connects vertically aligned points is removable. }
\label{fig:condutctunif-notsym}
\end{figure}

\begin{example}\label{ex:nonsym}  Let $N,L_*\geq 2$.
A lot of (non-symmetric) conductively uniform oriented IGS that satisfy the $L_*$-uniform scaling property can be obtained as follows. Take $N$ line segments and subdivide each of them to $L_*$ equal parts. Along each line segment, there are $N$ subdivision points $v_{(k,i)}$, where $v_{(k,i)}$ is the $k$'th subdivision point in the $i$'th copy of the line segment. We then form a new graph by identifying, for $k=2,\dots, L_*-1$, pairs of these subdivision points $v_{(k,i)}\sim v_{(k,j)}$ for $(i,j)\in S_k$, where $S_k$ are some given equivalence relations on $\{1,\dots, N\}$. Given enough such identifications, we obtain a connected IGS $G_1=(V_1,E_1)$. The gluing rules are then described as follows. First, we give an order for the equivalence classes $[v_{(k,i)}]$, which are ordered so that if $k<l$, then $[v_{(k,j)}]<[v_{(l,m)}]$ for all $k<l$ and all $j\leq m$. Let $v_i$, $i=1,\dots, M$  be the  ordered list of equivalence classes, where $M$ is the total number of vertices. The (oriented) gluing rules can then be given by  $I=\{1,\dots, N\}$ and $\phi_{v,\{v,u\}}(i)=\phi_-(i)=i$ if $v<u$, and $\phi_{u,\{v,u\}}(i)=\phi_+(i)=M-N+i$ if $v<u$.   
In these examples, conductive uniformity can be verified by using duality and using the weight function $\rho(e)=\frac{1}{L_*}$ and flow function $\mathcal{F}(x,y)=\frac{1}{N}$ if $x<y$ and $\{x,y\}\in E$. A direct calculation shows that $\cE_q(\cF)^{\frac{1}{q}}\cM_p(\rho)^{\frac{1}{p}}=1$ for all $p\in (1,\infty)$, and thus $\cF$ and $\rho$ are optimal by Proposition \ref{prop:duality}. See Figure \ref{fig:condutctunif-notsym}, which shows one example of an IGS obtained in this fashion that is non-symmetric and that satisfies the $L_*$-uniform scaling property. These examples are also all Loewner, and they attain their conformal dimensions, see Example \ref{ex:loewnerreplacement}. By adding ``removable edges'' to them, we obtain counterexamples to Kleiner's conjecture, see Section \ref{sec:mainthm}.
\end{example}

\begin{lemma}\label{lemma: Reversing flow}
    For every $e = \{ v,u \} \in E_n$ we have
    \begin{equation}\label{eq: Reversing flow}
        \mathcal{F}_{v,e} = -\mathcal{F}_{u,e}.
    \end{equation}
\end{lemma}

\begin{proof}
    As $\cE_q(-\cF_{v,e}) = \cE_q(\cF_{v,e})$ and $-\cF_{v,e}$ is a unit flow from $I_{u,e}$ to $I_{v,e}$, the claim follows from the uniqueness of the optimal flow.
\end{proof}

\begin{lemma}
    If the iterated graph system is doubling and conductively uniform then
    \begin{equation}\label{eq: Conductively uniform (Modulus)}
        \mathcal{M}_p:= \mathcal{M}_p(\rho_{e})
    \end{equation}
    does not depend on $e = \{v,u\} \in E_n$.
    In particular,
    \begin{equation}\label{eq: Resistance and Modulus}
        (\mathcal{M}_p)^{\frac{1}{p}} \cdot (\mathcal{R}_p)^{\frac{1}{q}} = 1.
    \end{equation}
\end{lemma}

\begin{proof}
    Directly follows from Proposition \ref{prop:duality} and \eqref{eq: Conductively uniform}.
\end{proof}

\begin{proposition}\label{prop: properties of cM_p}
    If the iterated graph system satisfies the $L_*$-uniform scaling property for $L_* \geq 2$, the doubling property and is conductively uniform, then $\mathcal{M}_p$ is continuously strictly decreasing in $p \in [1,\infty)$.
    Furthermore, if $G_1$ contains two edge-wise disjoint paths $\theta_1,\theta_2 \in \Theta_{v,e}$ for some $e = \{v,u\} \in E_1$, then $\mathcal{M}_1 > 1$ and there is a unique $Q_* \in (1,\infty)$ so that $\mathcal{M}_{Q_*} = 1$.
\end{proposition}

\begin{proof}
    In order to avoid confusion with the $p$-dependence of the optimal density, we will write $\rho_{e,p}$ as the optimal $\Theta_{v,e}$-admissible density.
    
    If $\rho$ is admissible for $\Theta$, then $\min(\rho,1)$ is also admissible. Thus, for the optimal admissible density function we have $\rho_{e,p}(f)\leq 1$ for all $f\in E.$ From this, it is direct to verify that for any non-empty family of paths $\Theta$ in a finite graph $G$, $\Mod_p(\Theta,G)$ is continuously decreasing in $p \in [1,\infty)$.
    To prove that $\cM_p = \Mod_p(\Theta_{v,e})$ is strictly decreasing, it is sufficient to prove that for all $p \in (1,\infty)$ and $f \in E_1$ we have $\rho_{e,p}(f) < 1$. Indeed then
    \[
        \cM_p = \cM_p(\rho_{e,p}) >  \cM_{p + \epsilon}(\rho_{e,p}) \geq \cM_{p + \epsilon}(\rho_{e,p + \epsilon}) = \cM_{p + \epsilon}
    \]
    for any $\epsilon > 0$.
    
    Fix $p \in (1,\infty)$. By Lemmas \ref{lemma: Energy minimizing and p-harmonic} and \ref{lemma: Mod-Cap} there is a function $U : V_1 \to \R$, which is $p$-harmonic in $V_1 \setminus (I_{v,e} \cup I_{u,e})$ and $\rho_{e,p}(\{ x,y \}) = \abs{U(x) - U(y)}$ for all $\{x,y\} \in E_1$. Now if $\{x,y\} \in E_1$, by the uniform scaling property and the fact that $I_{v,e},I_{u,e}$ are independent sets, at most one of the vertices $x,y$ is contained in $I_{v,e} \cup I_{u,e}$. Furthermore, by the doubling property, $V_1 \setminus (I_{v,e} \cup I_{u,e})$ is connected and $\overline{V_1 \setminus (I_{v,e} \cup I_{u,e})} = V_1$. By the strong maximum principle, Lemma \ref{lemma: Strong max principle}, we have
    \[
        \rho_{e,p}(\{ x,y \}) = \abs{U(x) - U(y)} < 1
    \]
    for all $\{ x,y \} \in E_1$ and so  we conclude that $\mathcal{M}_p$ is strictly decreasing.

    We proceed to the second part. Let  $\theta_1 = [v_1,\dots,v_k],\theta_2 = [u_1,\dots,u_l] \in \Theta_{v,e}$ be two edge-wise disjoint paths. By taking a sub-path, we may assume that both are simple paths.
    Now if $\rho$ is $\Theta_{v,e}$-admissible we have
    \[
        \mathcal{M}_1(\rho) = \sum_{e \in E_1} \rho(e) \geq \sum_{i = 1}^{k - 1} \rho(\{ v_i,v_{i + 1} \}) + \sum_{i = 1}^{l - 1} \rho(\{ u_i,u_{i + 1} \}) \geq 2
    \]
    Therefore $\mathcal{M}_1 > 1$.
    On the other hand, by the $L_*$-uniform scaling property, the density $\widetilde{\rho} \equiv L_*^{-1}$ is $\Theta_{v,e}$-admissible.
    Hence
    \[
        \mathcal{M}_p \leq \cM_p(\widetilde{\rho}) = \abs{E_1} \cdot \frac{1}{L_*^p} < 1
    \]
    for large enough $p > 1$.
    By continuity of $\cM_p$, there must be $Q_* \in (1,\infty)$ so that $\cM_{Q_*} = 1$.
    Since $\cM_p$ is strictly decreasing, this $Q_*$ is unique.
\end{proof}

For the rest of this section we assume that the IGS is doubling and conductively uniform.

\begin{proposition}\label{prop: Replacement density}
    Let $A, B \subseteq V_n$ be non-empty disjoint subsets and $m \in \N$.
    Given a $\Theta(A,B)$-admissible density $\rho_n : E_n \to \R_{\geq 0}$ there is a $\Theta\left(\pi_{n+m,n}^{-1}(A),\pi_{n+m,n}^{-1}(B)\right)$-admissible density $\rho_{n+m} : E_{n+m} \to \R_{\geq 0}$ with
    \begin{equation}\label{eq: Replacement density}
        \mathcal{M}_p(\rho_{n+m}) = \mathcal{M}_p(\rho_n)\cdot \mathcal{M}_p^m \text{ and } \supp(\rho_{n+m}) \subseteq \pi_{n+m,n}^{-1}(\supp(\rho_n)).
    \end{equation}
\end{proposition}

\begin{proof}
    It is sufficient to prove the case $m = 1$. Define
    \[
        \rho_{n+1}(\{[z_1,e],[z_2,e]\}) = \rho_n(e) \cdot \rho_e(\{ z_1,z_2 \}).
    \]
    By \ref{SM3} $\rho_{n+1}$ is well-defined and, by \eqref{eq: Conductively uniform (Modulus)}, it is sufficient to prove that $\rho_{n + 1}$ is $\Theta\left(\pi_{n+1,n}^{-1}(A),\pi_{n+1,n}^{-1}(B)\right)$-admissible.
    Let $\theta \in \Theta\left(\pi_{n+1,n}^{-1}(A),\pi_{n+1,n}^{-1}(B)\right)$ and $\theta= [v_1,\dots, v_k]$.
    By Proposition \ref{prop: Path decomposition} there is a path $\hat{\theta} = \left[u_1,\dots,u_l\right] \in \Theta(A,B)$ and disjoint sub-paths $\theta_1,\dots, \theta_{l - 1}$ of $\theta$ where 
    \[
        \theta_i \in \Theta(\pi_{n+1,n}^{-1}(u_i),\pi_{n+1,n}^{-1}(u_{i+1})) \text{ and } e_i = \{ u_i,u_{i+1} \}
    \]
    and $\theta_i$ is contained in $e_i \cdot G_1$ where $e_i = \{ u_i,u_{i+1} \}$ for all $i = 1,\dots,l-1$.
    By \ref{SM1}, $\hat{\theta}_i := \sigma_{e_i}^{-1}(\theta_i) \in \Theta_{e_i}$,
    so we have
    \begin{align*}
        L_{\rho_{n+1}}(\theta) & \geq \sum_{i = 1}^{l - 1} L_{\rho_{n+1}}(\theta_i)\\
        & = \sum_{i = 1}^{l-1} \rho_n\left(e_i\right) \cdot L_{\rho_{e_i}}\left(\hat{\theta}_i\right)\\
        & \geq \sum_{i = 1}^{l-1} \rho_n\left(e_i\right)\\
        & \geq 1.
    \end{align*}
\end{proof}

\begin{proposition}\label{prop: Replacement flow}
    Let $A, B \subseteq V_n$ be non-empty disjoint subsets.
    Given any unit flow $\mathcal{F}_n$ from $A$ to $B$ there is a unit flow
    $\mathcal{F}_{n+m}$ from $\pi_{n+m,n}^{-1}(A)$ to $\pi_{n+m,n}^{-1}(B)$ 
    with 
    \begin{equation}\label{eq: Replacement flow}
        \mathcal{E}_q(\mathcal{F}_{n+m}) = \mathcal{E}_q(\mathcal{F}_{n}) \cdot \mathcal{R}_p^m \text{ and } \supp(\mathcal{F}_{n+m}) \subseteq \pi_{n+m,n}^{-1}(\supp (\mathcal{F}_n)).
    \end{equation}
\end{proposition}

\begin{proof}
    Sufficient to prove the case $m = 1$.
    We define $\mathcal{F}_{n+1}$ as follows:
    For $e = \{v,u\} \in E_n$ and $z_1,z_2 \in V$
    we set
    \[
        \mathcal{F}_{n + 1}([z_1,\{ v,u \}],[z_2,\{v,u\}]) = \mathcal{F}_n(v,u) \cdot \mathcal{F}_{v,e}(z_1,z_2).
    \]
    The above definition is well-defined in the sense that it does not depend on the orientation of $e = \{v,u\}$.
    Indeed, by Lemma \ref{lemma: Reversing flow} we have
    \begin{align*}
        \mathcal{F}_{n + 1}([z_1,\{ u,v \}],[z_2,\{ u,v \}]) &  = \mathcal{F}_n(u,v) \cdot \mathcal{F}_{u,e}(z_1,z_2)\\
        & \hspace{-6pt}\stackrel{\eqref{eq: Reversing flow}}{=} (-\mathcal{F}_n(v,u)) \cdot (-\mathcal{F}_{v,e}(z_1,z_2))\\
        & = \mathcal{F}_n(v,u) \cdot \mathcal{F}_{v,e}(z_1,z_2)\\
        & = \mathcal{F}_{n + 1}([z_1,\{ v,u \}],[z_2,\{ v,u \}]).
    \end{align*}
    Also \eqref{eq: Replacement flow} follows from \eqref{eq: Conductively uniform} and the the definition of $\mathcal{F}_{n + 1}$ so we only need to prove that $\mathcal{F}_{n + 1}$ is a unit flow from $\pi_{n+1,n}^{-1}(A)$ to $ \pi_{n+1,n}^{-1}(B)$. Let $[z,\{v,u\}]\in V_{n+1}$ for some $e = \{ v,u \} \in E_n$ and $z \in V$. Assume that $[z,\{v,u\}]\not\in \pi_{n+1,n}^{-1}(A) \cup \pi_{n+1,n}^{-1}(B)$. There are two cases to consider. 
     First, assume that $z \notin I_{v,e} \cup I_{u,e}$. Then, by \ref{SM1} and \ref{SM2}, we have
    \[
        \divr(\mathcal{F}_{n+1})([z,\{ v,u \}]) = \divr(\mathcal{F}_{v,e})(z) \cdot \mathcal{F}_{n}(v,u) = 0.
    \]
    Next consider $z \in I_{v,e} \cup I_{u,e}$. By symmetry we may assume $z = \phi_{v,e}(a) \in I_{v,e}$ for some $a\in I$.
    Let $e_1,\dots,e_k$ be the edges in $E_{n+1}$ containing $[z,\{v,u\}]$ and write those as $e_i = \{ [z,\{ v,u \}], [z_i,\{ v_i,u_i\}]\}$.
    Then, by definition, $\{v,u\}$ and $\{ v_i,u_i\}$ have a common vertex, which must be $v$ since otherwise $z \in I_{v,e} \cap I_{u,e}$.
    Up to relabeling, we may assume $v_i = v$ so we have $z_i = \mathfrak{n}(\phi_{v_i,e_i}(a))$ and $v = v_i = \phi_{v_i,\{v_i,u_i\}}(a)$.
    We get
    \begin{align*}
        \divr(\mathcal{F}_{n + 1})([z,\{ v,u \}]) 
        & = \sum_{i = 1}^k \mathcal{F}_{n + 1}([z,\{ v,u \}],[z_i,\{ v_i,u_i \}])\\
        & = \sum_{i = 1}^k \mathcal{F}_{n+1}([\phi_{v_i,\{v_i,u_i\}}(a),\{ v_i,u_i \}],[z_i,\{ v_i,u_i \}])\\
        & = \sum_{i = 1}^k \mathcal{F}_{n}(v_i,u_i) \cdot \mathcal{F}_{v_i,\{v_i,u_i\}}(\phi_{v_i,\{v_i,u_i\}}(a),z_i)\\
        & \hspace{-6pt}\stackrel{\eqref{eq: Flows agree on boundary}}{=} \mathcal{F}_a \cdot \sum_{i = 1}^k F_{n}(v_i,u_i)\\
        & = \mathcal{F}_a \cdot \divr(F_n)(v).
    \end{align*}
    Recall that $[z,\{v,u\}]\not\in \pi_{n+1,n}^{-1}(A) \cup \pi_{n+1,n}^{-1}(B) $ and $z \in I_{v,\{v,u\}}$. Since $[z,\{v,u\}] \notin \pi_{n+1,n}^{-1}(A) \cup \pi_{n+1,n}^{-1}(B)$, we must have $v \notin A_n \cup B_n$. This concludes showing that $\cF_{n+1}$ is a flow from $\pi_{n+1,n}^{-1}(A)$ to $\pi_{n+1,n}^{-1}(B)$.
    
    Finally, if $v \in A_n$ and $z \in I_{v,\{v,u\}}$ then $[z,\{ v,u \}] \in \pi_{n+1,n}^{-1}(A)$.
    Hence
    \begin{align*}
        I(\mathcal{F}_{n+1}) & = 
        \sum_{[z,\{ v,u \}] \in \pi_{n+1,n}^{-1}(A)} \divr(\mathcal{F}_{n+1})([z,\{ v,u \}])\\
        & = \quad \sum_{a \in I} \mathcal{F}_a \cdot \sum_{v \in A_n} \divr(\mathcal{F}_n)(v)\\
        & = \quad \sum_{a \in I} \mathcal{F}_a \cdot I(\mathcal{F}_n)\\
        & = \quad 1.
    \end{align*}
    Thus $\cF$ is a unit flow, and the claim follows.
\end{proof}

\begin{corollary}\label{corollary: Modulus from L to R lvl m}
    If $e = \{ v,u \} \in E_n$ is any edge then
    \begin{equation}\label{eq: Modulus from L to R lvl m}
        \Mod_p\left(\Theta_{v,e}^{(m)}, G_m  \right) = \mathcal{M}_p^m
    \end{equation}
\end{corollary}

\begin{proof}
    By applying Propositions \ref{prop: Replacement density} and \ref{prop: Replacement flow} to the $\Theta_{v,e}$-admissible density $\rho_{e}$ and the unit flow $\mathcal{F}_{v,e}$ from $I_{v,e}$ to $I_{u,e}$ respectively, we obtain the $\Theta\left(I_{v,e}^{(m)},I_{u,e}^{(m)}\right)$-admissible density $\rho_{m}$ and the unit flow $\mathcal{F}_{m}$ from $I_{v,e}^{(m)}$ to $I_{u,e}^{(m)}$.
    By \eqref{eq: Replacement density} and \eqref{eq: Replacement flow}, they satisfy
    \[
        \mathcal{M}_p(\rho_{m}) = \mathcal{M}_p^m \text{ and } \mathcal{E}_q(\mathcal{F}_m) = \mathcal{R}_p^m.
    \]
    By \eqref{eq: Resistance and Modulus} we have $(\mathcal{M}_p)^{\frac{1}{p}}(\mathcal{R}_p)^\frac{1}{q}=1$, and thus $\mathcal{M}_p(\rho_{m})^{\frac{1}{p}}\mathcal{E}_q(\mathcal{F}_m)^{\frac{1}{q}}=1$. Thus Proposition \ref{prop:duality} yields \eqref{eq: Modulus from L to R lvl m}.
\end{proof}

\begin{corollary}\label{lemma: mass of rho_m}
    The density $\widetilde{\rho}_m$ is $\Theta_m$-admissible and there is a constant $C = C(p,C_{\deg},\abs{E_1})$ so that
    \begin{equation}\label{eq: mass of rho_m}
        \mathcal{M}_p(\widetilde{\rho}_m) \leq C \cdot  \mathcal{M}_p^{m}.
    \end{equation}
    Here $\widetilde{\rho}_m$ is as in \eqref{eq: rho_m}.
\end{corollary}

\begin{proof}
    Admissibility is clear and 
    the $p$-mass estimate follows by
    \begin{equation*}
        \mathcal{M}_p(\widetilde{\rho}_m) \leq \abs{E_1} \cdot \max_{e \in E_1} \mathcal{M}_p(\rho_{e,m}) \stackrel{\eqref{eq: Edge- and Vertex modulus}}{\leq} C \cdot  \mathcal{M}_p^{m}.
    \end{equation*}
\end{proof}

\begin{theorem}\label{thm: Moduli asymptotic}
    Let $A, B \subseteq V_n$ be non-empty disjoint subsets. Then
    \begin{equation}
        \Mod_p(\pi_{n+m,n}^{-1}(A),\pi_{n+m,n}^{-1}(B)) = \Mod_p(A,B) \cdot \mathcal{M}_p^{m}.
    \end{equation}
    Moreover, if $\rho_n : V_n \to \R_{\geq 0}$ is the optimal $\Theta(A,B)$-admissible density then the optimal $\Theta(\pi_{n+m,n}^{-1}(A),\pi_{n+m,n}^{-1}(B))$-admissible density is
    \begin{equation}\label{eq: Def replacement density}
        \rho_{n + m}(e) = \rho_n(\pi_{n+m,n}(e)) \cdot \rho_{\pi_{n+m,n}(e),m}\left(\sigma_{\pi_{n+m,n}(e),m}^{-1}(e)\right).
    \end{equation}
    In particular, $\supp \,\rho_{n+m} \subseteq \pi_{n + m,n}^{-1}(\supp \, \rho_{n}).$
\end{theorem}

\begin{proof}
    When we apply Proposition \ref{prop: Replacement density} to $\rho_n$, it follows from the proof of the proposition, that the density constructed is exactly $\rho_{n + m}$ as defined in \eqref{eq: Def replacement density}.
    Furthermore, it follows from similar argument as in Corollary \ref{corollary: Modulus from L to R lvl m} that $\rho_{n + m}$ is indeed the optimal $\Theta(\pi_{n+m,n}^{-1}(A),\pi_{n+m,n}^{-1}(B))$-admissible density.
\end{proof}

\section{Combinatorial Loewner property}\label{sec:combloew}
In this section we introduce easily verifiable sufficient conditions for the limit space $X$ of the IGS to satisfy the combinatorial Loewner property. These are all satisfied by the examples presented in this paper, except for the last one which fails only for Example \ref{ex:laaksodiamond}. Notice that the failure in this example is expected due to the existence of cut points.
\begin{assumption}\label{Assumptions: CLP}
    The iterated graph system satisfies
    \begin{enumerate}
        \item $L_*$-uniform scaling property,
        \item Doubling property,
        \item Conductively uniform property and
        \item[(4*)] the graph $G_1$ contains at least two edge-wise disjoint paths in $\Theta_{v,e}$ for some $e = \{ v,u \} \in E_1$.
    \end{enumerate}
\end{assumption}

The goal of this section is to prove the following theorem.
\begin{theorem}\label{thm:combloew} 
    If the iterated graph system satisfies Assumption \ref{Assumptions: CLP} then 
    $X$ satisfies the combinatorial $Q_*$-Loewner property for $Q_* = \dim_{\AR}(X) > 1$.
\end{theorem}

Recall the notation from Subsection \ref{subsec:combloewdef}. Throughout this section we assume that the IGS satisfy (1)-(3) in Assumption \ref{Assumptions: CLP}. The condition (4*) will only be used to ensure that the conformal dimension of the limit space is strictly greater than $1$ and is only used in the proof of Theorem \ref{thm:combloew}. Furthermore, for the rest of the paper, $\cM_p$ is always the moduli constant in \eqref{eq: Conductively uniform (Modulus)}.

During the moduli computations, we will adopt the notation $A \lesssim B$ (resp. $A \gtrsim B$) if there is a constant $C > 0$, which depends on at most $L_*,C_{\deg}, C_{\diam}, |E|,p$ and dependence on $p$ is continuous, so that $A \leq C\cdot B$ (resp. $A \geq B/C$). We also write $A \asymp B$ whenever $A \lesssim B$ and $B \lesssim A$.

\begin{lemma}\label{lemma: Discretization of connected set}
    If $F \subseteq X$ is a connected set then $G_m[F] \subseteq V_m$ is a connected set of the graph $G_m$. Moreover, if $\Gamma$ is a family of paths in $X$, then
    \[
        \Mod_p^{D}(\Gamma, G_m) = \Mod_p^V(G_m[\Gamma], G_m).
    \]
\end{lemma}

\begin{proof}
    Follows directly from connectedness of $X$ and from the fact that $G_m$ can be regarded as an incidence graph of $X$ corresponding to an open cover (Propositions \ref{prop: GH-convergence} and \ref{prop: Graph approximations}).
\end{proof}

\begin{proposition}\label{prop: Discrete and combinatorial modulus are comparable}
    Let $F_1,F_2 \subseteq X$ be disjoint subsets so that $G_m[F_1],G_m[F_2] \subseteq V_m$ are also disjoint. Then
    \begin{equation*}
    \Mod_p(\Theta(G_m[F_1],G_m[F_2]),G_m) \asymp \Mod_p^{D}(\Gamma(F_1,F_2), G_m)
    \end{equation*}
\end{proposition}

\begin{proof}
    By Lemma \ref{lemma: Edge- and Vertex modulus} it is sufficient to prove the comparison for the vertex modulus.
    First, it follows from Lemma \ref{lemma: Discretization of connected set} that $G_m[\Gamma(F_1,F_2)] \subseteq \Theta(G_m[F_1],G_m[F_2])$, which yields
    \[
        \Mod_p^{D}(\Gamma(F_1,F_2), G_m) \leq \Mod_p^V(\Theta(G_m[F_1],G_m[F_2]),G_m).
    \]
    To prove the other inequality, let $\rho : V_m \to \R_{\geq 0}$ be $\Gamma(F_1,F_2)$-admissible.
    Now given a path $\theta \in \Theta(G_m[F_1],G_m[F_2])$, as each $X_v$ is path connected by Proposition \ref{prop: GH-convergence} and self-similarity, there is a path $\gamma_{\theta} \in \Gamma(F_1,F_2)$ so that 
    \[
        \gamma_{\theta} \subseteq \bigcup_{v \in \theta} X_v.
    \]
    Hence
    \[
        G_m[\gamma_{\theta}] \subseteq \theta \cup \{ u \in V_m : \{ v,u \} \in E_m \text{ for some } v \in \theta \}.
    \]
    Now let $\widehat{\rho} : V_m \to \R_{\geq 0}$ so that
    \[
        \widehat{\rho}(v) := \max_{\substack{u \in V_m \\ \{u,v\} \in E_m}} \rho(u).
    \]
    Then
    \[
        C_{\deg} \cdot \sum_{v \in \theta} \widehat{\rho}(v) \geq \sum_{v \in \theta} \sum_{\{ u,v \} \in E_m} \rho(u) \geq \sum_{X_u \cap \gamma_{\theta} \neq \emptyset} \rho(u) \geq 1
    \]
    which gives
    \[
        \Mod_p^V(\Theta(G_m[F_1],G_m[F_2]),G_m) \leq (C_{\deg})^p \cdot \Mod_p^{D}(\Gamma(F_1,F_2),G_m)
    \]
\end{proof}

Recall the definition of $\mathcal{M}_{\delta,p}^{(m)}$ from \eqref{eq:moddeltadef}.
\begin{proposition}\label{prop: LR modulus and BK-type modulus}
    If $0 < \delta < 1$, then there is $C = C(p,\deg(G),\abs{E_1},\delta) \geq 1$ so that
    \begin{equation}\label{eq: LR modulus and BK-type modulus}
        C^{-1} \cdot \mathcal{M}_{p}^{m} \leq \mathcal{M}_{\delta,p}^{(m)} \leq C \cdot \mathcal{M}_{p}^{m}.
    \end{equation}
\end{proposition}

\begin{proof}
    As $\delta < 1$ and since the distance in $G_1$ from $I_{v,e}$ to $I_{u,e}$ is one for all $e=\{u,v\}\in G_1$, by Proposition \ref{prop: Discrete and combinatorial modulus are comparable}, we have
    \[
        \mathcal{M}_p^{m} = \Mod_p\left(\Theta_{v,e}^{(m)}, G_m\right) \lesssim 
        \mathcal{M}_{\delta,p}^{(m)}.
    \]
    In order to prove the other inequality, set $k_* \in \N$ to be the smallest positive integer so that $8C_{\diam} \cdot L_*^{-k_*} < \delta$. Then we define $\widehat{\rho}_{m + k_*} : V_{m + k_*} \to \R_{\geq 0}$
    \[
        \widehat{\rho}_{m + k_*}(v) = \max_{\substack{e \in E_{k_*}\\ v \in e \cdot G_m}} \widetilde{\rho}_m\left(\sigma_{e,m}^{-1}(v)\right).
    \]
    Here $\widetilde{\rho}_m$ is as in \eqref{eq: rho_m}.
    First, $\widehat{\rho}_{m + k_*}$ has the $p$-mass estimate
    \begin{align*}
        \mathcal{M}_p( \widehat{\rho}_{m + k_*})
        & \leq \sum_{e \in E_{k_*}} \sum_{z \in V_m} \widetilde{\rho}_m(z)^p
        \stackrel{\eqref{eq: mass of rho_m}}{\leq} \abs{E_{k_*}} \cdot C' \cdot \mathcal{M}_p^{m} = C \cdot \mathcal{M}_p^{m}
    \end{align*}
    where $C = C(p,\deg(G),\abs{E_1},\delta)$.
    Note that the dependence of $\delta$ comes from $k_*$.
    We will conclude the proof by showing that $\widehat{\rho}_{m + k_*}$ is $\Gamma_{\delta}$-admissible.
    Let $\gamma \in \Gamma_{\delta}$ and $\theta = G_{m + k_*}[\gamma]$.
    Since $\diam(X_v) \leq 2 \cdot C_{\diam}\cdot L_*^{-k_*}$ for all $v \in V_{k_*}$, by the choice of $k_*$, there are vertices $v_1,v_2 \in V_{k_*}$ so that so that $d_{G_{k_*}}(v_1,v_2) \geq 3$ and
    \[
        \gamma \cap X_{v_1}, \gamma \cap X_{v_2} \neq \emptyset.
    \]
    Then $\theta$ contains vertices in two disjoint edges $\{ v_1,u_1 \} \cdot G_m$ and $\{ v_2,u_2 \} \cdot G_m$ for some $u_1,u_2, v_1,v_2 \in V_{k_*}$.
    By Proposition \ref{prop: Path between disjoint edges}, there is an edge $f:=\{\hat{u},\hat{v}\}\in E_{k_*}$ so that
    $\theta$ contains a sub-path $\theta'$ from $\pi_{m + k_*}^{-1}(\hat{u})$ to $\pi_{m + k_*}^{-1}(\hat{v})$ and is contained in $f \cdot G_{m}$.
    Hence
    \[
        \sum_{v \in \theta} \widehat{\rho}_{m + k_*}(v) \geq  \sum_{v \in \theta'} \widehat{\rho}_{m + k_*}(v) \geq \sum_{v \in \sigma_{f,m}^{-1}(\theta')} \rho_{f,m}(v) \geq 1.
    \]
\end{proof}

\begin{proposition}\label{prop: Ball-Loewner}
    Let $A > 0, n \in \N$ and so that $\dist(B_1,B_2) \leq A \cdot r$ and $B_1 = B(x,r), B_2 = B(y,r)$ be two disjoint balls in $X$ so that $2C_{\diam} \cdot L_*^{-n} < r \leq 2C_{\diam} \cdot L_*^{-(n-1)}$. If $\Gamma$ is the family of paths
    \begin{equation}
        \Gamma = \left\{ \gamma \in \Gamma(B_1,B_2) : \diam(\gamma) \leq 2C_{\diam}L_*(A + 4) \cdot r \right\}
    \end{equation}
    then
    \begin{equation}
        \Mod_p^{D}(\Gamma, G_{n + m}) \gtrsim \left(\frac{1}{A + 4}\right)^{p-1} \cdot \mathcal{M}_p^{m}.
    \end{equation}
\end{proposition}

\begin{proof}
    Let $v,u \in V_n$ so that $x \in X_v$ and $y \in X_u$. It follows from Proposition \ref{prop: Graph approximations} that 
    \[
        B(z_v,C_{\diam}\cdot L_*^{-n}) \subseteq B(x,r) \text{ and } B(z_u,C_{\diam}\cdot L_*^{-n}) \subseteq B(y,r).
    \]
    In particular,
    \[
        d_n(v,u) \leq d_X(z_v,z_u) \leq (A + 4)\cdot r \leq 2C_{\diam} L_*^{-(n - 1)} \cdot (A + 4)
    \]
    so there is a path $\theta = [v_1,\dots,v_k]$ in $G_n$ from $v$ to $u$ of length at most $ 2C_{\diam}L_* \cdot (A + 4)$.
    We will now define the family of paths in $X$
    \[
        \widehat{\Gamma} = \left\{ \gamma \in \Gamma(B_1,B_2) : \gamma \subseteq \bigcup_{i = 1}^k X_{v_i} \right\}
    \]
    and families of paths in $G_{n + m}$
    \[
        \Theta = \left\{ \theta \in \Theta(G_{n+m}[B_1],G_{n+m}[B_2]) : \theta \subseteq \bigcup_{i = 1}^{k - 1} \{ v_i,v_{i+1} \} \cdot G_n \right\}
    \]
    and
    \[
        \widehat{\Theta} = \left\{ \theta \in \Theta(\pi_{n+m,n}^{-1}(v),\pi_{n+m,n}^{-1}(u)) : \theta \subseteq \bigcup_{i = 1}^{k - 1} \{ v_i,v_{i+1} \} \cdot G_n \right\}.
    \]
    Indeed, since
    \[
        \diam\left( \bigcup_{i = 1}^k X_{v_i} \right) \leq (k - 1) 2C_{\diam} \cdot L_*^{-n} \leq 2C_{\diam}L_* (A + 4)\cdot r 
    \]
    we have $\widehat{\Gamma} \subseteq \Gamma$.
    From $X_{v} \subseteq B_1$ and $X_u \subseteq B_2$ it follows that $\widehat{\Theta} \subseteq \Theta$.
    By similar argument as in Proposition \ref{prop: Discrete and combinatorial modulus are comparable} when proving the latter inequality, we see that
    \[
        \Mod_p^{D}\left(\widehat{\Gamma}, G_{n+m}\right) \gtrsim \Mod_p(\Theta; G_{n+m}).
    \]
    It now follows that
    \begin{align*}
        \Mod_p^{D}\left(\Gamma, G_{n + m}\right) & \geq \Mod_p^{D}\left(\widehat{\Gamma}, G_{n + m}\right)
         \gtrsim \Mod_p\left(\Theta, G_{n+m}\right)
         \geq \Mod_p\left(\widehat{\Theta}, G_{n + m}\right).
    \end{align*}
    Next let $\mathcal{F}_n$ to be the constant unit flow along $\theta$ from $u$ to $v$ (i.e. $\mathcal{F}_n(\{x,y\})=|\{i : x=v_{i},y=v_{i+1}\}|-|\{i : x=v_{i+1},y=v_{i}\}|$). We have $\cE_q(\cF_n)\leq 2C_{\diam}L_*$. 
    By Proposition \ref{prop: Replacement flow} there is a unit flow $\mathcal{F}_{n+m}$ from $\pi_{n+m,n}^{-1}(v)$ to $\pi_{n+m,n}^{-1}(u)$ with
    \begin{equation*}
        \mathcal{E}_q(\mathcal{F}_{n+m}) \leq  2C_{\diam}L_*(A + 4) \cdot \left(\frac{1}{\mathcal{M}_p^{(m)}}\right)^{\frac{1}{p - 1}} \text{ and } \supp(\mathcal{F}_{n+m}) \subseteq \bigcup_{i = 1}^{k - 1} \{ v_i,v_{i + 1} \} \cdot E_n.
    \end{equation*}
    By applying Proposition \ref{prop:duality} to the sub-graph that is induced by the vertices $\bigcup_{i = 1}^{k - 1} \{ v_i,v_{i+1} \} \cdot G_n$, we obtain
    \[
        \Mod_p\left( \widehat{\Theta}, G_{n + m} \right) \geq \left(\frac{1}{\mathcal{E}_q(\mathcal{F}_{n+m})}\right)^{p-1} \geq \left(\frac{1}{2C_{\diam}L_* \cdot (A + 4)}\right)^{p-1} \mathcal{M}_p^{m}.
    \]
    This concludes the proof.
\end{proof}

\begin{proposition}\label{prop: Annulus modulus}
    Let $B = B(x,r)$ where $4 C_{\diam} \cdot L_*^{-n} < r \leq 4 C_{\diam} \cdot L_*^{-(n - 1)}$ for some $n\in \N$.
    Then for all $m \in \N$
    \begin{equation}\label{eq: Annulus modulus}
        \Mod_p^{D}\left(\Gamma\left(\overline{B(x,r)},X \setminus B(x,2r)\right), G_{n+m}\right) \lesssim \mathcal{M}_p^{m}
    \end{equation}
\end{proposition}

\begin{proof}
    Let $v_1,\dots,v_l \in V_{n}$ be the vertices $v \in V_n$ so that $X_v \cap \overline{B(x,r)} \neq \emptyset$ and $u_1,\dots,u_s \in V_n \setminus \{ v_1,\dots,v_l \}$ be the vertices in $u \in V_n$ so that $X_{u} \cap X_{v_i} \neq \emptyset$ for some $i$.
    Note that
    \[
        d_X\left(x,X_{v_{i}}\right) \leq r + \diam\left(X_{v_{i}}\right) \leq r + 2C_{\diam} \cdot L_*^{-n} < 2r
    \]
    and
    \[
        d_X\left(x,X_{u_{j}}\right) \leq r + 4 C_{\diam} \cdot L_*^{-n} < 2r.
    \]
    Hence 
    \begin{equation}\label{eq: Cells contained}
        X_{v_i}, X_{u_{j}} \subseteq B(x,2r).
    \end{equation}
    
    Write $e_{1}, \dots, e_{k} \in E_n$ for the edges that contain $u_{j}$ for some $j$ and, for $e \in E_n$, we define $\lambda_{e,m}:E_m\to [0,1]$ by
    \[
        \lambda_{e,m} =
        \begin{cases}
            \widetilde{\rho}_m & \text{ if } e \in \left\{ e_{1}, \dots, e_{k} \right\}\\
            0 & \text{ otherwise.}
        \end{cases} 
    \]
    Here $\widetilde{\rho}_m$ is as in \eqref{eq: rho_m}.
    By the metric doubling property of $X$ and Proposition \ref{prop: Graph approximations}, $k$ has an upper bound depending only on the doubling constant $N_D$ of $X$.
    We now define $\widehat{\rho}_{n+m} : V_{n + m} \to \R_{\geq 0}$
    \[
        \widehat{\rho}_{n+m}(v) = \max_{\substack{e \in E_n \\ v \in e \cdot G_m}} \lambda_{e,m}\left(\sigma_{e,m}^{-1}(v)\right).
    \]
    The density $\widehat{\rho}_{n+m}$ has the $p$-mass estimate
    \begin{align*}
    \mathcal{M}_p(\widehat{\rho}_{n+m}(v)) & \leq \sum_{x \in V_{n + m}} \sum_{\substack{e \in E_n \\ v \in e \cdot G_m}} \lambda_{e,m}(\sigma_{e,m}^{-1}(v))^p\\
        & \stackrel{\eqref{eq: mass of rho_m}}{\leq} k \cdot C \cdot \mathcal{M}_p^{m}\\
        & \lesssim \mathcal{M}_p^{m}
    \end{align*}
    where $C = C(C_{\deg},\abs{E_1},N_D,p)$.
    We will conclude the proof by showing the admissibility of $\widehat{\rho}_{n+m}$.
    
    Let $\gamma$ be a path from $\overline{B(x,r)}$ to $X \setminus B(x,2 \cdot r)$ and $\theta = G_{n+m}[\gamma]$.
    Then $X_{v_i} \cap \gamma \neq \emptyset$ for some $i$. By \eqref{eq: Cells contained},
    there is $v \in V_{n}\setminus \{ v_1,\dots,v_l, u_1,\dots,u_s \}$ so that $X_{v} \cap \gamma \neq \emptyset$. 
    By Proposition \ref{prop: Path between disjoint edges} $\theta$ contains a sub-path $\theta'$ from $\pi_{n + m,n}^{-1}(u_j)$ to $\pi_{n + m,n}^{-1}(w)$ contained in $e_i \cdot G_m$ for some $e_i = \{ u_j,w \} \in E_{n}$.
    Hence
    \begin{align*}
        \sum_{v \cap \gamma \neq \emptyset} \rho(v) \geq \sum_{v \in \theta'} \rho(v) = \sum_{u \in \sigma_{e_l,m}^{-1}(\theta')} \rho_m(u) \geq 1.
    \end{align*}
\end{proof}

\begin{remark}\label{remark: Annulus modulus}
    The reverse inequality of \eqref{eq: Annulus modulus} is also true and it can be proven by a similar flow argument as in Proposition \ref{prop: Ball-Loewner}. Indeed choose $v \in V_n$ so that $X_v \subseteq B(x,r)$. If $x = [(e_i)_{i = 1}^\infty]$ then, by \ref{lemma: properties of d_X (2)}, we may choose $v$ to be an endpoint of $e_n$.
    By \ref{lemma: properties of d_X (1)} there is $u \in V_n$ so that $X_u \cap B(x,2r) = \emptyset$ and a path $\theta$ from $v$ to $u$ of length $k$, where $k$ depends only on $C_{\diam}$ and $L_*$. By Proposition \ref{prop: Replacement flow} and similar argument as in Proposition \ref{prop: Ball-Loewner}, we have
    \[
        \Mod_p^{D}\left(\Gamma\left(\overline{B(x,r)},X \setminus B(x,2r)\right), G_{n+m}\right) \gtrsim \mathcal{M}_p^{m}.
    \]
\end{remark}

We now have proven enough moduli estimates to establish the combinatorial Loewner property.
\begin{proof}[Proof of Theorem \ref{thm:combloew}]
    By Assumption \ref{Assumptions: CLP} and Proposition \ref{prop: properties of cM_p},
     there is a unique $Q_* \in (1,\infty)$ so that $\cM_{Q_*} = 1$.
    By combining Propositions \ref{prop: LR modulus and BK-type modulus} and \ref{prop:confdimchar}, this $Q_*$ is equal to $\dim_{\AR}(X)$.
    
    To prove the combinatorial Loewner property, Condition \ref{CLP1} follows from Proposition \ref{prop: Ball-Loewner}
    and an iterative argument given in \cite[Proposition 2.9]{BourK}. 
    We only have left to prove \ref{CLP2}. By Proposition \ref{prop: Annulus modulus}
    \[
        \Mod_{Q_*}^{D}\left(\Gamma\left(\overline{B(x,r)}, X \setminus B(x,2r)\right), G_m\right) \lesssim 1
    \]
    whenever $r \geq 4 C_{\diam} L_*^{-m}$.
    Fix a constant $C > 1$ and consider the modulus problem in \ref{CLP2}. First assume $C > 4 C_{\diam}$ and let $N = \ceil{\log_2(C)} - 1$. Define $\rho_j$ to be the optimal $\Gamma\left(\overline{B(x,2^{j - 1}r)}, X \setminus B(x,2^j r)\right)$-admissible density. As $\diam(X_v) \leq r/2$, each $X_v$ intersects at most two annuli $\overline{B(x,2^jr)} \setminus B(x,2^{j - 1}r)$, which implies that the density
    \[
        \rho = \frac{2}{N} \sum_{j = 1}^N \rho_j
    \]
    is $\Gamma\left(\overline{B(x,r)}, X \setminus B(x,Cr)\right)$-admissible with the $Q_*$-mass estimate
    \[
        \cM_{Q_*}(\rho) \lesssim \frac{1}{N^{Q_* - 1}} \lesssim \log_2(C)^{1 - Q_*}.
    \]
    Hence, for all large $C > 1$, we may set $\psi(t) = \log_2(1/t)^{1 - Q_*}$ which satisfies $\psi(t) \to 0$ as $t \to 0$.
    
    Then we consider the case where $1 < C \leq 4 C_{\diam}$. We choose $k_C \in \N$ to be the smallest integer so that $C - 1 > 4 C_{\diam} \cdot L_*^{-k_C}$. By \cite[Proposition 2.2]{BourK},
    \begin{align*}
        & \, \Mod_{Q_*}^{D}\left(\Gamma\left(\overline{B(x,r)}, X \setminus B(x,Cr)\right), G_m\right)\\
        \leq & \, A \cdot \Mod_{Q_*}^{D}\left(\Gamma\left(\overline{B(x,r)}, X \setminus B(x,Cr)\right), G_{m + k_C}\right)
    \end{align*}
    where $A = A(N_D,L_*,k_C,Q_*)$. The rest of the argument is identical to the one in Proposition \ref{prop: Annulus modulus} with the slight difference that the number of edges $k$ in the proof depends also on $C$.
\end{proof}

We conclude this section by returning to Example \ref{ex:nonsym}.

\begin{example}\label{ex:loewnerreplacement}
    The family of spaces described in Example \ref{ex:nonsym} is interesting, since it furnishes a large family of limit spaces $X$ which are $Q$-Ahlfors regular for $Q>1$, and with $\dims_{\AR}(X)=Q$. Indeed, let $Q=\log(N \cdot L_*)/\log(L_*)$ which is the Hausdorff dimension of $X$ by Lemma \ref{lem:ahlforsregularity}. The admissible function $\rho(e)=L_*^{-1}$ is optimal for all $p>1$, as described in Example \ref{ex:nonsym}. Therefore, one sees that $\mathcal{M}_p=L_*^{1-p}N$, and $\mathcal{M}_Q=1$. Thus, by Theorem $\ref{thm:combloew}$, we see that $X$ satisfies the combinatorial $Q$-Lowner property. Since $X$ is also $Q$-Ahlfors regular, then $X$ must be $Q$-Lowener, see the argument in \cite[Introduction]{CEB} or \cite[Proposition 11.13]{murugan2023first}. These give a large family of Loewner examples. The fact that these examples are Loewner could also be seen by applying the argument from \cite{CK} and recognizing the spaces as inverse limits. This analysis also justifies why we have stated that Examples in  \ref{ex:laaksospace} and \ref{fig:laakso} are Loewner.

\end{example}

\section{Proof of Main theorem}\label{sec:mainthm}

The proof of Theorem \ref{thm:mainthm} involves porosity, which we briefly study next. 

\subsection{Porous sets and Assouad dimension}\label{subsec:porosity}
We will prove Proposition \ref{prop:porous} in this section and give a natural condition for porosity of IGSs. First, we introduce some additional terminology that is just needed in this section. 
For a subset $A\subset X$ le
\[
N(A,r):=\inf\left\{N : A \subset \bigcup_{i=1}^N B(x_i,r)\right\}
\]
be the smallest number of balls of radius $r$ needed to cover the set $A$. The Assouad dimension of a metric space is then given by
\[
\dims_{\rm A}(X):=\inf\left\{s>0 : \exists C \geq 1, \forall 0<r<R, \forall x\in X, N(B(x,R),r)\leq C R^sr^{-s}\right\}.
\]
A standard volume counting argument shows that if $X$ is $Q$-Ahlfors regular, then $\dims_{\rm A}(X)=Q$. 
We need the following result which is from \cite[Proposition 14.14]{He}. In that reference, connectivity is replaced by the weaker condition of uniform perfectness. For us, this weaker statement suffices.
\begin{lemma}\label{lem:assouadahlfors}
    If $X$ is connected, then for every $Q>\dims_A(X)$ there exists a metric $d'\in \cG_{\rm AR}(X)$ which is $Q$-Ahlfors regular. In particular, $\dims_{\rm AR}(X)\leq \dims_{\rm A}(X)$. 
\end{lemma}

We now prove Proposition \ref{prop:porous}.

\begin{proof}[Proof of Proposition \ref{prop:porous}]
Let $Y\subset (X,d)$ be a porous subset with $\dims_{\rm AR}(X)=\dims_{AR}(Y)$. Let $d'\in \cG_{\rm AR}(X)$ be $Q$-Ahlfors regular.  Then, $\dims_{\rm A}((X,d'))=Q$. Since $(X,d')$ is quasisymmetric to $(X,d)$, via the identity map, it is direct to see that $(Y,d'|_Y)$ is a porous subset of $(X,d')$, and that $(Y,d|_Y)$ is quasisymmetric to $(Y,d'|_Y)$ and $d'|_Y\in \cG(Y)$. 

Porosity implies that the Hausdorff $Q$-measures of neighborhoods of $Y$ decay geometrically in their thickness. This, together with Ahlfors regularity and a volume counting argument yields an Assouad dimension bound. This fairly direct counting argument is classical, and yields that Assouad dimension of a porous subset of a $Q$-regular space is strictly less than the Assouad dimension of the space see e.g. \cite[Lemma 3.12]{BHR}, (result also contained in many other places, e.g. \cite[Proposition 3.5]{KLV}):
\[
\dims_{\rm A}(Y,d'|_Y)< Q.
\]
Thus, by Lemma \ref{lem:assouadahlfors}, we get
\[
\dims_{\rm AR}(Y)=\dims_{\rm AR}(Y,d'|_Y)\leq \dims_{\rm A}(Y,d')<Q.
\]
Thus, $Q>\dims_{\rm AR}(X)$, and $X$ does not attain its conformal dimension.
\end{proof}

\subsection{Proofs of main theorems}

Given an IGS, every connected sub-graph $\hat{G}_1 = (\hat{V}_1,\hat{E}_1)$ of $G_1$, so that $\hat{V}_1$ contains the gluing sets, admits a natural IGS by setting $\hat{\phi}_{v,e}= \phi_{v,e}$ for all $e \in \hat{E}_1$ and $v \in e$.  Here $\hat{\phi}_{v,e}$ are the maps associated to the gluing rules of $\hat{G}_1$. This is possible as long as $\hat{V}_1$ contains the gluing sets $\phi_{v,e}$ for every $e\in \hat{E}_1$ and $v\in e$. In this section we are interested in the sub-graph and the associated IGS obtained by eliminating a removable edge. Here the IGSs are referred to simply by their graphs at the first levels: $G_1$ and $\hat{G_1}$. Moreover, everything related to the IGS $\hat{G}_1$ is written with a '$\,\hat{}\,$' symbol, e.g ($\hat{C}_{\diam}, \hat{V}_m, \hat{G}_m, \hat{X},$ etc).

\begin{definition}\label{def:removable}
    Let $e^* \in E_1$ and denote the graph $\hat{G}_1 = (\hat{V},E \setminus \{e^*\})$, where 
    \[
      \hat{V} =
      \begin{cases}
          V \setminus \{ v^* \} & \text{ if } v^* \in e \text{ so that } \deg(v^*) = 1\\
          V & \text{ otherwise}.
      \end{cases}  
    \]
    We say that $e^* \in E_1$ is a \emph{removable edge} if the following hold.
    \begin{enumerate}
    \item The edge $e^*$ does not contain a vertex in $I_{v,e}$ for any edge $e \in E_1$ and $v \in e$.
    \item The graph $\hat{G}_1$ is connected and the associated IGS satisfies the $L_*$-uniform scaling property for the same $L_* \geq 2$ as $G_1$.
    \item \label{def: removable (3)} For all $e = \{v,u\} \in E_1 \setminus \{ e^* \}$ and $p > 1$ we have $\rho_{e}(e^*)=0$. Here $\rho_{e} : E_1 \to \R_{\geq 0}$ is as defined in Section \ref{subsection: Moduli computations}.
    \end{enumerate}
\end{definition}

\begin{remark}\label{remark: Removable edge by potential and flow}
    The last condition in Definition \ref{def:removable} can be also expressed in terms of the optimal potential function $U_{v,e}$ and the optimal unit flow $\cF_{v,e}$. Indeed, by Lemma \ref{lemma: Mod-Cap}, $\rho_{e}(e^*) = \abs{U_{v,e}(v^*) - U_{v,e}(u^*)}=0$ and $\cF_{v,e}(e^*)=0$ for any removable edge $e^*$. 
\end{remark}

\begin{lemma}\label{lem:porous}
    Suppose the iterated graph system $G_1$ satisfies the $L_*$-uniform scaling- and the doubling property. If $G_1$ contains a removable edge $e^* = \{ v^*,u^* \}$ and $\hat{G}_1$ is as in Definition \ref{def:removable}, there is a is a biLipschitz embedding $\iota: \hat{X}\to X$ so that $\iota(\hat{X})$ a porous subset of $X$.
\end{lemma}

\begin{proof}
    Note that $\hat{G}_m$ can naturally be regarded as a sub-graph of $G_m$. Hence we obtain the natural embedding $\iota : \hat{X} \to X, \iota([(e_i)_{i=1}^\infty])=[(e_i)_{i=1}^\infty]$.
    Choose any distinct points $x = [(e_i)_{i = 1}^\infty], y = [(f_i)_{i = 1}^\infty] \in \hat{X}$ and let $m \in \N$ be the smallest index so that $e_m$ and $f_m$ do not have a common vertex. Then
    \[
       L_*^{-m} \stackrel{\ref{lemma: properties of d_X (1)}}{\leq} d_{X}(\iota(x),\iota(y)),d_{\hat{X}}(x,y) \stackrel{\ref{lemma: properties of d_X (2)}}{\leq} 2(C_{\diam} \lor\hat{C}_{\diam} )L_*^{1-m}.
    \]
    This proves that $\iota$ is a biLipschitz embedding.

    Lastly we prove that $\iota(\hat{X}) \subseteq X$ is a porous subset. Let $x = [(e_i)_{i = 1}^{\infty}] \in \iota(\hat{X})$ and $r > 0$ so that
    \[
        3(C_{\diam} \lor\hat{C}_{\diam} )\cdot L_*^{-(m+1)} < r \leq 3(C_{\diam} \lor \hat{C}_{\diam} )\cdot L_*^{-m}
    \]
    for some $m \in \N$. Choose the point $y = [(f_i)_{i = 1}^\infty] \in X$ so that $f_i = e_i$ for all $i \leq m + 1$ and $f_{i + 1} = \{ [v^*,f_i], [u^*,f_i] \}$ for all $i > m + 1$. Then
    \[
        d_{X}(x,y) \stackrel{\ref{lemma: properties of d_X (2)}}{\leq} 2(C_{\diam} \lor\hat{C}_{\diam} )\cdot L_*^{-(m+1)}
    \]
    On the other hand, $v^*,u^* \notin I_{v,e}$ for any $e \in E_1$ and $v \in e$, so the edge $f_{m + 2}$ does not share a vertex with any edge contained in $\hat{G}_{m+2}$. Therefore
    \[
        d_{X}(y,z) \stackrel{\ref{lemma: properties of d_X (1)}}{\geq} L_*^{-(m+2)}
    \]
    for all $z \in \iota(\hat{X})$. By choosing $c = (4(C_{\diam} \lor \hat{C}_{\diam})L_*^2)^{-1}$ we have that $B(y,cr) \subseteq B(x,r) \setminus \iota(\hat{X})$.
\end{proof}

Next, we state the general form of our main theorem.

\begin{theorem}\label{thm:removable-edge}
    If the iterated graph system $G_1$ satisfies Assumption \ref{Assumptions: CLP} and contains a removable edge $e^* \in E_1$ then the limit space $X$ is combinatorially Loewner, approximately self-similar and does not attain its conformal dimension. In particular, $X$ is not quasisymmetric to a Loewner space.
\end{theorem}

\begin{proof}
By Theorem \ref{thm:combloew} and Proposition \ref{prop:noloew} we only need to verify that $X$ does not attain its conformal dimension.

It is clear that for all $p > 1$ and $e = \{ v,u \} \in \hat{E}_1$ the density is $(\rho_{e})|_{\hat{E}} : \hat{E}_1 \to \R_{\geq 0}$ is $\hat{\Theta}_{v,e}$-admissible. On the other hand, it follows from the definition of the removable edge, that $(\rho_{e})|_{\hat{E}}$ is the optimal $\hat{\Theta}_{v,e}$-admissible density and that $\cM_p = \hat{\cM}_p$ for all $p > 1$. In particular, by Theorem \ref{thm:combloew}, $\dim_{\AR}(X) = \dim_{\AR}(\hat{X})$. The claim now follows from Lemma \ref{lem:porous} and Proposition \ref{prop:porous}.
\end{proof}

We now briefly discuss a method which yields IGSs satisfying assumptions in Theorem \ref{thm:removable-edge}. This method relies on the existence of a certain type of symmetry, which we will use to verify the uniformly conductive property and the existence of a removable edge.
\begin{assumption}\label{Assumption: Symmetric replacement rule}
    The iterated graph system satisfies the following:
    \begin{enumerate}
    \item There are functions $\phi_+,\phi_- : I \to V_1$ so that $\{\phi_{v,e},\phi_{u,e}\} = \{ \phi_+,\phi_- \}$ for all $e = \{ v,u \} \in E_1$.
    \item There is a graph isomorphism $\eta : G_1 \to G_1$ so that $\phi_{\pm} = \eta \circ \phi_{\mp}$
    \item Asummption \ref{Assumptions: CLP} (1),(2) and (4*)
    \end{enumerate}
\end{assumption}

Whenever we assume Assumption \ref{Assumption: Symmetric replacement rule} to hold, we shall write $I_{\pm} := \phi_{\pm}(I)$, $\cF_{\pm}$ is the optimal unit flow from $I_{\pm}$ to $I_{\mp}$.
\begin{lemma}\label{lem:removableedges}
    An iterated graph system satisfying Assumption \ref{Assumption: Symmetric replacement rule} is conductively uniform. Moreover, if $e^* = \{ v^*,u^* \} \in E_1$ so that $\eta$ fixes both $v^*$ and $u^*$, then $e^*$ is a removable edge.
\end{lemma}

\begin{proof}
    Since $\degr(z) = 1$ for all $z \in I_+ \cup I_-$ by the doubling property, we must have $\eta(\mathfrak{n}_z) = \mathfrak{n}_{\eta(z)}$. Since $\eta(I_{\pm}) = I_{\mp}$, by the uniqueness of the optimal unit flow, for all egdes $\{ v,u \} \in E_1$ we have
    \begin{equation}\label{eq: Flow and symmetry}
        \cF_-(v,u) = \cF_+(\eta(v),\eta(u)).
    \end{equation}
    In particular, for any $a \in I$, we have
    \[
        \cF_-(\phi_-(a), \mathfrak{n}_{\phi_-(a)}) \stackrel{\eqref{eq: Flow and symmetry}}{=} \cF_+(\eta(\phi_-(a)), \eta(\mathfrak{n}_{\phi_-(a)})) = \cF_+(\phi_+(a),\mathfrak{n}_{\phi_+(a)}).
    \]
    Hence \eqref{eq: Flows agree on boundary} holds.

    To prove that $e^*$ is a removable edge, first note that $\eta(I_{\pm}) = I_{\mp}$ and $I_- \cap I_+ = \emptyset$. In particular, $\eta$ does not fix any vertex in $I_{\pm}$ and therefore $e^*$ does not contain a vertex in $I_{\pm}$. This proves the first condition in Definition \ref{def:removable}.
    In order to prove the second condition,
    we show that the edge if $\theta \in \Theta(I_-,I_+)$ contains $e^*$ then there is a strictly shorter path $\theta' \in \Theta(I_-,I_+)$. This is sufficient, as by the $L_*$-uniform scaling property, every shortest path from $I_-$ to $I_+$ is of length $L_*$. Let $\theta = [v_1,\dots, v_k,v^*,u^*,u_1 \dots,u_{l}]$ be a path from $v_1 \in I_-$ to $u_m \in I_+$.
    By applying symmetry $\eta$ to the path $\theta$, we may assume that $k \leq l$. Then we let $\hat{\theta} = [v_1,\dots,v_k,v^*,\eta(v_k),\dots,\eta(v_1)]$. Indeed, since $\eta$ fixes $v^*$, $\hat{\theta}$ is a path from $v_1 \in I_-$ to $\eta(v_1) \in I_+$ with length $\len(\hat{\theta}) = 2k < k + l + 1 = \len(\theta)$.

    We have dealt with first two conditions in Definition \ref{def:removable}. In order to prove the last one, we use the optimal flows $\cF_{\pm}$ and Remark \ref{remark: Removable edge by potential and flow}. Indeed we will show that $\cF_{-}(v^*,u^*) = 0$. This follows by
    \[
        \cF_{-}(v^*,u^*) \stackrel{\eqref{eq: Reversing flow}}{=} -\cF_{+}(v^*,u^*) \stackrel{\eqref{eq: Flow and symmetry}}{=} -\cF_{-}(\eta(v^*),\eta(u^*)) = -\cF_{-}(v^*,u^*).
    \]
    Hence $e^*$ is a removable edge.
\end{proof}

The previous Lemma only gives symmetric examples of Theorem \ref{thm:removable-edge}. To obtain non-symmetric ones, one can use the procedure from Figure \ref{fig:condutctunif-notsym} and Example \ref{ex:nonsym}. Indeed, in this example the copies of a line segment are glued along corresponding subdivision points, and this defines a potential function $U(v_i)=x(i)$, where $x(i)=k/L$ is the corresponding subdivision point of the line. As explained in Example \ref{ex:nonsym}, and as follows from Lemma \ref{lemma: Mod-Cap}, $U$ is an optimal potential function for all $p\in (1,\infty)$. Now, any edge $\{v,u\}$ can be added to produce a removable edge, as long as $U(v)=U(u)$. Figure \ref{fig:condutctunif-notsym} shows one such edge as a dashed edge. It yields a non-symmetric counterexample to Kleiner's conjecture.

Using the procedure from Lemma \ref{lem:removableedges}, several examples of the failure of Kleiner's conjecture can be constructed. In the following proof, we just focus on one such example.

\begin{proof}[Proof of Proposition \ref{prop:counterxample} and Theorem \ref{thm:mainthm}.] 
Consider the IGS $G_1$ as in
Example \ref{ex:counterexample}. We suggest that the reader also reviews Figure \ref{fig:replacementrule}. It is easily verified that this IGS satisfies Assumption \ref{Assumption: Symmetric replacement rule} by defining the mapping $\eta : V_1 \to V_1$ given by
\[
    \eta(v) =
    \begin{cases}
        v & \text{ if } v \in \{4,5\}\\
        7 & \text{ if } v = 1\\
        8 & \text{ if } v = 2\\
        6 & \text{ if } v = 3
    \end{cases}
\]
and $\eta^2 = \id_{V_1}$. It is also clear that $G_1$ contain two edge-wise disjoint paths from $I_+$ to $I_-$ so, by 
Lemma \ref{lem:removableedges}, $G_1$ satisfies Assumption \ref{Assumptions: CLP} and $\{4,5\} \in E_1$ is a removable edge. Our main result now follows from Theorem \ref{thm:removable-edge}. The $\log(9)/\log(4)$-Ahlfors regularity of the limit space $X$ follows from Lemma \ref{lem:ahlforsregularity}. Lastly, the combinatorial $3/2$-Loewner property of $X$ follows from Theorem \ref{thm:combloew} and the simple computation $\cM_Q = 8 \cdot 4^{-Q}$, which is equal to $1$ exactly when $Q = 3/2$.
\end{proof}

Indeed, the same proof scheme yields other examples. From Lemma \ref{lem:removableedges} and Theorem \ref{thm:removable-edge}, it follows that we can add an edge to the Laakso space IGS in Example \ref{ex:laaksospace}, and obtain the space in Figure \ref{fig:smallcounter}. This example seems to be the smallest and simplest counterexample that one can construct. Further, we saw earlier, that many (non-symmetric) counterexamples can be constructed with the procedure in Figure \ref{fig:condutctunif-notsym}.

\begin{figure}[!ht]
\begin{tikzpicture}[scale=1]
\draw (-6,0)--(-4,0);
\draw[<-,thick] (-3,0)--(-2,0);

\draw (-1,1)--(0,0)--(1,1);
\draw (-1,-1)--(0,0)--(1,-1);
\draw (0,0)--(0,-1);
\draw[<-,thick] (2,0)--(3,0);

\draw (4,1.25)--(4.5,0.5)--(5.25,0);
\draw (4,0.75)--(4.5,0.5)--(4.75,0);
\draw (4.5,0.5)--(4.1,0.25);
\draw (4,-1.25)--(4.5,-0.5)--(4.75,0);
\draw (4,-0.75)--(4.5,-0.5)--(5.25,0);
\draw (4.5,-0.5)--(4.1,-0.25);

\draw (5.25,0)--(5.5,0.5)--(6,1.25);
\draw (4.75,0)--(5.5,0.5)--(6,0.75);
\draw (5.5,0.5)--(5.9,0.25);
\draw (5.25,0)--(5.5,-0.5)--(6,-1.25);
\draw (4.75,0)--(5.5,-0.5)--(6,-0.75);
\draw (5.5,-0.5)--(5.9,-0.25);

\draw (5.25,0)--(5,-0.6)--(4.75,-1.25);
\draw (4.75,0)--(5,-0.6)--(5.25,-1.25);
\draw (5,-0.6)--(4.75,-0.6);

\foreach \p in {(-4,0),(-6,0),(-1,1),(0,0),(1,1),(1,-1),(-1,-1),(0,-1),
(4,1.25),(4.5,0.5),(4,0.75),(4.5,-0.5),(4,-0.75),(4,-1.25),(5.25,0),(4.75,0),(5.5,0.5),(5.5,-0.5), (6,1.25) , (6,0.75),(6,-0.75), (6,-1.25),
(4.75,-0.6),(5,-0.6),(4.75,-1.25),(5.25,-1.25),(5.9,-0.25),(5.9,0.25),(4.1,-0.25),(4.1,0.25)}
\node at \p [circle,fill,inner sep=1pt]{};

\node at (-1,1.4) {1};
\node at (-1,-0.6) {2};
\node at (0,0.4) {3};
\node at (0.3,-1) {4};
\node at (1,1.4) {5};
\node at (1,-0.6) {6};
\end{tikzpicture}
\caption{Figure of the smallest IGS that produces a counterexample to Conjecture \ref{conj:kleiner}. It is obtained from Example \ref{ex:laaksospace} by adding an edge to the vertex $v=3$ as in Lemma \ref{lem:removableedges}.}
\label{fig:smallcounter}
\end{figure}

\subsection{On the \texorpdfstring{$p$}{p}-walk dimension}\label{subsec: p-walkdim}
In probability theory, the walk dimension determines the order of the typical distance a diffusion travels in time $t > 0$. For its natural generalization, the $p$-walk dimension $d_{w,p}$, it is usually fairly easy to show that $d_{w,p} \geq p$ (see e.g. \cite[Proposition 3.5]{shimizu}).
For many fractals it has been observed that the 2-walk dimension is strictly greater than 2 (see e.g.\cite{kajinowalkdim,barlowperkinsgakset,barlowbass}). For general $p > 1$, it seems to be a far more difficult task to verify whether the strict inequality is true or not. However, the strict inequality is expected to be true in many cases. In a recent work of Kajino and Shimizu \cite{kajino2024contraction}, the strict inequality $d_{w,p} > p$ for all $p > 1$ was established for generalized Sierpi\'nski carpets and some gaskets. This problem is also observed to be connected to the Sobolev spaces in the works \cite{murugan2023first,kajino2024korevaarschoen,GrigoryanBesov2}. See also \cite[Problem 4 in Section 6.3]{kigami}.

In our framework, whenever (1)-(3) in Assumption \ref{Assumptions: CLP} is satisfied, we are able to compute the exact value of $p$-walk dimension. This is essentially due to Corollary \ref{corollary: Modulus from L to R lvl m} and Theorem \ref{thm: Moduli asymptotic}. We also present in Proposition \ref{Prop: Characterize p-walkdim} an easy characterization of when the inequality $d_{w,p} \geq p$ is strict.

\begin{definition}\label{def:walkdim}
If the iterated graph system satisfies (1)-(3) in Assumption \ref{Assumptions: CLP}, its \emph{$p$-walk dimension} is 
\[
    d_{w,p} := \frac{\log(\abs{E_1} \cdot (\cM_p)^{-1})}{\log(L_*)}.
\]
\end{definition}
We remark that in \cite{kajino2024contraction, kajino2024korevaarschoen} the neighbor disparity constant, which we do not consider, is also involved. Indeed we define $p$-walk dimension through the modulus constant $\cM_p$. Moreover, in the light of Proposition \ref{prop: Annulus modulus} and Remark \ref{remark: Annulus modulus}, it can be verified that the capacity/conductance constant in \cite{kajino2024contraction, kajino2024korevaarschoen} coincides with $\cM_p$ so our definition of $p$-walk dimension is the same as the ones in previously mentioned works.

\begin{proposition}\label{Prop: Characterize p-walkdim}
    Let $p \in (1,\infty)$. If the iterated graph system satisfies the (1)-(3) in Assumption \ref{Assumptions: CLP} then $d_{w,p} \geq p$. Furthermore, $d_{w,p} = p$ if and only if $\rho \equiv L_*^{-1}$ is the optimal $\Theta_{v,e}$-admissible density for all $e \in E_1$ and $v \in e$.
\end{proposition}

\begin{proof}
   By the $L_*$-unifrom scaling property, $\rho : E_1 \to \R_{\geq 0}$, $\rho \equiv L_*^{-1}$ is $\Theta_{v,e}$-admissible density for all $e \in E_1$ and $v \in e$, which yields
   \[
        d_{w,p} \geq \frac{\log(\abs{E_1} \cdot (\cM_p(\rho))^{-1})}{\log(L_*)} = p.
   \]
   The last assertion follows from the uniqueness of the optimal density.
\end{proof}

The following corollary is immediate.

\begin{corollary}\label{Cor: Walk dim > p}
    If the iterated graph system contains a removable edge then for all $p \in (1,\infty)$ we have $d_{w,p} > p$.
\end{corollary}

The equality $d_{w,p} = p$ certainly is possible. In fact, a direct computation reveals the equality for the IGS in Example \ref{ex:counterexample} that does not contain the middle edge.
More generally, the equality holds in the case of Example \ref{ex:nonsym}.
On the other hand, the converse of Corollary \ref{Cor: Walk dim > p} is not true, i.e., the strict inequality $d_{w,p} > p$ may hold even if the IGS does not contain removable edges. 
See Example \ref{ex:probably loewner}.

\begin{example}\label{ex:probably loewner}
    This example is shown in Figure \ref{fig:probably loewner}.
    Let $G_1 = (V,E)$ as in Example \ref{ex:counterexample}. We construct a new graph $\hat{G}_1 := (\hat{V}_1,\hat{E}_1)$ where $\hat{V}_1 := V_1 \cup \{9,10\}$ and $\hat{E}_1 := E_1 \cup \{ \{ 9,3 \}, \{ 6,10 \} \}$. We extend the gluing rules by $\hat{I} := I \cup \{c\}$, $(\hat{\phi}_{\pm})|_I := \phi_\pm$, $\hat{\phi}_-(c)= 9$ and $\hat{\phi}_+(c) = 10$.
    We set $\hat{\phi}_{9,\{9,3\}} = \hat{\phi}_-$, $\hat{\phi}_{3,\{ 9,3 \}} = \hat{\phi}_+$, $\hat{\phi}_{6,\{6,10\}} = \hat{\phi}_-$ and $\hat{\phi}_{10,\{6,10\}} = \hat{\phi}_+$. It is a direct computation to show that $\hat{G}_1$ satisfies Assumption \ref{Assumptions: CLP} and the following properties.
    \begin{enumerate}
        \item $\cM_Q < 1$ where $Q = \log(\abs{\hat{E}_1})/\log(L_*) = \log(10)/\log(4)$ is the Hausdorff dimension of the associated limit space.
        \item The combinatorial $Q_*$-Loewner property for some $Q_* \in (1,Q)$.
        \item $d_{w,p} > p$ for all $p \in (1,\infty)$.
    \end{enumerate}
    Indeed these can be verified by computing the optimal unit flow from $\hat{I}_+$ to $\hat{I}_-$, as it does not depend on $p$.
\end{example}

\begin{figure}
\begin{tikzpicture}[scale=0.75]
\draw (-6,0)--(-4,0);
\draw[<-,thick] (-3,0)--(-2,0);

\draw (-1,-1)--(0,0)--(1,1)--(2,0)--(3,1);
\draw (0,0)--(1,-1)--(2,0);
\draw (-1,0)--(0,0);
\draw (-1,1)--(0,0);
\draw (2,0)--(3,-1);
\draw (2,0)--(3,0);

\foreach \q in {(-6,0),(-4,0),(-1,1),(-1,-1),(0,0),(1,1),(1,-1),(2,0),(3,1),(3,-1),(-1,0),(3,0)}
\node at \q [circle,fill,inner sep=1pt]{};

\node at (-1,1.4) {1};
\node at (-1,-1.5) {2};
\node at (0,0.4) {3};
\node at (1,1.4) {4};
\node at (1,-1.5) {5};
\node at (2,0.4) {6};
\node at (3,1.4) {7};
\node at (3,-1.5) {8};
\node at (-1.3,0) {9};
\node at (3.3,0) {10};
\end{tikzpicture}
\caption{Figure of the IGS in Example \ref{ex:probably loewner}.}
\label{fig:probably loewner}
\end{figure}

\bibliographystyle{acm}
\bibliography{clp}

\vspace{1cm}

\end{document}